\documentclass[reqno,11pt]{amsart}

\usepackage[top=2.0cm,bottom=2.0cm,left=3cm,right=3cm]{geometry}
\usepackage{amsthm,amsmath,amssymb,dsfont}
\usepackage{mathrsfs,amsfonts,functan,extarrows,mathtools}
\usepackage[colorlinks]{hyperref}
\usepackage{marginnote}
\usepackage{xcolor}
\usepackage{stmaryrd}
\usepackage{esint}
\usepackage{graphicx}
\usepackage{bbm}

\usepackage{tikz}
\usetikzlibrary{positioning}

\newtheorem{theorem}{Theorem}[section]
\newtheorem{proposition}{Proposition}[section]

\newtheorem{corollary}{Corollary}[section]
\newtheorem{remark}{Remark}[section]
\newtheorem{lemma}{Lemma}[section]

\numberwithin{equation}{section}
\allowdisplaybreaks

\def\p{\partial}
\def\d{\mathrm{d}}
\def\no{\nonumber}
\def\R{\mathbb{R}}
\def\eps{\varepsilon}
\def\div{\mathrm{div}}
\def\u{\mathbf{u}}

\def\l{\langle}
\def\r{\rangle}
\def\exp{\mathrm{exp}}

\def\A{\mathcal{A}}
\def\B{\mathcal{B}}
\def\C{\mathcal{C}}
\def\u{\mathfrak{u}}
\def\L{\mathcal{L}}

\def\P{\mathcal{P}}
\def\I{\mathcal{I}}
\def\J{\mathcal{J}}
\def\N{\mathcal{N}}
\def\n{\mathsf{n}}
\def\m{\mathsf{m}}
\def\j{\mathsf{j}}

\newcounter{wronumber}

\begin{document}
\title[Incompressible Euler limit from the Boltzmann equation]
			{Incompressible Euler limit from the Boltzmann equation with Maxwell reflection boundary condition in the half-space}

\author[Ning Jiang]{Ning Jiang}
\address[Ning Jiang]{\newline School of Mathematics and Statistics, Wuhan University, Wuhan, 430072, P. R. China}
\email{njiang@whu.edu.cn}
 
\author[Chao Wang]{Chao Wang}
\address[Chao Wang]{\newline School of Mathematics Sciences, Peking University, Beijing, 100871, P. R. China}
\email{wangchao@math.pku.edu.cn}

\author[Yulong Wu]{Yulong Wu}
\address[Yulong Wu]{\newline School of Mathematics and Statistics, Wuhan University, Wuhan, 430072, P. R. China}
\email{yulong\string_wu@whu.edu.cn}

\author[Zhifei Zhang]{Zhifei Zhang}
\address[Zhifei Zhang]{\newline  School of Mathematics Sciences, Peking University, Beijing, 100871, P. R. China}
\email{zfzhang@math.pku.edu.cn}
\maketitle

\begin{abstract}
  In this paper, we rigorously justify the incompressible Euler limit of the Boltzmann equation with general Maxwell reflection boundary condition in the half-space. The accommodation coefficient $\alpha \in (0,1]$ is assumed to be $O(1)$.  Our construction of solutions includes the interior fluid part and Knudsen-Prandtl coupled boundary layers. The corresponding solutions to the nonlinear Euler and nonlinear Prandtl systems are taken to be shear flows. Due to the presence of the nonlinear Prandtl layer, the remainder equation loses  one order normal derivative. The key technical novelty lies in employing the full conservation laws to convert this loss of the normal derivative into the loss of tangential spatial derivative, avoiding any loss of regularity in time. By working within an analytic $L^2 \mbox{-} L^\infty$ framework, we establish the uniform estimate on the remainder equations, thus justify the validity of the incompressible Euler limit from the Boltzmann equation for the shear flow case. \\


\end{abstract}





\section{Introduction}

\subsection{Boltzmann equation with Maxwell boundary condition}
The Boltzmann equation in the incompressible Euler scaling is given by:
\begin{equation}\label{BE}
\left\{
\begin{array}{l}
\eps \p_t F_\eps + v \cdot \nabla_x F_\eps = \frac{1}{\eps^q} B (F_\eps, F_\eps) \quad \qquad {\text{in}}\  \R_+ \times \Omega \times \R^3 \,,\\[1.5mm]
F_\eps (0, x, v) = F_\eps^{in} (x,v) \geq 0 \qquad \ \, \qquad {\text{in}}\  \Omega \times \R^3 \,.
\end{array}
\right.
\end{equation}
Here, the dimensionless number $\eps>0$ is the Storuhal number, and $\eps^q$ with $q>1$ is the Knudsen number. The function $F_\eps (t,x,v) \geq 0$ represents the number density of particles, where $v \in \R^3$ is the velocity and $x \in \Omega = \{ x \in \R^3; x_3 > 0 \}$ denotes the position in the half-space. For the setting up of the scaled equation \eqref{BE} and its relation to the incompressible Euler equations, see \cite{BGL1} and \cite{SaintRaymond-2003-ARMA}.

The scaled Boltzmann equation \eqref{BE} is imposed with the Maxwell reflection boundary condition
\begin{equation}\label{MBC}
\gamma_- F_\eps = (1-\alpha) L \gamma_+ F_\eps + \alpha K \gamma_+ F_\eps \quad \textrm{on } \R_+ \times \Sigma_- \,.
\end{equation}
For the derivation from the dimensional Boltzmann equation to the scaled form \eqref{BE}, the readers can find the details in the books \cite{Sone-2002book,Sone-2007-Book}, or papers \cite{BGL1, BGL2}. This paper aims to rigorously justify the incompressible Euler limit of the Boltzmann equation. 

 The Boltzmann collision operator is defined as
\begin{equation}\label{B-collision}
  \begin{aligned}
    B (F_1, F_2) (v) =& \int_{\R^3} \int_{\mathbb{S}^2} |v-v_1|^{\gamma} F_1 (v_1^\prime ) F_2 (v^\prime ) b(\theta) \d \omega \d v_1 \\
   & - \int_{\R^3} \int_{\mathbb{S}^2} |v-v_1|^{\gamma} F_1 (v_1) F_2 (v) b(\theta) \d \omega \d v_1 \,,
  \end{aligned}
\end{equation}
where $v_1^\prime  = v_1 + [(v-v_1) \cdot \omega] \omega$, $v^\prime  = v_1 + [(v-v_1) \cdot \omega] \omega$, $\cos \theta = (v-v_1) \cdot \omega / |v - v_1|$, $0 < b (\theta) \leq C |\cos \theta|$, and $-3 < \gamma \leq 1$. Such collision operators cover the full range of cutoff collision kernels. The accommodation coefficient $\alpha \in [0,1]$ describes how much the molecules accommodate the state of the wall. The special case $\alpha=0$ corresponds to specular reflection, while $\alpha=1$ refers to complete diffusion. In this paper, we consider the case where $\alpha \in(0,1]$ and $\alpha =O(1)$. We remark that the specular and almost specular cases are much easier, as the Prandtl layer is weak and linear (see \cite{Huang-Wang-Wang-Xiao-2025,Guo-Huang-Wang-2021-ARMA,Jiang-Luo-Tang-2021-arXiv}).

We denote $n = (0, 0, -1)$ by the outward normal of $\Omega$. Let $\Sigma : = \p \Omega \times \R^3$ be the phase space boundary of $\Omega \times \R^3$. The phase boundary $\Sigma$ can be split by outgoing boundary $\Sigma_+$, incoming boundary $\Sigma_-$, and grazing boundary $\Sigma_0$:
\begin{equation*}
  \begin{aligned}
    & \Sigma_+ = \{ (x, v) : x \in \p \Omega \,, v \cdot n = - v_3 > 0 \} \,, \\
    & \Sigma_- = \{ (x, v) : x \in \p \Omega \,, v \cdot n = - v_3 < 0 \} \,, \\
    & \Sigma_0 = \{ (x, v) : x \in \p \Omega \,, v \cdot n = - v_3 = 0 \} \,.
  \end{aligned}
\end{equation*}
Let $\gamma_\pm F = \mathbbm{1}_{\Sigma_{\pm}} F$. The specular-reflection $L \gamma_+ F_\eps$ and the diffuse-reflection part $K \gamma_+ F_\eps$ in \eqref{MBC} are
\begin{equation}\label{def_Lgamma_Kgamma}
  \begin{aligned}
  L \gamma_+ F_\eps (t, x, v) = & F_\eps (t, x, v - 2[ (v - u_w) \cdot n]n) \,, \\
  K \gamma_+ F_\eps (t, x, v) = & \sqrt{\tfrac{2 \pi} {\theta_w( t, \bar{x})}} M_w (t, \bar{x}, v) \int_{v \cdot n >0} \gamma_+ F_\eps ( v \cdot n ) \d v \,,
  \end{aligned}
\end{equation}
respectively, where $M_w (t, \bar{x}, v)$ is the local Maxwellian distribution function corresponding to the wall (boundary) with the form
\begin{equation}\label{M_w}
  \begin{aligned}
    M_w (t, \bar{x}, v) = \frac{1}{(2 \pi \theta_w (t, \bar{x}))^{3/2}} \exp\Big\{ - \frac{|v- u_w (t, \bar{x})|^2}{2 \theta_w (t, \bar{x})} \Big\} \,.
  \end{aligned}
\end{equation}
The quantities $u_w, \theta_w$ are, respectively, the velocity and temperature of the boundary. It is worth mentioning that the boundary condition \eqref{MBC} with \eqref{def_Lgamma_Kgamma} and \eqref{M_w} implies that there is no instantaneous flow across the boundary:
\begin{equation*}
  \int_{\R^3} (v- u_w) \cdot n F_\eps \d v = 0\,.
\end{equation*}
For simplicity, we assume $u_w=0, \theta_w =1$ throughout this paper. The moving and non-isothermal boundary (i.e. $u_w\neq 0$ and $\theta_w$ is non-constant) case will be involved more fruitful phenomena, and will be considered in a separate paper. 

\subsection{Historical remarks.} The hydrodynamic limit of the Boltzmann equation is a classical problem. Among numerous references, we mention several standard books that provide comprehensive introductions to the Boltzmann equation and its fluid dynamic limits, including \cite{CIP-1994,Cercignani-1988,SaintRaymond-2009-book}. Comparing to the extensive studies for the other fluid limits (such as incompressible Navier-Stokes, compressible Navier-Stokes and Euler equations), rigorous results on the incompressible Euler limit are much fewer. One of the main reasons is that in the incompressible Euler scaling, the available energy estimates for the scaled Boltzmann equation is not enough to derive the incompressible Euler equations directly. For the domains with boundary, the situation is even worse, since this limiting process includes the inviscid limit from the incompressible Navier-Stokes to Euler. It is well-known that in the limit, there arises the nonlinear Prandtl equations which are  not well-understood yet, in particular for the classical solutions. Furthermore, when the Boltzmann equation is endowed with boundary conditions (Maxwell reflection or incoming boundaries), the kinetic boundary layers (or Knudsen layer) is also generated. Except some simple cases (for example, specular reflection boundary), the kinetic Knudsen layers and the fluid Prandtl layers are strongly coupled in the leading order of the expansion. In other words, the boundary layers appears as the strong boundary layers. This usually makes the problem much more challenging. 

For the case that there is no boundary, an early contribution by De Masi, Esposito, and Lebowitz \cite{De Masi-Esposito-Lebowitz-1989-CPAM} established the incompressible Euler and Navier-Stokes limits on the torus via the Hilbert expansion method in the framework of classical solutions. After the DiPerna-Lions \cite{DiPerna-Lions}, Saint-Raymond \cite{SaintRaymond-2003-ARMA,SaintRaymond-2009-book} made important contributions on this problem in the class of renormalized solutions. She used the {\em H\mbox{-}}theorem and relative entropy method to prove the convergence from the fluctuations (no temperature) of the renormalized solutions of the Boltzmann equation towards the very weak (but very general) solutions of the Euler equation constructed by Lions \cite{Lions-1} called dissipative solutions. Thus it provided for the first time a non-conditional result.  It is also shown the weak-strong stability, i.e. if the Euler equation admits a smooth solution, then the convergence is strong. 

More recently, there are two major progress on the incompressible Euler limit with boundary in the framework of analytic solutions. Jang and Kim \cite{Jang-Kim} (and \cite{Cao-Jang-Kim-JDE}) rigorously justified that under the incompressible Euler scaling, the asymptotics of the incompressible Navier-Stokes equations with $O(\eps^\kappa)$ (here $\eps^\kappa=\frac{Knudsen\quad\!\!\!\!number}{Mach\quad\!\!\!\!number}$) dissipation from Boltzmann equation endowed the complete diffuse reflection boundary condition in the domain $\mathbb{T}^2 \times \mathbb{R}_+$, corresponding to accommodation coefficient $\alpha = 1$. Thus, Prandtl layer was not involved in this process. As a consequence, the incompressible Euler limit follows as a byproduct of the inviscid limit. In \cite{Kim-Nguyen-2024-arXiv}, Kim and Nguyen justified the incompressible Euler and incompressible Prandtl limits for shear flows without passing through the Navier-Stokes equations. This work is the first to rigorously derive the nonlinear incompressible Prandtl equations directly from the Boltzmann equation. Their perspective was more on the fluid dynamic side: the Prandtl layer was obtained from the Navier-Stokes system with small dissipation, rather than being derived directly from the Boltzmann equation. As a result, the Knudsen layer was not visible in their framework. They constructed a local Maxwellian consisting of an Euler profile plus a Prandtl layer, which approximates the solutions to the Navier-Stokes equations. To ensure control over the remainder, they required analyticity in time and tangential variables $(t, x_1, x_2)$. We also mention the work \cite{Cao-Jang-Kim-JDE}, which addressed the incompressible Euler limit in the presence of heat convection. All three of the above-mentioned works considered the fully diffuse reflection boundary condition (i.e. $\alpha=1$) and focused on the regime where the Knudsen number scales as $\varepsilon^q$ with $q $ very close to $1$. 

In \cite{Huang-Wang-Wang-Xiao-2025}, the incompressible Euler limit from the Boltzmann equation under the specular reflection boundary condition (i.e. $\alpha=0$) is justified also by employing the expansion method. For this simple boundary condition, the Prandtl layers are linear, and appear in the higher orders as weak boundary layers.  The justification could be done in the framework of classical solutions.

For related works on the incompressible and compressible Navier-Stokes and compressible Euler limits, we refer the reader to \cite{BGL1,BGL2,Briant-JDE2015,Golse-SRM-04,Golse-SRM-09,Guo-2006-CPAM,Jiang-Masmoudi-CPAM,Guo-Huang-Wang-2021-ARMA,Jiang-Luo-Tang-2021-arXiv, Jiang-Wu-2025-KRM, Jiang-Wu-CNS2025} and references therein. We omit further details on these well-known results, as the focus of the present paper is to justify the incompressible Euler limit. We only point out the key difference between the (compressible and incompressible)  Navier-Stokes and Euler limits from the point view of boundary layers. In fluid limits, once the Maxwellian determined by the limiting (or asympotic) fluid equations can not satisfy the boundary condition endowed on the original Boltzmann equation, the Knudsen layer would be needed. For the Knudsen boundary layer equations, the number of the solvability conditions depend on the speed of the boundary. In particular, when the boundary is not moving vertically,  the number of solvability conditions is four.   For the detailed information, see \cite{Bardos-Caflisch-Nicolaenko-1986-CPAM, Coron-Golse-Sulem-1988-CPAM, Golse-2008, Golse-Perthame-Sulem-1988-ARMA, He-Jiang-Wu-2024, Huang-Wang-SIMA2022, Huang-Jiang-Wang-2023, Jiang-Wu-2025-KRM, Jiang-Wu-CNS2025, Jiang-Luo-Wu-arXiv, Jiang-Luo-Wu-Yang-2025, Ukai-Yang-Yu-CMP2004}. For the compressible and incompressible Navier-Stokes equations (including temperature), the number of boundary conditions is four. In theses cases, Knudsen layers are enough to match the expansions. In other words, the boundary layer from fluid side, i.e. the Prandtl layers are not needed. However, it is well-known that for the Euler equations (both compressible and incompressible), only one boundary condition is needed (for simplicity, we assume the boundary is not moving vertically). Thus, the number of the boundary condition provided from the solvability of the Knudsen layers does not match with the number of the boundary conditions the fluid equation needed.  So for the Euler limits, both Knudsen and Prandtl layers are needed, and furthermore, they are coupled. This is the reason that usually, the limits to Euler equations are harder than those to the Navier-Stokes limits. It is only discussed for the steady boundary cases above. When the boundary is moving, the situations will be even more complex. Another issue needed to be remarked is that the boundary layer are strong or weak (i.e. in the leading order or the higher order in the expansion). This depends essentially on the types of the boundary conditions of the originally kinetic equations (Maxwell reflection or incoming boundaries, and the orders of the accommodation coefficients )  and  the types of the fluid equations (nonlinear or linear, Navier-Stokes or Euler, compressible or incompressible, etc.). 

In the current paper, we aim to prove the incompressible Euler limit (for the shear flow case) from the Boltzmann equation with general Maxwell reflection boundary condition, i.e. the accommodation coefficient $\alpha\in (0, 1]$. We employ the classical Hilbert expansion method in the framework of analytic solutions. The key novelty is that our expansion includes the complete coupled Knudsen and Prandtl boundary layers. Furthermore, in the incompressible Euler scaling, the Prandtl boundary layers appear strongly. This coupling in the leading order bring the strong singular behavior in the remainder equation. The treatment of this difficulty will be explained in subsection 1.5. 

\subsection{Notations}\label{subsec_notations}
Throughout this paper, we use the notation $\bar{U} = (U_1, U_2)$ for any vector $U = (U_1, U_2, U_3) \in \R^3$. Moreover, for simplicity of presentations, we use the notation $V^0 : = V |_{x_3 = 0}$ for any symbol $V = V(x)$, which may be a function, vector or operator. For any derivative operator $D_x$, we denote by $ D_x V^0 = (D_x V)^0 \,.$

For the functional spaces,  $H^s$ denotes the Sobolev space $W^{s,2} (\Omega)$ with norm $\| \cdot \|_{H^s}$. The notation $\| f \|_p$ refers to the $L^p_{x,v}$-norm when $f = f(t,x,v)$; otherwise, if $f = f(t,x)$, it denotes the $L^p_{x}$-norm. The symbol  $\l \cdot, \cdot \r$ is the $L^2_{x,v}$-inner product. We also introduce the following analytic norms:
\begin{equation}\label{def_X_Y}
  \begin{aligned}
    &\|g\|_X^2 = \sum_{\m \in \mathbb{N}_0^2} \tfrac{\sigma(t)^{2 |\m|}}{(\m !)^2} \|\p^\m g \|_2^2 \,,\\
     & \|g\|_{X^{\frac{1}{2}}}^2 = \sum_{\m \in \mathbb{N}_0^2} \tfrac{\sigma(t)^{2 |\m| - 1}| \m |}{(\m !)^2} \| \p^\m g \|_2^2 \,,\\
    & \|g\|_Y^2 = \sum_{\m \in \mathbb{N}_0^2} \tfrac{\sigma(t)^{2 |\m|}}{(\m !)^2} \| \p^\m g \|_\infty^2 \,.
  \end{aligned}
\end{equation}
where we used the notation $\p^\m = \p^{\m_1}_1 \p_2^{\m_2}$ for $\m = (\m_1 , \m_2) \in \mathbb{N}_0^2$.

For any function $G = G(t, \bar{x}, y, v)$, with $(t, \bar{x}, y, v) \in \R_+ \times \R^2 \times \overline{\R}_+ \times \R^3$,  the Taylor expansion at $y = 0$ is
\begin{equation}
  G = G^0 + \sum_{1 \leq l \leq N} \tfrac{y^l}{l !} G^{(l)} + \tfrac{y^{N+1}}{(N+1) !} \widetilde{G}^{(N+1)} \,,
\end{equation}
where the symbols
\begin{equation*}
  \begin{aligned}
    G^{(l)} = (\p_y^l G ) (t, \bar{x}, 0, v) \,, \ \widetilde{G}^{(N+1)} = ( \p_y^{N+1} G ) (t, \bar{x}, \eta, v) \ \textrm{for some } \eta \in (0, y).
  \end{aligned}
\end{equation*}
We define the global Maxwellian $\mu(v)$ as 
\begin{equation}\label{def_mu}
	\mu(v) = \tfrac{1}{(2 \pi)^{3/2}} e^{-\frac{|v|^2}{2}}\,.
\end{equation}
By using the global Maxwellian $\mu(v)$, the linearized collision operator $\L$ and the bilinear term $\Gamma$ are defined as
\begin{equation*}
  \begin{aligned}
    \L g &= - \frac{1}{\sqrt{\mu}} \Big\{ B (\mu, \sqrt{\mu} g) + B (\sqrt{\mu} g, \mu) \Big\} \,,\\
	\Gamma (f, g) &= \frac{1}{\sqrt{\mu}}  B (\sqrt{\mu}f, \sqrt{\mu} g)  \,.
  \end{aligned}
\end{equation*}
The null space $\mathcal{N}$ of $\L$ is spanned by (see \cite{Caflish-1980-CPAM}, for instance)
\begin{equation*}
  \begin{aligned}
      \sqrt{\mu} \,, \ v_i \sqrt{\mu} \ (i=1,2,3) \,, \   \tfrac{|v|^2 -3 }{2}  \sqrt{\mu} \,.
  \end{aligned}
\end{equation*}
The weighted $L^2$-norm
\begin{equation*}
\begin{aligned}
\| g \|^2_\nu = \int_{\Omega} \int_{\R^3} |g (x,v)|^2 \nu (v) \d x \d v \,,
\end{aligned}
\end{equation*}
is defined by the collision frequency $\nu (v)$
\begin{equation*}
\begin{aligned}
\nu (v)  = \int_{\R^3} \int_{\mathbb{S}^2} b (\theta) |v - v^\prime |^{\gamma} \mu(v^\prime ) \d v^\prime  \d \omega \,,
\end{aligned}
\end{equation*}
which satisfies
\begin{equation}\label{nu-phi}
\begin{aligned}
\nu  \thicksim \l v \r^{\gamma} \,.
\end{aligned}
\end{equation}
Here, $\l v \r = \sqrt{1 + |v|^2}$ and the notation $A \thicksim B $ means $C^{-1} B \leq A \leq CB$ for some harmless constant $C>0$. Let $\mathcal{P} g$ be the $L^2_v$ projection with respect to $\mathcal{N}$. Then it is well-known (see for example \cite{Caflish-1980-CPAM}) that there exists a positive number $c_0 > 0$ such that
\begin{equation}\label{Hypocoercivity}
  \begin{aligned}
    \l \L g, g \r \geq c_0 \| \P^\perp g \|^2_\nu\,.
  \end{aligned}
\end{equation} 
We can be decomposed $g$ into the macroscopic and microscopic parts:
\begin{equation*}
  \begin{aligned}
    g = & \mathcal{P}  g + \P^\perp  g \\
    \equiv &   \big(\rho + v \cdot u + (\tfrac{|v|^2-3}{2}) \theta \big)  \sqrt{\mu} + \P^\perp g \,,
  \end{aligned}
\end{equation*}
where $(\rho, u, \theta)$ represent the macroscopic density, bulk velocity, and temperature, respectively. 

For later use, we define the following Burnett functions:
\begin{equation}\label{def_A_B}
	\A_{ij}(v)= (v_i v_j - \tfrac{\delta_{ij}}{3} |v|^2) \sqrt{\mu}\,, \quad \B_i(v) = \tfrac{|v|^2 - 5}{2} v_i \sqrt{\mu}\,, \quad (i,j=1,2,3)\,,
\end{equation}
and the scalar function
\begin{equation}\label{def_C}
	\C(v) = (\tfrac{|v|^4}{4} - \tfrac{5 |v|^2}{2} + \tfrac{15}{4}) \sqrt{\mu}\,.
\end{equation}
It is straightforward to verify that $\A, \B , \C \in \mathcal{N}^\perp$, and each entry of $\A,\B,\C$ is perpendicular to each other. We define
\begin{equation}\label{hat_A_B}
  \hat{\A} = \L^{-1} \A\,, \quad \hat{\B} = \L^{-1} \B\,,
\end{equation}
where $\L^{-1}$ is the pseudo-inverse of $\L$, defined on $\mathcal{N}^\perp$. The properties of $\L^{-1}$ can be found in \cite{JLT-2022-arXiv}.

To derive uniform estimates for the remainder term, we further introduce the following modified collision operator:
\begin{equation}\label{def_L_eps_Gamma_eps}
  \L_\eps g = - \tfrac{1}{\sqrt{\mu_\eps}} \big\{ B(\mu_\eps, \sqrt{\mu_\eps} g ) + B( \sqrt{\mu_\eps} g , \mu_\eps ) \big\}, \quad  \Gamma_\eps (f,g) = \tfrac{1}{\sqrt{\mu_\eps}} B(\sqrt{\mu_\eps} f, \sqrt{\mu_\eps} g ) \,,
\end{equation}
where $\mu_\eps$ denotes the local Maxwellian
\begin{equation}\label{def_mu_eps}
  \mu_\eps(v) : = \tfrac{\rho_\eps}{(2 \pi \theta_\eps )^{\frac{3}{2  }}}e^{-\frac{|v - u_\eps|^2}{2 \theta_\eps}} \,,
\end{equation}
with $(\rho_\eps, u_\eps, \theta_\eps) = (1 + \eps(\rho_0 + \rho^b_0), \eps (u_0 + u^b_0), 1 + \eps (\theta_0 + \theta^b_0) )$. Here, $(\rho_0, u_0 ,\theta_0)$ solves the incompressible Euler system \eqref{incompressible_Euler_0}, and $(\rho^b_0, u^b_0 ,\theta^b_0)$ solves the incompressible Prandtl equation \eqref{nonlinear_Prandtl}. 

Define the null space of $\L_\eps$ by $\N_\eps$, which is spanned by 
\begin{equation}\label{def_chi_0_4}
  \chi_0 = \tfrac{1}{\sqrt{\rho_\eps}} \sqrt{\mu_\eps} \,,\ \chi_j = \tfrac{v_j - u_{\eps,j}}{\sqrt{\rho_\eps \theta_\eps}} \sqrt{\mu_\eps} \ (j=1,2,3)\,, \ \chi_4 = \tfrac{1}{\sqrt{\rho_\eps}} \big(\tfrac{|v - u_\eps|^2 }{2 \theta_\eps} - \tfrac{3}{2}\big)\sqrt{\mu_\eps}\,.
\end{equation}
 We decompose \(g = \P_\eps g + \P^\perp_\eps g \) into the macroscopic part and microscopic part, where
\begin{equation}
  \P_\eps g= a \chi_0 + \sum_{j=1}^{3} b_j  \chi_j + c \chi_4 \,,
\end{equation}
with
\begin{equation}\label{def_a_b_c}
  a = \int_{\R^3} g \chi_0 \d v\,,\  b_j = \int_{\R^3} g \chi_j \d v\,,\ c = \tfrac{2}{3} \int_{\R^3} g \chi_4 \d v \,.
\end{equation} 
Then, the coercivity of $\L_\eps$ implies that
\begin{equation}\label{def_tilde_c_0}
  \langle \L_\eps g , g \rangle \geq \hat{c}_0 \| \P_\eps^\perp g \|_{\nu_\eps}^2 \geq \tilde{c}_0\|\P^\perp_\eps g \|_\nu^2 \,,
\end{equation}
for some constants $\hat{c}_0, \tilde{c}_0>0$, where $\nu_\eps (v)$ denotes the collision frequency, given by 
\begin{equation}\label{def_nu_eps}
  \nu_\eps(v) = \int_{\R^3} \int_{\mathbb{S}^2} b(\theta) |v - v^\prime |^\gamma \mu_\eps(v^\prime) \d v^\prime \d \omega \sim \nu(v) \sim \langle v \rangle^\gamma \,.
\end{equation}

To employ the $L^2 \mbox{-} L^\infty$ framework \cite{Guo-2010-ARMA,Guo-Jang-Jiang-2009-KRM,Guo-Jang-Jiang-2010-CPAM}, we also introduce a global Maxwellian to derive the $L^\infty_{x,v}$ estimates
\begin{equation}\label{def_mu_M}
  \mu_M(v) := \tfrac{1}{(2 \pi \theta_M )^{\frac{3}{2  }}}e^{-\frac{|v  |^2}{2 \theta_M}} \,,
\end{equation}
where $\theta_M $ satisfies the condition
\begin{equation}\label{assumption_theta_M}
  \theta_M < \min_{(t,x) \in [0,T] \times \Omega} \theta_\eps \leq \max_{(t,x) \in [0,T] \times \Omega} \theta_\eps < 2 \theta_M \,.
\end{equation}
Note that under the condition \eqref{assumption_theta_M}, there exist constants $c_1$, $c_2$ such that for some $\frac{1}{2} < \alpha_0 <1$ and for each $(t,x,v) \in [0,T] \times \Omega \times \R^3$, the following holds
\begin{equation}\label{mu_eps_mu_M}
  c_1 \mu_M \leq \mu_\eps \leq c_2 \mu_M^{\alpha_0} \,.
\end{equation}
For the remainder term $g_{R,\eps}$, we define
 \begin{equation}\label{def_h}
  h_{R,\eps} = \sqrt{\tfrac{\mu_\eps}{\mu_M}} \omega_\beta (v) g_{R,\eps}\,,
 \end{equation}
 with the weight function
 \begin{equation}\label{def_omega_beta}
  \omega_\beta (v) := (1 + |v|^2)^{\beta} \quad \textrm{for} \quad \! \beta \geq \tfrac{9 - 2\gamma}{2}\,.
 \end{equation}

\subsection{Main result}In this paper, we employ the Hilbert expansion method to rigorously justify the incompressible Euler limit of the Boltzmann equation \eqref{BE} with the general Maxwell boundary condition \eqref{MBC} as $\varepsilon \to 0$. We consider the scaling regime where the Mach number is equal to the Strouhal number. Then, according to the classical von Karman relation, the Reynolds number is given by the ratio of the Mach number to the Knudsen number, i.e., 
\begin{equation*}
  \mathrm{Re}= \tfrac{\mathrm{Ma}}{\mathrm{Kn}} = \tfrac{1}{\eps^{q - 1}}\,.
\end{equation*}
The thickness of the Prandtl layer is $ O( \tfrac{1}{\sqrt{\mathrm{Re}}})= O(\eps^{\frac{q-1}{2}})$ (see \cite{Schlichting-1955-book} for instance) and the thickness of the Knudsen layer is $O(\mathrm{Kn})=O(\eps^q)$ (see \cite{Sone-2002book,Sone-2007-Book}). We choose $q= 2 $ in \eqref{BE} for simplicity of the ansatz. For the case of general $q>1$, the formal expansion will be more complex. But the analysis of the corresponding boundary layer equations will be basically the same. An outline of the formal derivation for the general $q>1$ will be provided in Section \ref{Sec_general_q}. As shown in Section \ref{sec_Hilbert} below, when $q =2 $, our truncated ansatz of solution to the Boltzmann equation takes the form
\begin{equation*}
   \begin{aligned}
    F_\eps (t,x,v) =& \mu +\eps \sum_{k=0}^{13} \sqrt{\mu}(\sqrt{\eps}^k g_k(t,x,v) +\sqrt{\eps}^k g^{b}_k(t,\bar{x},\tfrac{x_3}{\sqrt{\eps}},v) + \sqrt{\eps}^k g^{bb}_k)(t,\bar{x},\tfrac{x_3}{\eps^2},v) \\
    &+ \sqrt{\mu_\eps}\sqrt{\eps}^8 g_{R,\eps}(t,x,v) \,.
   \end{aligned}
\end{equation*}
Here, $\mu$ and $\mu_\eps$ are defined in \eqref{def_mu} and \eqref{def_mu_eps}, respectively. The functions $g_k$, $g^b_k$, and $g^{bb}_k$ correspond to the interior solution, the Prandtl boundary layer, and the Knudsen boundary layer, respectively, while $g_{R,\eps}$ denotes the remainder term.

For the initial data of the Boltzmann equation, $g_k^{in}, g^{b,in}_k, g^{bb,in}_k$ can be constructed by the same ways of constructing the expansions $g_k, g^{b}_k, g^{bb}_k$ in Section \ref{sec_Hilbert}, respectively. We impose the well-prepared initial data on the scaled Boltzmann equation \eqref{BE}: 
\begin{equation}\label{BE_initial_data}
 \begin{aligned}
   F_\eps (0,x,v) = & \mu +\eps \sum_{k=0}^{13} \sqrt{\mu}(\sqrt{\eps}^k g_k^{in}(x,v) +\sqrt{\eps}^k g^{b,in}_k(\bar{x},\tfrac{x_3}{\sqrt{\eps}},v) + \sqrt{\eps}^k g^{bb,in}_k)(\bar{x},\tfrac{x_3}{\eps^2},v) \\
   & + \sqrt{\mu_\eps}\sqrt{\eps}^8 g^{in}_{R,\eps}(x,v) \,.
 \end{aligned}
\end{equation}

Recall that the accommodation coefficient $\alpha \in (0,1]$ is assumed to be $O(1)$. Refer to \eqref{def_X_Y} for the definition of the analytic norms. We now present the result on the incompressible Euler limit for the shear flow.
\begin{theorem}\label{Main_theorem}
  Let $-3 < \gamma \leq 1$ and $q =2$. Assume that $\alpha \in(0 ,1]$ with $\alpha = O(1)$. Let $g_k(t,x,v) , g^b_k(t,\bar{x},\tfrac{x_3}{\sqrt{\eps}},v)$ and $g^{bb}_k(t,\bar{x},\tfrac{x_3}{\eps^2},v)$ be constructed as in Proposition \ref{Prop_g_k_g_b_k}. In particular, $$g_0 = (\rho_0 + v \cdot u_0 +\tfrac{|v|^2 -3}{2} \theta_0 )\sqrt{\mu} \quad \textrm{and} \quad \! g^b_0 = (\rho^b_0 + v \cdot u^b_0 + \tfrac{|v|^2 -3}{2}\theta^b_0 )\sqrt{\mu}\,,$$ where $(\rho_0,u_0,\theta_0)$ and $(\rho^b_0,u^b_0,\theta^b_0)$ are the solutions to the incompressible Euler system \eqref{incompressible_Euler_0} and the incompressible Prandtl system \eqref{nonlinear_Prandtl} for the shear flow, respectively. There is a small constant $\eps_0 >0$ such that if for $\eps \in (0,\eps_0)$
  \begin{equation*}
    \|g_{R,\eps}(0) \|_X+  \eps^3\| h_{R,\eps}(0) \|_{Y} < \infty\,,
  \end{equation*}
 where $h_{R,\eps}$ is given by \eqref{def_h} with $\beta \geq \tfrac{9 - 2 \gamma}{2}$, then there exist a time $T>0$ such that the scaled Boltzmann equation \eqref{BE} with Maxwell reflection boundary condition \eqref{MBC} and well-prepared initial data \eqref{BE_initial_data} admits a unique solution over the time interval $t \in[0,T]$ with the expanded form 
\begin{equation*}
   \begin{aligned}
    F_\eps (t,x,v) =& \mu +\eps \sum_{k=0}^{13} \sqrt{\mu}(\sqrt{\eps}^k g_k(t,x,v) +\sqrt{\eps}^k g^{b}_k(t,\bar{x},\tfrac{x_3}{\sqrt{\eps}},v) + \sqrt{\eps}^k g^{bb}_k)(t,\bar{x},\tfrac{x_3}{\eps^2},v) \\
    &+ \sqrt{\mu_\eps}\sqrt{\eps}^8 g_{R,\eps}(t,x,v) \geq 0 \,,
   \end{aligned}
\end{equation*}
where the remainder $g_{R,\eps}$ satisfies
  \begin{equation*}
   \| g_{R,\eps}\|_X^2 + \eps^6 \| h_{R,\eps}\|_Y^2  \leq C_T <\infty \,.
  \end{equation*}
\end{theorem}

\begin{remark}
  Theorem \ref{Main_theorem} indicates that 
  \begin{equation*}
    \sup_{t \in [0,T]} \| \tfrac{1}{\sqrt{\mu_\eps}} \big\{F_\eps - \mu - \eps (\rho_0 + v \cdot u_0  +\tfrac{|v|^2 -3}{2} \theta_0 ) \mu - \eps (\rho^b_0 +v \cdot u^b_0  +\tfrac{|v|^2 -3}{2} \theta^b_0 ) \mu\big\} \|_2 \leq C \eps^\frac{3}{2} \rightarrow 0 \,.
  \end{equation*} 
This justifies the incompressible Euler limit for the shear flow from the Boltzmann equation in the half-space.
\end{remark}
\begin{remark}
    For general $q > 1$, a formal derivation will be outlined in Section \ref{Sec_general_q}. By truncating the ansatz at a suitable order, the same method as in Section \ref{Sec_remainder_uniform} can be applied to obtain uniform estimates for the remainder term, thereby justifying the incompressible Euler limit for general $q > 1$. While the computations are slightly more involved, they remain analogous to the case $q = 2$. For simplicity, we focus primarily on the case $q = 2$ in this paper.
\end{remark}

\subsection{Sketch of ideas and novelties}\label{subsec_idea} We employ the Hilbert expansion method to construct an approximate solution to the scaled Boltzmann equation \eqref{BE} under the incompressible Euler scaling. Due to the Maxwell reflection boundary condition, coupled viscous-kinetic boundary layers (i.e. Prandtl and Knudsen layers) arise. When $q = 2$, the thickness of the Prandtl layer is $O(\sqrt{\eps})$, while that of the Knudsen layer is $O(\eps^2)$. Motivated by Caflisch \cite{Caflish-1980-CPAM}, we truncate the Hilbert expansion at a suitable order. As a result, the expansion takes the form
  \begin{equation*}
      F_\eps (t,x,v) = \mu +\sum_{k=0}^{13} \sqrt{\mu}(\sqrt{\eps}^k g_k +\sqrt{\eps}^k g^{b}_k + \sqrt{\eps}^k g^{bb}_k) + \sqrt{\mu} \sqrt{\eps}^8 \tilde{g}_{R,\eps}   \,,
  \end{equation*}
with the remainder equation
\begin{equation}\label{idea_tilde_g}
  \begin{aligned}
    \p_t \tilde{g}_{R, \eps} + \tfrac{1}{\eps} v \cdot \nabla_x \tilde{g}_{R, \eps} & +  \tfrac{1}{\eps^3} \L \tilde{g}_{R,\eps} = \eps^{2} \Gamma(\tilde{g}_{R,\eps}, \tilde{g}_{R,\eps})  \\
   & + \tfrac{1}{\eps^2}   [\Gamma(\tilde{g}_{R,\eps}, g_0 + g^b_{0}) + \Gamma( g_0 + g^b_{0}, \tilde{g}_{R,\eps})] + \textrm{s.o.t.}\,,
  \end{aligned}
\end{equation} 
where and hereafter s.o.t. refers to some other terms that are straightforward to be controlled. Here, $g_0 = (\rho_0 + u_0 \cdot v + \theta_0 \tfrac{|v|^2 -3}{2} )\sqrt{\mu}$, where $(\rho_0 , u_0 , \theta_0)$ solves the incompressible Euler system \eqref{incompressible_Euler_0}, and $g^b_0 = (\rho^b_0 + u^b_0 \cdot v + \theta^b_0 \tfrac{|v|^2 -3}{2} )\sqrt{\mu}$, where $(\rho^b_0 , u^b_0 , \theta^b_0)$ solves the incompressible Prandtl system \eqref{nonlinear_Prandtl}. To obtain the uniform estimate on the remainder equation \eqref{idea_tilde_g}, the most difficult term is $$\tfrac{1}{\eps^2}   [\Gamma(\tilde{g}_{R,\eps}, g_0 + g^b_{0}) + \Gamma( g_0 + g^b_{0}, \tilde{g}_{R,\eps})]\,.$$ 
We remark that this singular term does not appear in the previous approaches. In \cite{Jang-Kim}, there is no Prandtl layer term $g^b_0$ since they worked in incompressible Navier-Stokes asymptotics. In \cite{Huang-Wang-Wang-Xiao-2025}, the specular reflection was considered. For this boundary condition, the linear Prandtl layer appears in higher order. 

This singular term is treated in the following way: by performing the $L^2$ estimate on \eqref{idea_tilde_g}, we obtain
\begin{equation*}
  \tfrac{1}{\eps^2}   [\Gamma(\tilde{g}_{R,\eps}, g_0 + g^b_{0}) + \Gamma( g_0 + g^b_{0}, \tilde{g}_{R,\eps})] \tilde{g}_{R,\eps} \leq \tfrac{\delta}{\eps^3}\|\P^\perp \tilde{g}_{R,\eps} \|_\nu^2 + \tfrac{C_\delta}{\eps} \| \tilde{g}_{R,\eps} \|^2_\nu \|g_0 + g^b_0 \|_\nu^2\,.
\end{equation*}
Thanks to the coercivity of the linearized collision operator $\L$, the first term on the right-hand side can be absorbed into the left-hand side by choosing $\delta$ sufficiently small. However, the second is of order $O(\tfrac{1}{\eps})$ and thus singular. It is also emphasized that for different values of $q>1$, this singularity behaves like $O(\tfrac{1}{\eps^{q - 1}})$. Motivated by \cite{De Masi-Esposito-Lebowitz-1989-CPAM}, we  introduce $(\rho_\eps, u_\eps, \theta_\eps) = (1 + \eps(\rho_0 + \rho^b_0), \eps (u_0 + u^b_0), 1 + \eps (\theta_0 + \theta^b_0) )$ and the local Maxwellian
\begin{equation*} 
  \mu_\eps(v) = \tfrac{\rho_\eps}{(2 \pi \theta_\eps )^{\frac{3}{2  }}}e^{-\frac{|v - u_\eps|^2}{2 \theta_\eps}} \,.
\end{equation*}  
By introducing the collision operators $\L_\eps$ and $\Gamma_\eps$ around $\mu_\eps$, it follows that
\begin{equation*}
 \tfrac{1}{\eps^3} \L_\eps (\sqrt{\tfrac{\mu}{\mu_\eps}} \tilde{g}_{R,\eps}) \sim \tfrac{1}{\eps^3} \L \tilde{g}_{R,\eps} - \tfrac{1}{\eps^2}[\Gamma(\tilde{g}_{R,\eps}, g_0 + g^b_{0}) + \Gamma( g_0 + g^b_{0}, \tilde{g}_{R,\eps})] + \textrm{s.o.t.} \,.
\end{equation*}
This implies that the singular term can be absorbed into $\L_\eps$. Defining $g_{R,\eps} = \sqrt{\tfrac{\mu}{\mu_\eps}} \tilde{g}_{R,\eps}$, the equation for $g_{R,\eps}$ is then expressed as follows
\begin{equation}
 \begin{aligned}
   \p_t g_{R,\eps} + \frac{1}{\eps} v \cdot \nabla_x g_{R,\eps} + \tfrac{1}{\eps^3} \L_\eps g_{R,\eps} + \tfrac{(\p_t + \frac{1}{\eps} v \cdot \nabla_x ) {\mu_\eps}}{2 \mu_\eps} g_{R,\eps}  = \eps^2 \Gamma_\eps (g_{R,\eps}, g_{R,\eps})+ \textrm{s.o.t.}\,.
 \end{aligned}
\end{equation} 
However, due to the presence of the Prandtl layer, the normal derivatives  $\p_3(\rho^b_0, u^b_0, \theta^b_0) $ exhibit singular behavior of order $O(\tfrac{1}{\sqrt{\eps}})$. Consequently, the term $\tfrac{1}{\eps} v \cdot \nabla_x \mu_\eps \sim O(\tfrac{1}{\sqrt{\eps}})$. The $L^2$ estimate indicates that
\begin{equation*}
  \int_{\Omega \times \R^3} \tfrac{(\p_t + \tfrac{1}{\eps} v \cdot \nabla_x) \mu_\eps}{2 \mu_\eps} g_{R,\eps}^2 \d x \d v = \tfrac{1}{\eps} \int_\Omega \sum_{i=1}^{2}\p_3 u_{\eps, i} b_3 b_i + \p_3(\rho_\eps + \theta_\eps) b_3 (a + c) + \p_3 \theta_\eps b_3 c \d x + \textrm{s.o.t.} \,,
\end{equation*}
where $(a,b,c)$ are the macroscopic quantities of $g_{R,\eps}$. The conservation laws of the Boltzmann equation imply that
  \begin{equation}\label{idea_conservation}
    \begin{cases}
      &\eps \p_t a  + \sqrt{\theta_\eps} \nabla_x \cdot b  = \textrm{s.o.t.} \,,\\
      &\eps \p_t b  + \sqrt{\theta_\eps} \nabla_x (a  + c ) = - \nabla_x \cdot (u_\eps \otimes b )  -\nabla_x \cdot \int_{\R^3}  \theta_\eps^{-\frac{1}{4}} \A(v_\eps) \P_\eps^\perp g  \d v +\textrm{s.o.t.} \,,\\
      &\eps \p_t c  + \tfrac{2}{3}\sqrt{\theta_\eps} \nabla_x \cdot b  =  - \tfrac{2}{3} \nabla_x \cdot \int_{\R^3} \theta_\eps^{-\frac{1}{4}} \B(v_\eps) \P_\eps^\perp g  \d v + \textrm{s.o.t.} \,.\\
    \end{cases}
  \end{equation}
We explain below our major new idea, which is based on the key observation that the singularity can be transferred into a loss of tangential derivatives. Importantly, this observation does not entail any loss of the time derivative. Hence, the solution is required to be analytic only in the tangential spatial variables, not in time. This represents a significant improvement compared to \cite{Kim-Nguyen-2024-arXiv}. In what follows, we only consider the term $\tfrac{1}{\eps} \int_{\Omega} \p_3 u_{\eps, i} b_3 b_i \d x \ (i=1,2)$. Noticing that $ \tfrac{1}{\eps}x_3 \p_3 (\rho_\eps , u_\eps , \theta_\eps) = O(1) $ and $b_3(t, \bar{x} ,0) = 0$, we have by $\eqref{idea_conservation}_1$,
\begin{equation*}
  \begin{aligned}
    \tfrac{1}{\eps}\int_{\Omega} \p_3 u_{\eps,i} b_3 b_i \d x =&  \tfrac{1}{\eps}\int_{\Omega} \p_3 u_{\eps,i} \int_{0}^{x_3} \p_3 b_3 \d x^\prime_3 b_i \d x \\
    = &- \tfrac{1}{\eps}\int_{\Omega} \p_3 u_{\eps,i} \int_{0}^{x_3} (\eps \p_t a + \sum_{j=1}^{2}\p_j b_j )\d x^\prime_3 b_i \d x + \textrm{s.o.t.} \\
    = &- \tfrac{1}{\eps}\int_{\Omega} \p_3 u_{\eps,i} \int_{0}^{x_3} (\eps \p_t a )\d x^\prime_3 b_i \d x - \tfrac{1}{\eps}\int_{\Omega} x_3 \p_3 u_{\eps,i} \tfrac{\int_{0}^{x_3} (\sum_{j=1}^{2}\p_j b_j )\d x^\prime_3}{x_3} b_i \d x + \textrm{s.o.t.}  \\
    \leq &- \tfrac{1}{\eps}\int_{\Omega} \p_3 u_{\eps,i} \int_{0}^{x_3} (\eps \p_t a )\d x^\prime_3 b_i \d x + C \sum_{j=1}^{2}\| \p_j b_j \|_2 \|b_i \|_2 + \textrm{s.o.t.}  \,, \\
  \end{aligned}
\end{equation*} 
where we used the Hardy inequality. The second term loss one order tangential derivative. To get rid of the loss of time derivative , we use the equality $\p_t f  g = - \p_t g f + \tfrac{\d}{\d t} (fg)$ and $\eqref{idea_conservation}_2$ to derive
    \begin{equation*}
  \begin{aligned}
   - \tfrac{1}{\eps}\int_{\Omega} \p_3 u_{\eps,i} \int_{0}^{x_3} (\eps \p_t a )\d x^\prime_3 b_i \d x   
   = & \tfrac{1}{\eps}\int_{\Omega} \p_3 u_{\eps,i} \int_{0}^{x_3} (\eps a )\d x^\prime_3 \p_t b_i \d x  + \textrm{s.o.t.}\\
    = & \tfrac{1}{\eps}\int_{\Omega} \p_3 u_{\eps,i} \int_{0}^{x_3}  a \d x^\prime_3 ( - \p_i (a + c) - \div \int_{\R^3}  \theta_\eps^{-\frac{1}{4}} \A(v_\eps) \P_\eps^\perp g  \d v ) \d x  \\
    \leq & C \|a \|_2 \|\p_i (a+c) \|_2 + \sum_{j=1}^{2} \|a\|_2 \|\p_j \P^\perp_\eps g_{R,\eps} \|_2 \\
    &+ \tfrac{1}{\eps}\int_{\Omega} \p_3 u_{\eps,i} \int_{0}^{x_3}  a \d x^\prime_3 ( - \p_3 \int_{\R^3}  \theta_\eps^{-\frac{1}{4}} \A(v_\eps) \P_\eps^\perp g  \d v ) \d x + \textrm{s.o.t.} \,,
  \end{aligned}
\end{equation*} 
where the Hardy inequality has been used again. By performing integration by parts, it follows that
\begin{equation*}
  \begin{aligned}
     & \tfrac{1}{\eps}\int_{\Omega} \p_3 u_{\eps,i} \int_{0}^{x_3}  a \d x^\prime_3 ( - \p_3 \int_{\R^3}  \theta_\eps^{-\frac{1}{4}} \A(v_\eps) \P_\eps^\perp g_{R,\eps}  \d v  ) \d x \\
    =& \tfrac{1}{\eps}\int_{\Omega} \p_3 u_{\eps,i} a ( \int_{\R^3}  \theta_\eps^{-\frac{1}{4}} \A(v_\eps) \P_\eps^\perp g_{R,\eps}  \d v  ) \d x  
    \leq \tfrac{C}{\sqrt{\eps}} \|a\|_2 \| \P^\perp g_{R,\eps}\|_\nu \,,
  \end{aligned}
\end{equation*}
which can be controlled since $\|\P^\perp_\eps g_{R,\eps} \|^2_\nu = O(\tfrac{1}{\eps^3}) $. As a result,
\begin{equation*}
  \tfrac{1}{\eps}\int_{\Omega} \p_3 u_{\eps,i} b_3 b_i \d x \leq C \| g_{R,\eps}\|_2(\sum_{j=1}^{2} \|\p_j g_{R,\eps} \|_2) +  \textrm{s.o.t.} \,.
\end{equation*}
This implies a loss of one order in tangential spatial derivatives. Thus, the energy estimates can be closed by working within analytic function spaces. Moreover, the nonlinear term $\varepsilon^2 \Gamma_\varepsilon(g_{R,\varepsilon}, g_{R,\varepsilon})$ can be estimated using the standard $L^2\mbox{-}L^\infty$ framework developed in \cite{Guo-2010-ARMA, Guo-Jang-Jiang-2009-KRM, Guo-Jang-Jiang-2010-CPAM}. 

Finally, we make some remarks on the non-shear flow case. In this setting, the above analysis remains valid in principle. Because of the non-commutativity the operator $\L_\eps$ and tangential derivatives, some further new estimates on the corresponding commutator will be needed. In particular, the term $\tfrac{1}{\eps^3} [\L_\eps , \p_i] g $ is highly singular. For simplicity, we only focus on the shear flow in the current paper. The incompressible Euler limit for general non-shear flow case will be treated in a separate paper.

\subsection{Organization of this paper}
The paper is organized as follows. In Section \ref{sec_Hilbert}, we present a formal analysis, where we derive the incompressible Euler and Prandtl systems, along with the corresponding asymptotic expansions. Section \ref{Sec_solution_expansion} is devoted to establishing uniform estimates for the solutions of the constructed expansions. In Section \ref{Sec_remainder_uniform}, we obtain uniform estimates for the remainder term within the analytic $L^2\mbox{-}L^\infty$ framework. Finally, in Section \ref{Sec_general_q}, we outline the formal analysis for the case of a general scaling parameter $q>1$.

\section{Hilbert expansion}\label{sec_Hilbert}
In this section, we employ the Hilbert expansion method to formally derive the incompressible Euler system from the Boltzmann equation \eqref{BE}–\eqref{MBC} in the limit $\varepsilon \to 0$, under the scaling parameter $q = 2$. This derivation is based on analyzing the fluctuations of the distribution function $F_\eps$ around a global Maxwellian. More precisely, we take 
\begin{equation}\label{fluctuation}
	F_\eps = \mu + \eps \sqrt{\mu} g_\eps\,,
\end{equation}
where $\mu$ is defined in \eqref{def_mu}. From plugging \eqref{fluctuation} into \eqref{BE}-\eqref{MBC}, we obtain
\begin{equation}\label{BE_g_eps}
	\begin{cases}
    \eps \p_t g_\eps + v \cdot \nabla_x g_\eps +\tfrac{1}{\eps^2}\L g_\eps=\tfrac{1}{\eps}\Gamma(g_\eps,g_\eps) \quad \text{in} \quad \! \R_+ \times \Omega \times \R^3\,, \\
    g_\eps(0,x,v) = g^{in}_\eps (x,v) \quad \text{in}  \quad \! \Omega \times \R^3\,,
  \end{cases}
\end{equation}
with the Maxwell reflection boundary condition 
\begin{equation}\label{MBC_g_eps}
    \gamma_- g_\eps (t,x,v) = (1 - \alpha) g_\eps (t,x,R_x  v)+ \alpha P_\gamma g(t,x,v) \,,
\end{equation}
where we used the notation $R_x v = v - 2 (v \cdot n) n = (\bar{v}, -v_3)$, and
\begin{equation}\label{def_P_gamma}
  P_\gamma g(t,x,v)=\sqrt{2\pi \mu(v)}\int_{v^\prime\cdot n>0}(v^\prime\cdot n)\sqrt{\mu(v^\prime)} g(t,x,v^\prime)\d v^\prime.
\end{equation}

As will be explained below, our expansion has the form
\begin{equation}\label{Hilbert-Expnd-Form}
  \begin{aligned}
     g_\eps (t, x , v) = & \sum_{k\geq 0} \sqrt{\eps}^k \big\{ g_k (t,x,v) + g_k^b (t, \bar{x}, \tfrac{x_3}{\sqrt{\eps}}, v) \big\} \\
     &+ \sum_{k\geq 2} \sqrt{\eps}^k  g_k^{bb} (t, \bar{x}, \tfrac{x_3}{\eps^2}, v) \,,
  \end{aligned}
\end{equation}
where $g_k (t,x,v)$, $ g_k^b (t, \bar{x}, \tfrac{x_3}{\sqrt{\eps}}, v)$ and $g_k^{bb} (t, \bar{x}, \tfrac{x_3}{\eps^2}, v)$ refer to the interior part, the Prandtl layer part, and the Knudsen layer part, respectively.
\subsection{Interior expansion}
By the famous von Karman relation, the Reynolds number equal to the ratio of the Mach number to the Knudsen number, i.e.,
\begin{equation*}
  \mathrm{Re}= \tfrac{\mathrm{Ma}}{\mathrm{Kn}} = \tfrac{1}{\eps}\,.
\end{equation*}
The thickness of the Prandtl layer is $ O( \tfrac{1}{\sqrt{\mathrm{Re}}})= O(\sqrt{\eps})$, and the thickness of the Knudsen layer is $O(\mathrm{Kn})=O(\eps^2)$. Therefore, our expansion in interior takes the form
\begin{equation}\label{expansion_interior}
  \begin{aligned}
    g_\eps (t, x, v) \thicksim \sum_{k \geq 0} \sqrt{\eps}^k g_k (t, x, v) \,.
  \end{aligned}
\end{equation}
Plugging \eqref{expansion_interior} into \eqref{BE_g_eps} and collecting the same orders,
\begin{equation}\label{Order_Anal_Interior}
\begin{aligned}
\sqrt{\eps}^{-4}:& \quad \L g_0 = 0\,,\\
\sqrt{\eps}^{-3}:& \quad \L g_1 = 0\,,\\
\sqrt{\eps}^{-2}:& \quad \L g_2 = \Gamma(g_0,g_0)\,,\\
\sqrt{\eps}^{-1}:& \quad \L g_3 = \Gamma(g_0,g_1) + \Gamma(g_1,g_0)\,,\\
\sqrt{\eps}^{0}:& \quad v \cdot \nabla_x g_0 + \L g_4 = \Gamma(g_0,g_2) + \Gamma(g_2,g_0) + \Gamma(g_1,g_1)\,,\\
\cdots \cdots &\\
\sqrt{\eps}^k:& \quad \p_t g_{k-2} + v \cdot \nabla_x g_k + \L g_{k+4} = \sum_{\substack{i+j=k+2\,,\\ i, j\ge 0}} \Gamma(g_i, g_j)\,.
\end{aligned}
\end{equation}
We denote
\begin{equation*}
  \P g_k = \big(\rho_k + u_k \cdot v + \theta_k \tfrac{|v|^2-3}{2}\big)\sqrt{\mu} := (\sqrt{\mu},v\sqrt{\mu}, \tfrac{|v|^2-3}{2}\sqrt{\mu})\cdot U_k\,,
\end{equation*}
where $ (\rho_k, u_k, \theta_k)$ represent the macroscopic density, bulk velocity and temperature, respectively. Then the equation of $O(\sqrt{\eps}^{-4})$ and {\em H}-theorem implies that $g_0$ must be an infinitesimal Maxwellian:
\begin{equation}\label{def_g0}
g_0 (t, x, v)  = I_0(U_0) : = \big(\rho_0(t, x) + v \cdot u_0(t,x) +( \tfrac{|v|^2-3}{2}) \theta_0(t,x) \big) \sqrt{\mu}\,,
\end{equation}
where $U_0 = (\rho_0,u_0,\theta_0)$ ($U_1,U^b_1,\cdots$, that appears later are defined in the same way).\\
The contribution of $O(\sqrt{\eps}^{-3})$ gives 
\begin{equation}\label{def_g1}
g_1 (t, x, v)  = I_0(U_1) : = \big(\rho_1(t, x) + v \cdot u_1(t, x) + \tfrac{|v|^2-3}{2} \theta_1(t, x) \big) \sqrt{\mu}\,.
\end{equation}
The contribution of $O(\sqrt{\eps}^{-2})$ gives 
\begin{equation}\label{def_g2}
  g_2=I_0(U_2) + I_1(U_0) : = \big(\rho_2(t, x) + v \cdot u_2(t, x) + \tfrac{|v|^2-3}{2} \theta_2(t, x) \big) \sqrt{\mu} + \tfrac{1}{2}\big(\A : u_0 \otimes u_0 + 2 \B \cdot u_0 \theta_0 + C \theta_0^2 \big)\,,
\end{equation}
where $\A \in \R^{3\times 3}, \B\in \R^3$ and $\C \in \R$ are the Burnett functions defined in subsection \ref{subsec_notations} and we have used the fact that (see \cite{BGL1})
\begin{equation}\label{Gamma_f_g}
	\begin{aligned}
		\L^{-1} \{\Gamma(\P f, \P g) + \Gamma(\P g,\P f)\} =& \P^\perp \{\tfrac{\P f \cdot \P g}{\sqrt{\mu}}\}\,\\
		=&\P^\perp \big\{ \big(\rho_f + v \cdot u_f + \tfrac{|v|^2-3}{2} \theta_f \big) \big(\rho_g + v \cdot u_g + \tfrac{|v|^2-3}{2} \theta_g \big) \sqrt{\mu} \big\}\,,\\
		=&  \A : u_f \otimes u_g +  \B \cdot (u_f \theta_g + u_f \theta_g )+ \C \theta_f \theta_g\,.
	\end{aligned}
\end{equation}
Here, $(\rho_f,u_f,\theta_f)$ represent the macroscopic quantities of the function $f$.

The contribution of $O(\sqrt{\eps}^{-1})$ gives 
\begin{equation}\label{def_g3}
  g_3=I_0(U_3) + \big(\A : u_0 \otimes u_1 +  \B \cdot (u_0 \theta_1 + u_1 \theta_0 )+ \C \theta_0 \theta_1 \big)
\end{equation}
The solvability condition of $O(1)$ is
 \begin{equation*}
  v \cdot \nabla_x g_0 \in \mathcal{N}^\perp\,.
 \end{equation*}
 By simple calculations and using \eqref{def_g0},
 \begin{equation*}
	v \cdot \nabla_x g_0 =(1,v,\tfrac{|v|^2-3}{2})\sqrt{\mu}\mathscr{A}U_0 + \A : \nabla_x u_0 + \B \cdot \nabla_x \theta_0\,,
 \end{equation*}
 where $\mathscr{A}$ is the acoustic operator defined by
 \begin{equation}
  \mathscr{A}\begin{pmatrix}
    \rho \\
     u\\
    \theta 
    \end{pmatrix}
    =\begin{pmatrix}
    \operatorname{div} u  \\
    \nabla_x (\rho+ \theta ) \\
    \tfrac{2}{3}\operatorname{div} u
    \end{pmatrix}\,.
 \end{equation}
Since $\A,\B \in \mathcal{N}^\perp$, we must have
 \begin{equation}\label{Boussinesq_incompressibility_0}
  \operatorname{div} u_0 =0\,, \quad \nabla_x(\rho_0 + \theta_0 ) =0\,.
 \end{equation}
 This represents the incompressibility condition and the Boussinesq relation. Integrability requires that $\rho_0 + \theta_0 \equiv 0$ (for cases in a bounded domain, the Boussinesq relation should be considered alongside the conservation law of the Boltzmann equation to determine the value of $\rho_0 + \theta_0$, see \cite{Sone-2002book,Sone-2007-Book,Takata-2012-JSP}).

The solvability condition of $O(\sqrt{\eps}^{2})$ is
\begin{equation*}
	\p_t g_0 + v \cdot \nabla_x g_2 \in \mathcal{N}^\perp\,.
\end{equation*}
By simple calculations, we have
 \begin{equation}\label{incompressible_Euler_0}
  \begin{cases}
    \p_t \rho_0 + \operatorname{div} u_2 =0\,,\\
    \p_t u_0 + \nabla_x (\rho_2 + \theta_2) + \operatorname{div} (u_0 \otimes u_0) -\tfrac{1}{3} \nabla_x |u_0|^2 =0 \,,\\
    \p_t \theta_0 + \tfrac{2}{3} \operatorname{div} u_2 + \tfrac{5}{3}\operatorname{div} (u_0 \theta_0 )=0\,.
  \end{cases}
 \end{equation}
 By a simple cancellation and using \eqref{Boussinesq_incompressibility_0}, we know that $(\rho_0, u_0 , \theta_0)$ satisfy the incompressible Euler system:
 \begin{equation}
  \begin{cases}
    \rho_0 + \theta_0 = 0 \,,\\
    \operatorname{div} u_0 = 0 \,,\\
    \p_t u_0 + u_0 \cdot \nabla_x u_0 + \nabla_x p_0 =0\,,\\
    \p_t \theta_0 + u_0 \cdot \nabla \theta_0 = 0\,, 
  \end{cases}
 \end{equation}
 where  
 \begin{equation}\label{def_p0}
  p_0(t,x)= (\rho_2 +\theta_2) - \tfrac{1}{3}|u_0|^2
 \end{equation}
 is the pressure. Moreover, we can obtain the incompressibility condition and Boussinesq relation for $(u_2, \theta_2)$ from \eqref{incompressible_Euler_0}:
 \begin{equation}\label{Boussinesq_incompressibility_2}
  \operatorname{div} u_2 = -\p_t \rho_0\,, \quad \nabla_x (\rho_2 + \theta_2 - \tfrac{1}{3} |u_0|^2) =0\,.
 \end{equation}
 
 Next, we derive the system for $(\rho_k, u_k, \theta_k) \ (k \geq 1)$. Indeed, from \eqref{Order_Anal_Interior}, we have
 \begin{equation}
	\L g_{k+2} = - \p_t g_{k-4} - v \cdot \nabla g_{k-2} + \sum_{\substack{i+j=k\,,\\ i, j\ge 0}} \Gamma(g_i, g_j)\\
	=\Gamma (\P g_k, g_0) + \Gamma (g_0, \P g_k) + F_{k-1}\,,
 \end{equation}
 where 
 \begin{equation}
	F_{k-1} = - \p_t g_{k-4} - v \cdot \nabla_x g_{k-2} + \sum_{\substack{i+j=k\,,\\ i, j\ge 1}} \Gamma(g_i, g_j) + \Gamma (\P^\perp g_k, g_0) + \Gamma (g_0, \P^\perp g_k)
 \end{equation}
 depends only on $(\rho_j, u_j, \theta_j) \ (j\leq k-1)$. Multiplying \eqref{Order_Anal_Interior} with $O(\sqrt{\eps}^{k+2})$ by $\sqrt{\mu},v \sqrt{\mu} ,\tfrac{|v|^2-3}{2}\sqrt{\mu}$, one has
 \begin{equation}\label{linear_Euler}
	\begin{cases}
		\p_t \rho_k + \operatorname{div} u_{k+2} =0\,,\\
		\p_t u_k + \operatorname{div} ( u_k \otimes u_0 + u_0 \otimes u_k) + \nabla_x (\rho_{k+2} + \theta_{k+2}) - \tfrac{2}{3}\nabla_x (u_0 \cdot u_k) = \mathcal{F}_{k-1} \,,\\
		\tfrac{3}{2}\p_t \theta_k + \operatorname{div} u_{k+2} + \tfrac{5}{2} \operatorname{div} (u_k \theta_0 + u_0 \theta_k) = \mathcal{G}_{k-1}\,,
	\end{cases}
 \end{equation}
 where 
 \begin{equation}
	\begin{aligned}
		\mathcal{F}_{k-1} &= - \div  \int_{\R^3} \A \L^{-1} F_{k-1} \d v\,,\\
		\mathcal{G}_{k-1} &= - \div  \int_{\R^3} \B \L^{-1} F_{k-1} \d v\,.\\
	\end{aligned}
 \end{equation}
 Once $(\rho_{k}, u_{k}, \theta_{k})$ are determined, the incompressibility and the Boussinesq relation of $(\rho_{k+2},\\ u_{k+2}, \theta_{k+2})$ can be derived. Together with \eqref{linear_Euler} for $k+2$, $(\rho_{k+2}, u_{k+2}, \theta_{k+2})$ can be completely solved, inductively.
 
 We further impose the initial data on $(\rho_k,u_k,\theta_k)$ (this can be realized by setting the special form of the initial data of the Boltzmann equation \eqref{BE}):
\begin{equation}\label{IC_incompressible_Euler}
  \begin{aligned}
   (\rho_0, u_0, \theta_0) (0, x) = (\rho_0^{in}, u^{in}_0, \theta_0^{in}) (x_3)\,,\quad  (\rho_k, u_k, \theta_k) (0, x) = (\rho_k^{in}, u^{in}_k, \theta_k^{in}) (x)\,,\quad k=1,2\cdots \,.
  \end{aligned}
\end{equation}
This choice implies that the leading-order term corresponds to a shear flow.

\subsection{Viscous boundary layer expansion}
Obviously, $g_0$ in \eqref{def_g0} matches the Maxwell reflection boundary condition \eqref{MBC} if and only if 
\begin{equation*}
  u_0^0=0\,,\quad \theta_0^0 =0\,.
\end{equation*}
We recall that the superscript $0$ indicates the value on the boundary. However, the constraints above are too many as the boundary conditions for the incompressible Euler system. Therefore, the boundary layer correction is naturally needed. Since the thickness of the viscous layer is $O(\sqrt{\eps})$, the expansion inside the viscous layer has the form
\begin{equation}
  \begin{aligned}
    g^b_\eps (t, \bar{x}, \zeta, v) \thicksim \sum_{k \geq 0} \sqrt{\eps}^k g^b_k (t, \bar{x}, \zeta, v) \,,
  \end{aligned}
\end{equation}
where $\zeta = \tfrac{x_3}{\sqrt{\eps}} $ is the scaled normal coordinate. Throughout our paper, the far-field condition is always assumed:
\begin{equation}\label{viscous_far_field}
  \begin{aligned}
    g^b_k (t, \bar{x}, \zeta , v) \to 0 \,, \quad \textrm{as} \quad\! \zeta \to + \infty \,.
  \end{aligned}
\end{equation}
Plugging $g_\eps + g^b_\eps$ into the Boltzmann equation \eqref{BE_g_eps} and rearrange by the order of $\eps$, we have
\begin{align}\label{Order_Anal_Viscous}
\no \sqrt{\eps}^{-4}:\quad & \L g^b_0=  0\,,\\
\no \sqrt{\eps}^{-3}:\quad & \L g^b_1=  0\,,\\
\no \sqrt{\eps}^{-2}:\quad & \L g^b_2=  \Gamma (g_0^0, g^b_0) + \Gamma (g_0^b, g^0_0) + \Gamma(g^b_0,g^b_0)\,,\\
\no \sqrt{\eps}^{-1}:\quad & v_3 \cdot \p_\zeta g^b_0 + \L g^b_3= \big[ \Gamma(g^b_0, g^b_1 + g^0_1) +\Gamma( g^b_1 + g^0_1, g^b_0)\big] + \big[\Gamma(g^0_0, g^b_1) + \Gamma(g^b_1, g^0_0)\big]  \\
\no & \qquad \qquad  + \zeta \big[ \Gamma(g_0^{(1)}, g^b_0) +  \Gamma(g^b_0, g_0^{(1)})\big] \,,\\
\no & \cdots \cdots\\
 \sqrt{\eps}^k: \quad & \p_t g^b_{k-2} + \bar{v} \cdot \nabla_{\bar{x}} g^b_k + v_3 \cdot \p_\zeta g^b_{k+1} + \L g^b_{k+4} \\
\no & \qquad = \sum_{\substack{i+j=k+2\,,\\ i,j \ge 0 }} \Gamma(g^b_i, g^b_j) + \sum_{\substack{i+j=k+2\,,\\ i,j \ge 0 }} \Big[ \Gamma(g^0_i, g^b_j) + \Gamma(g^b_j, g^0_i) \Big]\\
\no & \qquad \quad + \sum_{\substack{i+l+j=k+2\,,\\ 1\le l \le N\,, i,j \ge 0}} \tfrac{\zeta^l}{l!} \Big[ \Gamma(g^{(l)}_i, g^b_j) + \Gamma(g^b_j, g^{(l)}_i) \Big] \,,
\end{align}
where the Taylor expansion at $x_3 = 0$ is used:
\begin{equation*}
  \begin{aligned}
    g_i = g_i^0 + \sum_{1 \leq l \leq N} \tfrac{\zeta^l}{l !} g_i^{(l)} + \tfrac{\zeta^{N+1}}{(N+1) !} \widetilde{g}_i^{(N+1)} \,.
  \end{aligned}
\end{equation*}
Here, the number $N \in \mathbb{N}_+$ will be chosen later. We define
\begin{equation}\label{pb_k}
  \begin{aligned}
    p^b_k = \rho^b_k + \theta^b_k \,.
  \end{aligned}
\end{equation}
Similar to the analysis in the interior, we have
\begin{equation}\label{def_gb0}
  g^b_0= B(U^b_0) :=  \big( \rho^b_0 + u^b_0 \cdot v + \theta^b_0 \tfrac{|v|^2 -3}{2}\big)\sqrt{\mu}\,,
\end{equation}
\begin{equation}\label{def_gb1}
  g^b_1= B(U^b_1) :=  \big( \rho^b_1 + u^b_1 \cdot v + \theta^b_1 \tfrac{|v|^2 -3}{2}\big)\sqrt{\mu}\,,
\end{equation}
\begin{equation}\label{def_gb2}
  \begin{aligned}
    g^b_2=& B(U^b_2) + \big( \A : u^0_0 \otimes u^b_0 +  \B (u^0_0 \theta^b_0 + u^b_0 \theta^0_0 )+ \C \theta^0_0 \theta^b_1 \big)\\
    &+ \tfrac{1}{2}\big( \A : u^b_0 \otimes u^b_0 + 2 \B u^b_0 \theta^b_0 + \C (\theta^b_0)^2 \big)\,.
  \end{aligned}
\end{equation}
The contribution of $O(\sqrt{\eps}^{-1})$ gives
\begin{equation}
  v_3 \partial_\zeta g^b_0 \in \mathcal{N}^\perp  \,.
\end{equation}
That is 
\begin{equation}
  \partial_\zeta (\rho^b_0 + \theta^b_0) v_3 \sqrt{\mu} + \partial_\zeta u^b_{0,3} \tfrac{|v|^2}{3} \sqrt{\mu} + \partial_\zeta \theta^b_0 \B_3 + \sum_{j=1}^{3} \A_{3j} \partial_\zeta u^b_{0,j} \in \mathcal{N}^\perp\,,
\end{equation}
which further implies
\begin{equation}
  \partial_\zeta (\rho^b_0 + \theta^b_0)=0\,, \partial_\zeta u^b_{0,3}=0\,.
\end{equation}
Together with the far-field condition \eqref{viscous_far_field}, we obtain
\begin{equation}\label{pb0_ub03}
 p^b_0 (t,\bar{x},\zeta)= (\rho^b_0 + \theta^b_0)(t,\bar{x},\zeta) \equiv 0\,,\quad  u^b_{0,3}(t,\bar{x},\zeta) \equiv 0\,.
\end{equation}
Moreover, we obtain that
\begin{equation}\label{def_gb3}
 \begin{aligned}
   g^b_3 = &B(U^b_3) + \L^{-1} \big\{ - \P^\perp (v_3 \p_\zeta g^b_0) + \big[ \Gamma(g^b_0, g^b_1 + g^0_1) +\Gamma( g^b_1 + g^0_1, g^b_0)\big] + \big[\Gamma(g^0_0, g^b_1) + \Gamma(g^b_1, g^0_0)\big]  \\
    &+ \zeta \big[ \Gamma(g_0^{(1)}, g^b_0) +  \Gamma(g^b_0, g_0^{(1)})\big] \big\} \,,\\
   = & B(U^b_3) -\sum_{j=1}^{3} \hat{\A}_{3j} \p_\zeta u^b_{0,j} - \hat{\B}_3 \p_\zeta \theta^b_0 \\
   &+ \big( \A : u^b_0 \otimes (u^b_1 + u^0_1 + \zeta u^{(1)}_0 )+  \B (u^b_0 (\theta^b_1 + \theta^0_1 + \zeta \theta^{(1)}_0 ) + \theta^b_0 (u^b_1 + u^0_1 + \zeta u^{(1)}_0) )\\
   &+ \C \theta^b_0 (\theta^b_1 + \theta^0_1 + \zeta \theta^{(1)}_0) \big)\\
   &+ \big(\A : u^0_0 \otimes u^b_1 +  \B (u^0_0 \theta^b_1 + u^b_1 \theta^0_0 )+ \C \theta^0_0 \theta^b_1 \big) \,,
 \end{aligned}
\end{equation}
where $\hat{\A} = \L^{-1} \A\,, \hat{\B} = \L^{-1} \B$ are defined in \eqref{hat_A_B}.

The following lemma states that the fluid quantities of $g^b_k$ satisfies the Prandtl-type equations.
\begin{lemma}\label{Lmm-Prandtl}
Let $g^b_k$ be the solution of \eqref{Order_Anal_Viscous}, then $(u^b_{0,1},u^b_{0,2},u^b_{1,3},\theta^b_0)$ satisfy the {\em nonlinear incompressible Prandtl equations} \eqref{nonlinear_Prandtl}, and $(u^b_{k,1},u^b_{k,2},u^b_{k+1,3},\theta^b_k)$ $(k\geq 1)$ satisfy the {\em linear incompressible Prandtl equations} \eqref{linear_incompressible_Prandtl}. The sources $\mathrm{f}^b_{k-1,i}$ and $\mathrm{g}^b_{k}$ are given in \eqref{f_b_k-1} and \eqref{g_b_k}, respectively.  Once we have solved $(u^b_{k,1},u^b_{k,2},u^b_{k+1,3},\theta^b_k)(k\geq 1)$, then $ \P^\perp g^b_{k+2}$ and $p^b_{k+2}$ can be determined by \eqref{def_gb_k+2} and \eqref{pb_k+2}, respectively.
\end{lemma}
Before going into the proof, we present some frequently used equalities. Let
$$M:= \A: u_1 \otimes u_2 + \B (u_1 \theta_2 + u_2 \theta_1) + \C(\theta_1 \theta_2)\,.$$
Then, for fixed $i$, we have
\begin{align}
  \sum_{j=1}^{2}\p_j \int_{\R^3} \A_{ij} M \d v &=\sum_{j=1}^{2} \p_j(u_{1,i} u_{2,j} + u_{1,j} u_{2,i} - \tfrac{2}{3} \delta_{ij} u_1 \cdot u_2) \,, \label{A_ij_M}\\
  \p_\zeta \int_{\R^3} \A_{i3} M \d v &= \p_\zeta(u_{1,3} u_{2,i} + u_{1,i} u_{2,3} - \tfrac{2}{3} \delta_{i3} u_1 \cdot u_2) \,, \label{A_i3_M}\\
  \sum_{j=1}^{2}\p_j \int_{\R^3} \B_{j} M \d v &=\sum_{j=1}^{2} \tfrac{5}{2} \p_j(u_{1} \theta_2 + u_{2} \theta_1)_j \,, \label{B_j_M}\\
  \p_\zeta \int_{\R^3} \B_{3} M \d v &= \tfrac{5}{2} \p_\zeta (u_{1} \theta_2 + u_{2} \theta_1)_3 \,,\label{B_3_M}
\end{align}
where we have used (see \cite[lemma 4.4]{BGL2} for instance) 
\begin{equation}\label{A_ij_A_kl}
\l {\A}_{ij},\A_{kl} \r =(\delta_{ij}\delta_{jl}+\delta_{il}\delta_{jk}-\tfrac{2}{3}\delta_{ij}\delta_{kl})\,, \quad 
  \l {\B}_{i},\B_{j} \r = \tfrac{5}{2} \delta_{ij}\,. 
\end{equation}
We also have 
\begin{equation}\label{hat_A_ij_A_kl}
  \l \hat{\A}_{ij},\A_{kl} \r = \kappa_1(\delta_{ij}\delta_{jl}+\delta_{il}\delta_{jk}-\tfrac{2}{3}\delta_{ij}\delta_{kl})\,, \quad 
    \l \hat{\B}_{i},\B_{j} \r = \kappa_2 \tfrac{5}{2} \delta_{ij} \,,
  \end{equation}
  for constants $\kappa_1,\kappa_2 >0$.
\begin{proof}[Proof of Lemma \ref{Lmm-Prandtl}]
Multiplying the equation of $O(\sqrt{\eps}^{k+2})$ in \eqref{Order_Anal_Viscous} by $\sqrt{\mu},v \sqrt{\mu},\tfrac{|v|^2-3}{2}\sqrt{\mu}$, and integrating over $v \in \mathbb{R}^3$, yields
\begin{equation}\label{viscous_macro_eqn}
  \begin{cases}
    \partial_t\rho^b_k+\sum_{j=1}^2\partial_ju_{k+2,j}^b+\partial_\zeta u^b_{k+3,3}=0\,,\\
    \partial_tu^b_{k,i}+\partial_i(\rho^b_{k+2}+\theta^b_{k+2}) +\sum_{j=1}^{2}\int_{\mathbb{R}^3} \A_{ij} \P^\perp \partial_j g^b_{k+2} \d v + \int_{\mathbb{R}^3} \A_{i3} \P^\perp \partial_\zeta g^b_{k+3} \d v =0   \quad (i=1,2)\,,\\
    \partial_t u^b_{k,3} + \partial_\zeta(\rho^b_{k+3}+\theta^b_{k+3}) +\sum_{j=1}^{2}\int_{\mathbb{R}^3} \A_{3j} \P^\perp \partial_j g^b_{k+2} \d v + \int_{\mathbb{R}^3} \A_{33} \P^\perp \partial_\zeta g^b_{k+3} \d v=0\,,   \\
    \tfrac{3}{2} \partial_t\theta^b_k - \partial_t\rho^b_k +\sum_{j=1}^{2}\int_{\mathbb{R}^3} \B_j \P^\perp\partial_j g^b_{k+2} \d v +\int_{\mathbb{R}^3} \B_3 \P^\perp \partial_\zeta g^b_{k+3} \d v=0\,.
  \end{cases}
\end{equation}
By setting $k=-2$ in $\eqref{viscous_macro_eqn}_1$, we obtain the incompressibility condition
\begin{equation}\label{incompressibility_ub0}
  \sum_{j=1}^{2}\p_j u^b_{0,j} + \p_\zeta u^b_{1,3} =0\,.
\end{equation}
By setting $k=-1$ in $\eqref{viscous_macro_eqn}_3$, and using the fact that $u^b_{0,3} (t, \bar{x}, \zeta) \equiv 0$, we have
\begin{equation*}
  \p_\zeta (\rho^b_2 + \theta^b_2) - \tfrac{2}{3} \p_\zeta[u^b_0 \cdot u^0_0 + \tfrac{1}{2}|u^b_0|^2] =0\,.
\end{equation*}
Together with the far-field condition \eqref{viscous_far_field},
\begin{equation}\label{pb2}
   (\rho^b_2 + \theta^b_2) - \tfrac{2}{3}  [u^b_0 \cdot u^0_0 + \tfrac{1}{2}|u^b_0|^2] \equiv 0 \,.
\end{equation}
Now, we are already to derive the equations for $( u^b_{0,1}, u^b_{0,2}, u^b_{1,3}, \theta^b_0)$. \\
\underline{\em Calculation of $\sum_{j=1}^{2}\int_{\mathbb{R}^3} \A_{ij} \P^\perp \partial_j g^b_{2} \d v$ with $i=1,2$.}
By \eqref{def_gb2} and \eqref{A_ij_M},
\begin{equation*}
	\begin{aligned}
		\sum_{j=1}^{2}\int_{\mathbb{R}^3} \A_{ij} \P^\perp \partial_j g^b_{2} \d v 
		= &\sum_{j=1}^{2} \p_j [u^b_{0,i}(u^0_0 + \tfrac{1}{2} u^b_0)_j +u^b_{0,j}(u^0_0 + \tfrac{1}{2} u^b_0)_i - \tfrac{2}{3} \delta_{ij} u^b_{0} \cdot (u^0_0 + \tfrac{1}{2} u^b_0) ]\\
		= &\sum_{j=1}^{2} \p_j [u^b_{0,i} u^b_{0,j} + u^b_{0,i} u^0_{0,j} +u^b_{0,j}u^0_{0,i} - \tfrac{2}{3}u^b_{0} \delta_{ij} \cdot (u^0_0 + \tfrac{1}{2} u^b_0) ]\,.
	\end{aligned}
\end{equation*}
\underline{\em Calculation of $\int_{\mathbb{R}^3} \A_{i3} \P^\perp \partial_\zeta g^b_{3} \d v$ with $i=1,2$.} By \eqref{def_gb3}, \eqref{A_i3_M} and \eqref{hat_A_ij_A_kl},
\begin{equation*}
	\begin{aligned}
		\int_{\mathbb{R}^3} \A_{i3} \P^\perp \partial_\zeta g^b_{3} \d v
		= & - \p^2_\zeta \int_{\R^3} \A_{3i} \sum_{j=1}^{3}\hat{ \A}_{3j} u^b_{0,j} \d v  + \p_\zeta [u^b_{0,3}(u^b_1 + u^0_1 + \zeta u^{(1)}_0)_i 
		+ u^b_{0,i}(u^b_1 + u^0_1 + \zeta u^{(1)}_0)_3] \\
    & +\p_\zeta (u^0_{0,3} u^b_{1,i} + u^b_{1,3} u^0_{0,i})\\
		= & -\kappa_1 \p^2_\zeta u^b_{0,i} + \p_\zeta  [ u^b_{0,i}(u^b_1 + u^0_1 + \zeta u^{(1)}_0)_3] +\p_\zeta (u^b_{1,3} u^0_{0,i})\,,
	\end{aligned}
\end{equation*}
where we have used the fact that $u^b_{0,3} = u^0_{0,3} =0$ in the last equality, see \eqref{pb0_ub03} and \eqref{BC_u_03}.\\
\underline{\bf Equation for $u^b_{0,i}$ with $i=1,2$.} By \eqref{viscous_macro_eqn} and the above two identities,
\begin{equation*}
	\begin{aligned}
		\partial_tu^b_{0,i}+\partial_i(\rho^b_{2}+\theta^b_{2}) +\sum_{j=1}^{2} \p_j [u^b_{0,i} u^b_{0,j} + u^b_{0,i} u^0_{0,j} +u^b_{0,j}u^0_{0,i} - \delta_{ij} \tfrac{2}{3}u^b_{0} \cdot (u^0_0 + \tfrac{1}{2} u^b_0) ] \\
		-\kappa_1 \p^2_\zeta u^b_{0,i} + \p_\zeta  [ u^b_{0,i}(u^b_1 + u^0_1 + \zeta u^{(1)}_0)_3] +\p_\zeta (u^b_{1,3} u^0_{0,i})=0\,.
	\end{aligned}
\end{equation*}
By a simple cancellation,
\begin{equation}\label{eqn_ub_0i}
	\begin{aligned}
		\partial_tu^b_{0,i} &+\sum_{j=1}^{2} u^b_{0,j} \p_j u^b_{0,i} + u^b_{1,3} \p_\zeta u^b_{0,i} - \kappa_1 \p^2_\zeta u^b_{0,i} \\
		&+ \p_j \sum_{j=1}^{2} (u^b_{0,i} u^0_{0,j}) + \sum_{j=1}^{2} u^b_{0,j} \p_j u^0_{0,i} + \p_\zeta  [ u^b_{0,i}( u^0_1 + \zeta u^{(1)}_0)_3] =0\,,
	\end{aligned}
\end{equation}
where we have used \eqref{incompressibility_ub0}  and \eqref{pb2}. These are the equations for $u^b_{0,1}$ and $u^b_{0,2}$.

\underline{\em Calculation of $\sum_{j=1}^{2}\int_{\mathbb{R}^3} \A_{ij} \P^\perp \partial_j g^b_{2} \d v$ with $i=3$.}
Noting that $u^b_{0,3} = u^0_{0,3} = 0$, we have from \eqref{def_gb2} and \eqref{A_ij_M}
\begin{equation*}
	\begin{aligned}
		\sum_{j=1}^{2}\int_{\mathbb{R}^3} \A_{3j} \P^\perp \partial_j g^b_{2} \d v 
		= \sum_{j=1}^{2} \p_j [u^b_{0,3}(u^0_0 + \tfrac{1}{2} u^b_0)_j +u^b_{0,j}(u^0_0 + \tfrac{1}{2} u^b_0)_3]= 0\,.
	\end{aligned}
\end{equation*}
\underline{\em Calculation of $\int_{\mathbb{R}^3}\A_{i3} \P^\perp \partial_\zeta g^b_{3} \d v$ with $i=3$.} By \eqref{def_gb3} and \eqref{A_i3_M},
\begin{equation*}
	\begin{aligned}
		\int_{\mathbb{R}^3} \A_{33} \P^\perp \partial_\zeta g^b_{3} \d v
		= &- \p^2_\zeta \int_{\R^3} \A_{33} \sum_{j=1}^{3}\hat{ \A}_{3j} u^b_{0,j} \d v \\
		&+ \p_\zeta [ 2 u^b_{0,3}(u^b_1 + u^0_1 + \zeta u^{(1)}_0)_3 +
		  2 u^0_{0,3} u^b_{1,3} ] - \tfrac{2}{3} \p_\zeta[   u^b_{0} \cdot (u^b_1 + u^0_1 + \zeta u^{(1)}_0)
		  + u^0_{0} \cdot u^b_{1} ] \\
		= & -\tfrac{4}{3} \kappa_1 \p^2_\zeta u^b_{0,3} - \tfrac{2}{3} \p_\zeta  [ u^b_{0} \cdot (u^b_1 + u^0_1 + \zeta u^{(1)}_0) + (u^b_{1} \cdot u^0_{0}) ]\,.
	\end{aligned}
\end{equation*}
\underline{\bf Equation for $u^b_{0,3}$}
\begin{equation*}
	\begin{aligned}
		\partial_tu^b_{0,3}+\partial_\zeta(\rho^b_{3}+\theta^b_{3})- \tfrac{2}{3} \p_\zeta [ u^b_{0} \cdot (u^b_1 + u^0_1 + \zeta u^{(1)}_0) + (u^b_{1} \cdot u^0_{0}) ]=0\,.
	\end{aligned}
\end{equation*}
\begin{remark}
  The value of $u^b_{0,3}$ is already obtained from \eqref{pb0_ub03}. The above system actually provides the equation for $p^b_3 = \rho^b_3 + \theta^b_3$.
\end{remark}
Next, we consider the equation for $\theta^b_0$.\\
\underline{\em Calculation of $\sum_{j=1}^{2}\int_{\mathbb{R}^3} \B_{j} \P^\perp \partial_j g^b_{2} \d v$.} By \eqref{B_j_M} and \eqref{def_gb2}
\begin{equation*}
	\sum_{j=1}^{2}\int_{\mathbb{R}^3} \B_{j} \P^\perp \partial_j g^b_{2} \d v = \tfrac{5}{2} \sum_{j=1}^{2} \p_j (u^b_{0,j} \theta^b_0 + u^0_{0,j} \theta^b_0 + u^b_{0,j} \theta^0_0 )\,.
 \end{equation*}
\underline{\em Calculation of $\int_{\mathbb{R}^3} \B_3 \P^\perp \partial_j g^b_{3} \d v$.} By \eqref{B_3_M}, \eqref{hat_A_ij_A_kl}, and \eqref{def_gb3}
\begin{equation*}
	\begin{aligned}
		\int_{\mathbb{R}^3} \B_3 \P^\perp \partial_j g^b_{3} \d v
		=& - \tfrac{5}{2} \kappa_2 \p^2_\zeta \theta^b_0 + \tfrac{5}{2}\p_\zeta [u^b_{0,3}(\theta^b_1 + \theta^0_1 + \zeta \theta^{(1)}_0) + \theta^b_{0}(u^b_1 + u^0_1 + \zeta u^{(1)}_0)_3]\\
   & + \tfrac{5}{2} \p_\zeta (u^0_{0,3} \theta^b_1 + u^b_{1,3} \theta^0_0 )\,.
	\end{aligned}
\end{equation*}
\underline{\bf Equation for $\theta^b_0$.} By \eqref{viscous_macro_eqn} and the above equalities,
\begin{equation*}
	\tfrac{3}{2}\p_t \theta^b_0 - \p_t \rho^b_0 + \tfrac{5}{2} \sum_{j=1}^{2} \p_j (u^b_{0,j} \theta^b_0 + u^0_{0,j} \theta^b_0 + u^b_{0,j} \theta^0_0 ) - \tfrac{5}{2} \kappa_2 \p^2_\zeta \theta^b_0 + \tfrac{5}{2}\p_\zeta [\theta^b_{0}(u^b_1 + u^0_1 + \zeta u^{(1)}_0)_3 +u^b_{1,3} \theta^0_0] = 0\,,
\end{equation*}
where we have used the fact that $u^b_{0,3} = u^0_{0,3}$ again. By a simple cancellation
\begin{equation*}
	 \p_t \theta^b_0 +  \sum_{j=1}^{2}   u^b_{0,j} \p_j \theta^b_0 + u^b_{1,3} \p_\zeta \theta^b_0 -  \kappa_2 \p^2_\zeta \theta^b_0 + \sum_{j=1}^{2} \p_j (u^0_{0,j} \theta^b_0 + u^b_{0,j} \theta^0_0) + \p_\zeta [\theta^b_{0}(u^0_1 + \zeta u^{(1)}_0)_3 + u^b_{1,3}\theta^0_0] =0\,,
\end{equation*} 
where we used \eqref{pb0_ub03} and \eqref{incompressibility_ub0}. Equivalently,
\begin{equation}\label{eqn_thetab_0}
  \p_t \theta^b_0 +  \sum_{j=1}^{2}   u^b_{0,j} \p_j \theta^b_0 + u^b_{1,3} \p_\zeta \theta^b_0 -  \kappa_2 \p^2_\zeta \theta^b_0 + \sum_{j=1}^{2} \p_j (u^0_{0,j} \theta^b_0) + \sum_{j=1}^{2} u^b_{0,j} \p_j \theta^0_0 + \p_\zeta [\theta^b_{0}(u^0_1 + \zeta u^{(1)}_0)_3] =0\,.
\end{equation} 
As a consequence, we obtain the following incompressible Prandtl equation:
\begin{equation}\label{nonlinear_Prandtl}
	\begin{cases}
    \rho^b_0 + \theta^b_0 = 0\,,\\
		\sum_{j=1}^{2}\p_j u^b_{0,j} + \p_\zeta u^b_{1,3} =0\,,\\
		\partial_tu^b_{0,i} +\sum_{j=1}^{2} u^b_{0,j} \p_j u^b_{0,i} + u^b_{1,3} \p_\zeta u^b_{0,i} - \kappa_1 \p^2_\zeta u^b_{0,i} \\
		\qquad\qquad + \p_j \sum_{j=1}^{2} (u^b_{0,i} u^0_{0,j}) + \sum_{j=1}^{2} u^b_{0,j} \p_j u^0_{0,i} + \p_\zeta  [ u^b_{0,i}( u^0_1 + \zeta u^{(1)}_0)_3] = 0 \quad (i=1,2)\,,\\
		\p_t \theta^b_0 +  \sum_{j=1}^{2}   u^b_{0,j} \p_j \theta^b_0 + u^b_{1,3} \p_\zeta \theta^b_0 -  \kappa_2 \p^2_\zeta \theta^b_0 + \sum_{j=1}^{2} \p_j (u^0_{0,j} \theta^b_0) \\
    \qquad\qquad+ \sum_{j=1}^{2} u^b_{0,j} \p_j \theta^0_0 + \p_\zeta [\theta^b_{0}(u^0_1 + \zeta u^{(1)}_0)_3] =0\,,\\
    \lim_{\zeta \to \infty}(\bar{u}^b_0,u^b_{1,3},\theta^b_0) =0\,.
	\end{cases}
\end{equation}
\begin{remark}
 The quantities $u_0$ and $\theta_0$ are uniquely determined by \eqref{incompressible_Euler_0} and the boundary condition \eqref{BC_u_03}, while the quantity $u^0_{1,3}$ can be expressed in terms of $u^{b}_{1,3}$ using the relation \eqref{BC_u_13} and the far-field condition \eqref{viscous_far_field}, i.e.
 $$u^0_{1,3} = - u^{b,0}_{1,3} =  \int_0^\infty \p_\zeta u^{b}_{1,3} \d \zeta \,.$$
 Hence, the above system is well-posed.
\end{remark}
Next, we consider the cases $k \geq 1$ in \eqref{viscous_macro_eqn}.
It is known from \eqref{Order_Anal_Viscous} that 
  \begin{align*}
    \mathcal{L}g^b_{k+2}=& -\P^\perp \big(\partial_t g^b_{k-4} +  \bar{v}\cdot \nabla_{\bar{x}} g^b_{k-2} + v_3\partial_\zeta g^b_{k-1}\big)\\
    &+\sum_{\stackrel{i+j=k,}{i, j \geq 0}} \Gamma\left(g_{i}^{b}, g_{j}^{b}\right) + \sum_{\stackrel{i+j=k, }{i, j \geq 0}}\left[\Gamma\left(g_{i}^{0}, g_{j}^{b}\right)+\Gamma\left(g_{j}^{b}, g_{i}^{0}\right)\right] \\
    & +\sum_{\stackrel{i+j+l=k,}{ 1 \leq l \leq N, i, j \geq 0}} \frac{\zeta^{l}}{l!}\left[\Gamma\left(g_{i}^{(l)}, g_{j}^{b}\right)+\Gamma\left(g_{j}^{b}, g_{i}^{(l)}\right)\right] \\
    : =& \Gamma(g_0^b + g_0^0, \P g^b_k) + \Gamma(\P g^b_k, g_0^b + g_0^0 ) + \L J_{k-1}\,,
  \end{align*} 
where 
\begin{equation}\label{def_J_k-1}
  \begin{aligned} 
   J_{k-1}=&\mathcal{L}^{-1}\left\{-\P^\perp  \big(\p_t g^b_{k-4} + \bar{v}\cdot\nabla_{\bar{x}}  g^b_{k-2} + v_3 \p_\zeta g^b_{k-1}\big) \right.\\
   &\phantom{=\;\;}
   \left.{+\sum_{\stackrel{i+j=k, }{i \geq 1, j \geq 0}}\left[\Gamma\left(g_{i}^{0}, g_{j}^{b}\right) + \Gamma\left(g_{j}^{b}, g_{i}^{0}\right)\right] 
 +\sum_{\stackrel{i+j=k,}{i, j \geq 1}} \Gamma\left(g_{i}^{b}, g_{j}^{b}\right)} \right.\\
 &\phantom{=\;\;}
 \left.\sum_{\stackrel{i+j+l=k, }{i,j \geq 0, 1\leq l \leq N}}\frac{\zeta^l}{l!}\left[\Gamma(g^{(l)}_i,g^b_j)+\Gamma(g^b_j,g^{(l)}_i)\right]\right.\\
 &\phantom{=\;\;}
 \left.{
 +\Gamma(g^b_0 +g^0_0,\P^\perp g^b_k)+\Gamma( \P^\perp g^b_k,g^b_0 +g^0_0) }
 \right\}\,,
\end{aligned}
\end{equation}
which depends only on the lower-order terms. Moreover, 
  \begin{align*}
    \mathcal{L}g^b_{k+3}=& -\P^\perp \big(\partial_t g^b_{k-3} +  \bar{v}\cdot \nabla_{\bar{x}} g^b_{k-1} + v_3\partial_\zeta g^b_{k}\big)\\
    &+\sum_{\stackrel{i+j=k+1,}{i, j \geq 0}} \Gamma\left(g_{i}^{b}, g_{j}^{b}\right) + \sum_{\stackrel{i+j=k+1, }{i, j \geq 0}}\left[\Gamma\left(g_{i}^{0}, g_{j}^{b}\right)+\Gamma\left(g_{j}^{b}, g_{i}^{0}\right)\right] \\
    & +\sum_{\stackrel{i+j+l=k+1,}{ 1 \leq l \leq N, i, j \geq 0}} \frac{\zeta^{l}}{l!}\left[\Gamma\left(g_{i}^{(l)}, g_{j}^{b}\right)+\Gamma\left(g_{j}^{b}, g_{i}^{(l)}\right)\right] \\
    =& - \P^\perp (v_3 \p_\zeta \P g^b_k) + \Gamma(g_0^b + g_0^0, \P g^b_{k+1}) + \Gamma(\P g^b_{k+1}, g_0^b + g_0^0 )\\
    & + \Gamma(g_1^b + g_1^0, \P g^b_{k}) + \Gamma(\P g^b_{k}, g_1^b + g_1^0 ) + \zeta [\Gamma( g^{(1)}_0, \P g^b_k) + \Gamma( \P g^b_k, g^{(1)}_0 ) ] \\
    &+  [\Gamma( g^b_0, \P g^0_{k+1}) + \Gamma( \P g^0_{k+1}, g^b_0 ) ] - \delta_{k1} \Gamma(g^b_k,g^b_1)+ \L I_{k-1} \,,
  \end{align*}
where
\begin{equation} \label{def_I_k-1}
  \begin{aligned} 
   I_{k-1}=&\mathcal{L}^{-1}\left\{-\P^\perp  \big(\p_t g^b_{k-3} + \bar{v}\cdot\nabla_{\bar{x}}  g^b_{k-1} + v_3 \p_\zeta \P^\perp g^b_{k}\big) \right.\\
   &\phantom{=\;\;}
   \left.{+\sum_{\stackrel{i+j=k+1, }{i \geq 2, j \geq 1}}\left[\Gamma\left(g_{i}^{0}, g_{j}^{b}\right) + \Gamma\left(g_{j}^{b}, g_{i}^{0}\right)\right] 
 +\sum_{\stackrel{i+j=k+1,}{i, j \geq 2}} \Gamma\left(g_{i}^{b}, g_{j}^{b}\right)} \right.\\
 &\phantom{=\;\;}
 \left.\sum_{\stackrel{i+j+l=k+1, }{i \geq 0, 0 \leq j \leq k-1, 1\leq l \leq N}} \tfrac{\zeta^l}{l!}\left[\Gamma(g^{(l)}_i,g^b_j)+\Gamma(g^b_j,g^{(l)}_i)\right]\right.\\
 &\phantom{=\;\;}
 \left.
 +\Gamma(g^b_0 + g^0_0, \P^\perp g^b_{k+1} )+\Gamma(\P^\perp g^b_{k+1} , g^b_0 + g^0_0 )+\Gamma(g^b_1 + g^0_1, \P^\perp g^b_k)+\Gamma( \P^\perp g^b_k, g^b_1 + g^0_1) 
 \right.\\
 &\phantom{=\;\;}
 \left.
  + \zeta [\Gamma( g^{(1)}_0, \P^\perp g^b_k) + \Gamma( \P^\perp g^b_k, g^{(1)}_0 )] +[\Gamma( g^b_0, \P^\perp g^0_{k+1}) + \Gamma( \P^\perp g^0_{k+1}, g^b_0 ) ]
  \right\} \,.
\end{aligned}
\end{equation}
Together with \eqref{Gamma_f_g}, we obtain
\begin{equation}\label{def_gb_k+2}
  \begin{aligned}
    \P^\perp g^b_{k+2}=& \A(v):u^b_k\otimes (u^b_0 + u^0_0) 
    + \B(v) \cdot \big(u^b_k(\theta^b_0 + \theta^0_0) +\theta^b_k(u^b_0 + u^0_0)\big)\\
    &+ \C(v)\theta^b_k(\theta^b_0 + \theta^0_0) +J_{k-1} \,,
  \end{aligned}
\end{equation}
and
\begin{equation}\label{def_gb_k+3}
  \begin{aligned}
    \P^\perp g^b_{k+3}=&-\sum_{j=1}^{3}\hat{\A}_{3j}\partial_\zeta u^b_{k,j} -\hat{\B}_3 \partial_\zeta \theta^b_k + \A(v):u^b_{k+1} \otimes (u^b_0 + u^0_0)\\
    &+ \B(v) \cdot \big(u^b_{k+1}( \theta^b_0 + \theta^0_0) +\theta^b_{k+1} (u^b_0 + u^0_0)\big) + \C(v)\theta^b_{k+1} (\theta^b_0 + \theta^0_0) \\
     &+\A(v):u^b_k\otimes (u^b_1 + u^0_1+ \zeta u^{(1)}_0) 
     + \B(v) \cdot \big(u^b_k(\theta^b_1 + \theta^0_1 + \zeta \theta^{(1)}_0) +\theta^b_k(u^b_1 + u^0_1 + \zeta u^{(1)}_0)\big)\\
     &+ \C(v)\theta^b_k(\theta^b_1 + \theta^0_1 + \zeta \theta^{(1)}_0)\\
     & +\A(v):u^0_{k+1} \otimes u^b_0 + \B(v) \cdot \big( u^0_{k+1} \theta^b_0 +\theta^0_{k+1} u^b_0 \big) + \C(v)\theta^0_{k+1} \theta^b_0 \\
     &- \delta_{k1}\tfrac{1}{2}[\A(v):u^b_{1} \otimes u^b_{1} + 2 \B(v) \cdot u^b_{1} \theta^b_1  + \C(v)\theta^b_{1} \theta^b_1]
     + I_{k-1}\,.
  \end{aligned}
\end{equation}
\underline{\em Calculation of $\sum_{j=1}^{2}\int_{\mathbb{R}^3} \A_{ij} \P^\perp \partial_j g^b_{k+2} \d v$ with $i=1,2$.} By \eqref{def_gb_k+2} and \eqref{A_ij_M},
	\begin{align*}
		&\sum_{j=1}^{2}\int_{\mathbb{R}^3} \A_{ij} \P^\perp \partial_j g^b_{k+2} \d v \\
		= &\sum_{j=1}^{2} \p_j [u^b_{k,i}(u^0_0 +  u^b_0)_j +u^b_{k,j}(u^0_0 +  u^b_0)_i - \delta_{ij} \tfrac{2}{3}u^b_{k} \cdot (u^0_0 + u^b_0) ] + \sum_{j=1}^{2} \int_{\R^3} \A_{ij} \p_j J_{k-1} \d v\,.
	\end{align*} 
\underline{\em Calculation of $\int_{\mathbb{R}^3}A_{i3} \P^\perp \partial_\zeta g^b_{k+3} \d v$ with $i=1,2$.} By \eqref{def_gb_k+3}, \eqref{A_i3_M} and \eqref{hat_A_ij_A_kl},
		\begin{align*}
      \int_{\mathbb{R}^3} \A_{i3} \P^\perp \partial_\zeta g^b_{k+3} \d v
      		= & -\kappa_1 \p^2_\zeta u^b_{k,i} 
      		+ \p_\zeta [u^b_{k,3}(u^b_1 + u^0_1 + \zeta u^{(1)}_0)_i
      		+ u^b_{k,i}(u^b_1 + u^0_1 + \zeta u^{(1)}_0)_3] \\
      		&+\p_\zeta [(u^b_{0}+u^0_0)_3 u^b_{k+1,i} + u^b_{k+1,3} (u^0_{0} + u^b_0)_i] + \p_\zeta[u^0_{k+1,3} u^b_{0,i} + u^0_{k+1,i} u^b_{0,3}] \\
          &- \delta_{k1} \p_\zeta (u^b_{1,3}u^b_{1,i})+ \int_{\R^3} \A_{3i} \p_\zeta I_{k-1} \d v\,.
    \end{align*}
\underline{\bf Equation for $u^b_{k,i}$ with $i=1,2$.} By \eqref{viscous_macro_eqn} and the above two identities,
		\begin{align*}
      \partial_tu^b_{k,i} &+\partial_i(\rho^b_{k+2}+\theta^b_{k+2}) +\sum_{j=1}^{2} \p_j [u^b_{k,i} (u^b_{0}+ u^0_0)_j + u^b_{k,j}(u^b_{0} + u^0_0)_i - \delta_{ij} \tfrac{2}{3}u^b_{k} \cdot (u^0_0 + u^b_0) ] \\
      		&-\kappa_1 \p^2_\zeta u^b_{k,i} + \p_\zeta  [ u^b_{k,i}(u^b_1 + u^0_1 + \zeta u^{(1)}_0)_3 + u^b_{k,3}(u^b_1 + u^0_1 + \zeta u^{(1)}_0)_i] \\
      		&+\p_\zeta [u^b_{k+1,3}(u^b_0 + u^0_0)_i + u^0_{k+1,3} u^b_{0,i}] - \delta_{1k} \p_\zeta (u^b_{1,3}u^b_{1,i}) \\
          & + \sum_{j=1}^{2}\int_{\R^3} \A_{ij} \p_j J_{k-1} \d v + \int_{\R^3} \A_{i3} \p_\zeta I_{k-1} \d v  =0\,.
    \end{align*}
By a simple cancellation, using \eqref{pb_k+2} and $\eqref{viscous_macro_eqn}_1$,
\begin{equation}\label{f_b_k-1}
	\begin{aligned}
		\partial_tu^b_{k,i} &+ \sum_{j=1}^{2} (u^b_{0}+ u^0_0)_j \p_j u^b_{k,i} + u^b_{k+1,3} \p_\zeta (u^b_{0} + u^0_0)_i - \kappa_1 \p^2_\zeta u^b_{k,i} +  \sum_{j=1}^{2} u^b_{k,j} \p_j(u^b_0 +u^0_{0})_i \\
    &+ \sum_{j=1}^{2} u^b_{k,i} \p_j (u^b_{0}+ u^0_0)_j + \p_\zeta [u^b_{k,i} u^b_{1,3} + u^0_{k+1,3}u^b_{0,i}]+ \p_\zeta  [ u^b_{k,i}( u^0_1 + \zeta u^{(1)}_0)_3] \\
    =&\p_t \rho^b_{k-2}(u^b_0 + u^0_0)_i - \p_i H^b_{k-1} - \p_\zeta[u^b_{k,3}( u^b_1 + u^0_1 + \zeta u^{(1)}_0)_i] \\
    &+ \delta_{1k} \p_\zeta (u^b_{1,3}u^b_{1,i}) - \sum_{j=1}^{2}\int_{\R^3} \A_{ij} \p_j J_{k-1} \d v - \int_{\R^3} \A_{i3} \p_\zeta I_{k-1} \d v \\
    : =& \mathrm{f}^b_{k-1,i}\,,
	\end{aligned}
\end{equation}
which further implies 
		\begin{align*}
      \partial_tu^b_{k,i} & + \sum_{j=1}^{2} (u^b_{0}+ u^0_0)_j \p_j u^b_{k,i} + u^b_{k+1,3} \p_\zeta (u^b_{0} + u^0_0)_i - \kappa_1 \p^2_\zeta u^b_{k,i} +  \sum_{j=1}^{2} u^b_{k,j} \p_j(u^b_0 +u^0_{0})_i\\
          & + \sum_{j=1}^{2} u^b_{k,i} \p_j u^0_{0,j} + \p_\zeta u^b_{k,i} u^b_{1,3} + \p_\zeta [ u^0_{k+1,3}u^b_{0,i} + u^b_{k,i}( u^0_1 + \zeta u^{(1)}_0)_3] =\mathrm{f}^b_{k-1,i} \,,
    \end{align*} 
    where we used the fact that $ \sum_{j=1}^{2}\p_j u^b_{0,j} + \p_\zeta u^b_{1,3} = 0$.
\begin{remark}
 The term $\mathrm{f}^b_{k-1,i}$ is treated as a source term that is independent of $u^b_{k,i} \ (i=1,2)$. In particular, when $k = 1$, it is clear that the terms in $\mathrm{f}^b_{k-1,i}$ that depend on $u^b_{1,i} \ (i=1,2)$ vanish.
\end{remark}
Next, we derive the equation for $u^b_{k,3}$. As in the case $k = 0$, this in fact yields the equation for $p^b_{k+3}$, since $u^b_{k,3}$ has already been determined together with $(u^b_{k,i})$ for $i = 1,2$.\\
\underline{\em Calculation of $\sum_{j=1}^{2}\int_{\mathbb{R}^3} \A_{ij} \P^\perp \partial_j g^b_{k+2} \d v$ with $i=3$.} By \eqref{def_gb_k+2} and \eqref{A_ij_M},
\begin{equation*}
	\begin{aligned}
		\sum_{j=1}^{2}\int_{\mathbb{R}^3} \A_{3j} \P^\perp \partial_j g^b_{k+2} \d v 
		= &\sum_{j=1}^{2} \p_j [u^b_{k,3}(u^0_0 + u^b_0)_j +u^b_{k,j}(u^0_0 +  u^b_0)_3] + \sum_{j=1}^{2} \int_{\R^3} \A_{3j} \p_j J_{k-1} \d v\\
		= &\sum_{j=1}^{2} \p_j [u^b_{k,3}(u^0_0 + u^b_0)_i] + \sum_{j=1}^{2} \int_{\R^3} \A_{3j} \p_j J_{k-1} \d v\,.
	\end{aligned}
\end{equation*}
\underline{\em Calculation of $\int_{\mathbb{R}^3}A_{i3} \P^\perp \partial_\zeta g^b_{k+3} \d v$ with $i=3$.} By \eqref{def_gb_k+3}, \eqref{A_i3_M} and \eqref{hat_A_ij_A_kl},
		\begin{align*}
      &\int_{\mathbb{R}^3} \A_{33} \P^\perp \partial_\zeta g^b_{k+3} \d v
      		= - \p^2_\zeta \int_{\R^3} \A_{33} \sum_{j=1}^{3}\hat{ \A}_{3j} u^b_{k,j} \d v \\
      		&+ \p_\zeta [2 u^b_{k,3}(u^b_1 + u^0_1 + \zeta u^{(1)}_0)_3 - \tfrac{2}{3} u^b_{k} \cdot (u^b_1 + u^0_1 + \zeta u^{(1)}_0)] 
          +\p_\zeta [2 u^0_{k+1,3} u^b_{0,3} - \tfrac{2}{3} u^0_{k+1} \cdot u^b_0] \\
          &+\p_\zeta [2 u^b_{k+1,3} (u^b_{0} + u^0_0)_3 -\tfrac{2}{3} u^b_{k+1} \cdot(u^b_0 + u^0_0)] -\delta_{1k}\p_\zeta [(u^b_{1,3})^2 -\tfrac{1}{3}|u^b_1|^2] + \int_{\R^3} \A_{33} \p_\zeta I_{k-1} \d v\\
      		=& -\tfrac{4}{3} \kappa_1 \p^2_\zeta u^b_{k,3} + \p_\zeta [2 u^b_{k,3}(u^b_1 + u^0_1 + \zeta u^{(1)}_0)_3 - \tfrac{2}{3} u^b_{k} \cdot (u^b_1 + u^0_1 + \zeta u^{(1)}_0)] \\
      		&- \tfrac{2}{3} \p_\zeta [u^b_{k+1} \cdot(u^b_0 + u^0_0) + u^0_{k+1} \cdot u^b_0] -\delta_{1k}\p_\zeta [(u^b_{1,3})^2 -\tfrac{1}{3}|u^b_1|^2] + \int_{\R^3} \A_{33} \p_\zeta I_{k-1} \d v\,.
    \end{align*}
\underline{\bf Equation for $u^b_{k,3}$.}
		\begin{align*}
      &\partial_tu^b_{k,3}+\partial_\zeta(\rho^b_{k+3}+\theta^b_{k+3}) + \sum_{j=1}^{2} \p_j [u^b_{k,3} (u^b_0 + u^0_0)_j] - \tfrac{4}{3} \kappa_1 \p^2_\zeta u^b_{k,3}\\ 
          &- \tfrac{2}{3} \p_\zeta [u^b_{k+1}(u^b_0 + u^0_0) + u^0_{k+1} \cdot u^b_0] + \p_\zeta  [ 2 u^b_{k,3} (u^b_1 + u^0_1 + \zeta u^{(1)}_0)_3- \tfrac{2}{3} u^b_{k} \cdot (u^b_1 + u^0_1 + \zeta u^{(1)}_0)] \\
          & -\delta_{1k}\p_\zeta [(u^b_{1,3})^2 -\tfrac{1}{3}|u^b_1|^2] 
           + \sum_{j=1}^{2} \int_{\R^3} \A_{3j} \p_j J_{k-1} \d v + \int_{\R^3} \A_{33} \p_\zeta I_{k-1} \d v=0\,.
    \end{align*}
\begin{remark}
  This gives the system for pressure $p^b_{k+2}$.
 \begin{equation}\label{W_b_k-1}
   \begin{aligned}
    &\partial_\zeta(\rho^b_{k+2}+\theta^b_{k+2}) - \tfrac{2}{3} \p_\zeta [u^b_{k} \cdot(u^b_0 + u^0_0) ]\\
    =& -\partial_tu^b_{k-1,3} - \sum_{j=1}^{2} \p_j [u^b_{k-1,3} (u^b_0 + u^0_0)_j] + \tfrac{4}{3} \kappa_1 \p^2_\zeta u^b_{k-1,3}+ \tfrac{2}{3} \p_\zeta [ u^0_{k} \cdot u^b_0]\\
  		& - \p_\zeta  [ 2 u^b_{k-1,3} (u^b_1 + u^0_1 + \zeta u^{(1)}_0)_3- \tfrac{2}{3} u^b_{k-1} \cdot (u^b_1 + u^0_1 + \zeta u^{(1)}_0)] \\
      &+\delta_{1\{k-1\}}\p_\zeta [(u^b_{1,3})^2 -\tfrac{1}{3}|u^b_1|^2]   - \sum_{j=1}^{2} \int_{\R^3} \A_{3j} \p_j J_{k-2} - \int_{\R^3} \A_{33} \p_\zeta I_{k-2} \d v:=W^b_{k-1}\,.
   \end{aligned}
 \end{equation}
 Together with the far-field condition \eqref{viscous_far_field}, we have
 \begin{equation}\label{pb_k+2}
  \Big((\rho^b_{k+2}+\theta^b_{k+2}) - \tfrac{2}{3} [u^b_{k} \cdot(u^b_0 + u^0_0) ]\Big)(t,\bar{x},\zeta) = - \int_{\zeta}^{\infty} W^b_{k-1}(t,\bar{x},\zeta^\prime ) \d \zeta^\prime  := H^b_{k-1}\,.
 \end{equation}
 Once we have solved the system for $\bar{u}^b_k$, we can obtain $p^b_{k+2}$ by \eqref{pb_k+2}.
\end{remark}

Next, we derive the equation for $\theta^b_k$.\\
\underline{\em Calculation of $\sum_{j=1}^{2}\int_{\mathbb{R}^3} \B_{j} \P^\perp \partial_j g^b_{k+2} \d v$.}
\begin{equation}
	\sum_{j=1}^{2}\int_{\mathbb{R}^3} \B_{j} \P^\perp \partial_j g^b_{k+2} \d v = \tfrac{5}{2} \sum_{j=1}^{2} \p_j [u^b_{k,j}( \theta^b_0 + \theta^0_0) + \theta^b_{k}( u^b_0 + u^0_0)_j ] + \sum_{j=1}^{2} \int_{\R^3} \B_j \p_j J_{k-1} \d v \,.
\end{equation}
\underline{\em Calculation of $\int_{\mathbb{R}^3} \B_3 \P^\perp \partial_\zeta g^b_{k+3} \d v$.} 
	\begin{align*}
		\int_{\mathbb{R}^3} \B_3 \P^\perp \partial_\zeta g^b_{k+3} \d v
		=& - \tfrac{5}{2} \kappa_2 \p^2_\zeta \theta^b_k + \tfrac{5}{2}\p_\zeta [u^b_{k,3}(\theta^b_1 + \theta^0_1 + \zeta \theta^{(1)}_0) + \theta^b_{k}(u^b_1 + u^0_1 + \zeta u^{(1)}_0)_3] \\
		 &+ \tfrac{5}{2}\p_\zeta [u^b_{k+1,3}( \theta^b_0 + \theta^0_0) + \theta^b_{k+1} (u^b_0 + u^0_0)_3] + \tfrac{5}{2}\p_\zeta [u^0_{k+1,3} \theta^b_0  + \theta^0_{k+1} u^b_{0,3}] \\
     & - \tfrac{5}{2}\delta_{k1} \p_\zeta[u^b_{1,3} \theta^b_1] + \int_{\R^3} \B_3 \p_\zeta I_{k-1} \d v \,.
	\end{align*} 
\underline{\bf Equation for $\theta^b_k$.} As a result, we derive
		\begin{align*}
      \tfrac{3}{2}\p_t \theta^b_k - \p_t \rho^b_k & +  \tfrac{5}{2} \sum_{j=1}^{2} \p_j [u^b_{k,j}( \theta^b_0 + \theta^0_0) + \theta^b_{k}( u^b_0 + u^0_0)_j ]  - \tfrac{5}{2} \kappa_2 \p^2_\zeta \theta^b_k \\
         & + \tfrac{5}{2} \p_\zeta [u^b_{k,3}( \theta^b_1 + \theta^0_1 + \zeta \theta^{(1)}_0) + \theta^b_{k} (u^b_1 + u^0_1 + \zeta u^{(1)}_0)_3] \\
      		& + \tfrac{5}{2}\p_\zeta [u^b_{k+1,3} (\theta^b_0 + \theta^0_0) + u^0_{k+1,3}\theta^b_0] - \tfrac{5}{2}\delta_{k1} \p_\zeta[u^b_{1,3} \theta^b_1]\\
          & + \sum_{j=1}^{2} \int_{\R^3}\B_j \p_j J_{k-1} \d v + \int_{\R^3} \B_3 \p_\zeta I_{k-1} \d v  = 0 \,.
    \end{align*}
By a simple cancellation, and using \eqref{pb_k},
\begin{equation}\label{g_b_k}
\begin{aligned}
  	 &\p_t \theta^b_k - \kappa_2 \p^2_\zeta \theta^b_k +  \sum_{j=1}^{2} \p_j  [ (u^b_{0}+ u^0_0)_j \theta^b_k ] + \p_\zeta [\theta^b_{k}(u^b_1 + u^0_1 + \zeta u^{(1)}_0)_3] \\
     = &\tfrac{2}{5} \p_t p^b_k  -\sum_{j=1}^{2} \p_j [u^b_{k,j}( \theta^b_0 + \theta^0_0) ] - \p_\zeta [u^b_{k,3}( \theta^b_1 + \theta^0_1 + \zeta \theta^{(1)}_0)]
     - \p_\zeta [u^b_{k+1,3} (\theta^b_0 + \theta^0_0) + u^0_{k+1,3}\theta^b_0] \\
     &+ \delta_{k1} \p_\zeta[u^b_{1,3} \theta^b_1] - \sum_{j=1}^{2} \tfrac{2}{5} \int_{\R^3}\B_j \p_j J_{k-1} \d v - \tfrac{2}{5} \int_{\R^3} \B_3 \p_\zeta I_{k-1} \d v  \\
     :=& \mathrm{g}^b_{k}\,,
\end{aligned}
\end{equation}
where $\mathrm{g}^b_{k}$ represents known terms when solving $\theta^b_k$. {\bf This is because the equation of $\bar{u}^b_k$ is independent of $\theta^b_k$, allowing us to first solve for $u^b_k$, then determine $\theta^b_k$.}

As a consequence, we obtain the following linear incompressible Prandtl equation:
\begin{equation}\label{linear_incompressible_Prandtl}
  \begin{cases}
    \sum_{j=1}^{2}\p_j u^b_{k,j} + \p_\zeta u^b_{k+1,3} = -\p_t \rho_{k-2}^b\,,\\
    \partial_tu^b_{k,i}  + \sum_{j=1}^{2} (u^b_{0}+ u^0_0)_j \p_j u^b_{k,i} + u^b_{k+1,3} \p_\zeta (u^b_{0} + u^0_0)_i - \kappa_1 \p^2_\zeta u^b_{k,i} +  \sum_{j=1}^{2} u^b_{k,j} \p_j(u^b_0 +u^0_{0})_i\\
    \qquad\quad + \sum_{j=1}^{2} u^b_{k,i} \p_j u^0_{0,j} + \p_\zeta u^b_{k,i} u^b_{1,3} + \p_\zeta [ u^0_{k+1,3}u^b_{0,i} + u^b_{k,i}( u^0_1 + \zeta u^{(1)}_0)_3] =\mathrm{f}^b_{k-1,i} \ (i=1,2)\,,\\
    \p_t \theta^b_k - \kappa_2 \p^2_\zeta \theta^b_k +  \sum_{j=1}^{2} \p_j  [ (u^b_{0}+ u^0_0)_j \theta^b_k ] + \p_\zeta [\theta^b_{k}(u^b_1 + u^0_1 + \zeta u^{(1)}_0)_3] = \mathrm{g}^b_{k}\,,\\
    \lim_{\zeta \to \infty}( \bar{u}^b_{k}, u^b_{k+1,3}, \theta^b_k )(t,\bar{x},\zeta)=0\,.
  \end{cases}
\end{equation}
\end{proof}
We emphasize that $W^b_{k-1}$, $H^b_{k-1}$, $J_{k-1}$, $I_{k-1}$, $\mathrm{f}^b_{k-1,i}$ and $\mathrm{g}^b_{k-1}$ depend on $g^b_j\ (1 \le j \le k-1)$ and $g_j \ (1 \le j \le k)$. Moreover, we use the notation $f_k =0$ whenever the subscript $k <0$. Finally, the initial conditions of \eqref{nonlinear_Prandtl} and \eqref{linear_incompressible_Prandtl} are imposed on
\begin{equation}\label{IC_Prandtl}
\begin{aligned}
(\bar{u}^b_0, \theta^b_k) (0, \bar{x}, \zeta) = (\bar{u}^{b, in}_0 , \theta^{b, in}_0 ) ( \zeta)  \,,\quad  (\bar{u}^b_k, \theta^b_k) (0, \bar{x}, \zeta) = (\bar{u}^{b, in}_k , \theta^{b, in}_k ) (\bar{x}, \zeta)  \,, \quad k = 1,2, \cdots
\end{aligned}
\end{equation}
with $\lim_{\zeta \to \infty} (\bar{u}^{b, in}_k , \theta^{b, in}_k) (\bar{x}, \zeta) = 0$ for $k =0,1,2\cdots$. This implies that the leading-order term corresponds a shear flow.

\subsection{Knudsen boundary layer and boundary conditions}
Since the sum of the Hilbert and viscous boundary layer solutions is a infinitesimal Maxwellian at $O(1)$ and $O(\sqrt{\eps})$. It matches the Maxwell reflection boundary condition if and only if 
\begin{equation}\label{BC_u_01_ub_01}
  u^0_j + u^{b,0}_j=0\,, \quad  \theta^0_j + \theta^{b,0}_j=0\,, \quad \text{for}\quad \! j=0,1\,.
\end{equation}
This gives the boundary conditions to the macroscopic part of $g_0,g_1,g^b_0,$ and $g^b_1$. Recalling \eqref{pb0_ub03}, we obtain the boundary condition for the incompressible Euler system \eqref{incompressible_Euler_0}
\begin{equation}\label{BC_u_03}
	u^0_{0,3} =0\,.
\end{equation}
Once the system \eqref{incompressible_Euler_0} is solved, we obtain the boundary conditions for the nonlinear Prandtl layer problem \eqref{nonlinear_Prandtl}:
\begin{equation}\label{BC_ub_thetab_0}
	u^{b,0}_{0,i} = - u^0_{0,i} \ (i=1,2)\,, \quad
	\theta^{b,0}_0 = - \theta^0_{0}\,.
\end{equation}
While solving the Prandtl layer problem \eqref{nonlinear_Prandtl}, we can obtain the boundary value of $u^b_{1,3}$, and therefore we obtain the boundary condition for $u^0_{1,3}$:
\begin{equation}\label{BC_u_13}
	 u^0_{1,3}= - u^{b,0}_{1,3} \,.
\end{equation}
Then, the linear incompressible Euler system \eqref{linear_Euler} with $k=1$ can be solved. Consequently, the boundary conditions for the linear incompressible Prandtl system \eqref{linear_incompressible_Prandtl} with $k=1$ are derived from \eqref{BC_u_01_ub_01}. To solve it, we still need to know the relation between $u^b_{2,3}$ and $u^0_{2,3}$. This can be done by the analysis of the Knudsen layer, see \eqref{BC_u2_ub_2}.
 
Note that the solution $g_\eps + g^b _\eps$ does not match the boundary condition at $O(\sqrt{\eps}^k)$ for $(k \geq 2)$, and thus the Knudsen layer correction $g^{bb}_\eps$ is required. Inside the Knudsen layer, the new scaled normal coordinate is introduced (as mentioned earlier, the thickness of the Knudsen layer is $O(\eps^2)$):
\begin{equation*}
  \begin{aligned}
    \xi = \tfrac{x_3}{\eps^2} \,.
  \end{aligned}
\end{equation*}
Then the Knudsen boundary expansion is defined as
\begin{equation*}
  \begin{aligned}
    g^{bb}_\eps (t, \bar{x}, \xi, v) \thicksim \sum_{k \geq 2} \sqrt{\eps}^k g^{bb}_k (t, \bar{x}, \xi, v) \,.
  \end{aligned}
\end{equation*}
We also assume the following far-field condition 
\begin{equation}\label{far_condition_Knudsen}
	g^{bb}_k \to 0 \quad \text{as} \quad \! \xi\to \infty \,.
\end{equation}
From plugging $g_\eps + g^b_\eps + g^{bb}_\eps$ into \eqref{BE_g_eps}, we have
\begin{equation}\label{Order_Anal_Knudsen}
\begin{aligned}
\sqrt{\eps}^{-2} :\qquad & v_3 \p_\xi g^{bb}_2 + \L g^{bb}_2 = 0 \,,\\
\sqrt{\eps}^{-1} : \qquad & v_3 \p_\xi g^{bb}_3 + \L g^{bb}_3 = 0 \,,\\
......&\\
\sqrt{\eps}^k :\qquad &  v_3 \p_\xi g^{bb}_{k+4} + \L g^{bb}_{k+4} \\
= & - \p_t g^{bb}_{k-2} - \bar{v} \cdot \nabla_{\bar{x}} g^{bb}_k \\
& + \sum_{\substack{i+j=k+2\,,\\ i, j\ge 0}} \big[ \Gamma (g^0_i + g^{b,0}_i, g^{bb}_j) + \Gamma (g^{bb}_j, g^0_i + g^{b,0}_i) + \Gamma (g^{bb}_i, g^{bb}_j) \big] \\
& + \sum_{\substack{i+4l+j=k+2\,,\\1\le l \le N, i,j\ge 0}} \frac{\xi^l}{l!} \big[ \Gamma (g^{(l)}_i, g^{bb}_j) + \Gamma (g^{bb}_j, g^{(l)}_i)\big] \\
& + \sum_{\substack{i+3l+j=k+2\,,\\1\le l \le N, i,j\ge 0}} \frac{\xi^l}{l!} \big[ \Gamma (g^{b, (l)}_i, g^{bb}_j) + \Gamma (g^{bb}_j, g^{b, (l)}_i)\big]\,,
\end{aligned}
\end{equation}
where the Taylor expansion of $g^b_i$ at $\zeta = 0$ is utilized:
\begin{equation*}
  \begin{aligned}
    g_i^b = g_i^{b,0} + \sum_{1 \leq l \leq N} \tfrac{\xi^l}{l!} g_i^{b, (l)} + \tfrac{\xi^{N+1}}{(N+1)!} \widetilde{g}_i^{b, (N+1)} \,.
  \end{aligned}
\end{equation*}
We used the notation that $g^{bb}_k =0 $ for $k \leq 2$. To solve the Knudsen layer problems and obtain the boundary conditions for the fluid equations, we first consider $k=2$. Recalling \eqref{def_g2} and \eqref{def_gb2}, and applying \eqref{BC_u_01_ub_01}, we derive the system for $g^{bb}_2$
\begin{equation}
	\begin{cases}
		v_3 \p_\xi g^{bb}_2 + \L g^{bb}_2 =0 \,,\\
		(L^R- \alpha L^D) g^{bb}_2 = - (L^R - \alpha L^D) (\rho^0_2 + \rho^{b,0}_2 +( u^0_2 + u^{b,0}_2 )\cdot v +(\theta^0_2 + \theta^{b,0}_2) (\tfrac{|v|^2-3}{2}) ) \sqrt{ \mu}\,,\\
		\lim_{\xi \to \infty} g^{bb}_2 =0\,.
	\end{cases}
\end{equation}
This system is solvable only if 
\begin{equation}\label{BC_u2_ub_2}
		u^0_2 + u^{b,0}_2 =0 \,,
		\quad \theta^0_2 + \theta^{b,0}_2 =0 \,.
\end{equation} This gives the boundary conditions for the fluid variables of $g_2$ and $g^{b}_2$\,. 

Next, we consider $g^{bb}_3$. Recalling \eqref{def_g3} and \eqref{def_gb3}, and applying \eqref{BC_u_01_ub_01}, we derive the system for $g^{bb}_3$:
\begin{equation}\label{eqn_gbb_3}
	\begin{cases}
		v_3 \p_\xi g^{bb}_3 + \L g^{bb}_3 =0 \,,\\
		(L^R - \alpha L^D) g^{bb}_3 = - (L^R - \alpha L^D) (\rho^0_3 + \rho^{b,0}_3 + u^0_3 + u^{b,0}_3 \cdot v +(\theta^0_3 + \theta^{b,0}_3) (\tfrac{|v|^2-3}{2}) ) \sqrt{ \mu}\,,\\
		\qquad \qquad \qquad \qquad + (L^R - \alpha L^D) \hat{\B}_3 \p_\zeta \theta^{b,0}_0 + (L^R - \alpha L^D) \sum_{j=1}^{2} \hat{\A}_{3j} \p_\zeta u^{b,0}_{0,j} \,,\\
		\lim_{\xi \to \infty} g^{bb}_3 =0\,.
	\end{cases}
\end{equation}
The above system, together with the far-field condition $\lim_{\xi \to \infty} g^{bb}_3 = 0$, is overdetermined; see \cite{Bardos-Caflisch-Nicolaenko-1986-CPAM,Coron-Golse-Sulem-1988-CPAM,Golse-Perthame-Sulem-1988-ARMA,Jiang-Luo-Wu-arXiv}. There are four solvability conditions associated with this system, which yield the boundary conditions for the fluid equations. From the isotropic properties of $\L$ and $L^R- \alpha L^D$ (see \cite{Aoki-2017-JSP,He-Jiang-Wu-2024,Jiang-Wu-2025-KRM}), the above system is solvable if and only if 
\begin{equation}
	\begin{cases}
		u^0_{3,i} + u^{b,0}_{3,i} = b_1 \p_\zeta u^{b,0}_{0,i} \quad (i=1,2) \,, \\
		u^{0}_{3,3} + u^{b,0}_{3,3} = 0 \,,\\
		\theta^0_3 + \theta^{b,0}_3 = c_1 \p_\zeta \theta^{b,0}_0 \,,
	\end{cases}
\end{equation}
 where $b_1$ and $c_1$ are the so-called slip coefficients, uniquely determined together with $g^{bb}_3$. This gives the boundary conditions to the linear incompressible Euler system and the Prandtl-type system. Moreover, we know that the boundary conditions of the Prandtl system are of Dirichlet type. 
 
Furthermore, the boundary conditions for $u_k + u^b_k$ and $\theta_k + \theta^b_k$ for $k \geq 4$ can be determined inductively in this manner. Combined with the equations derived above, these boundary conditions allow us to solve for $g_k$, $g^b_k$, and $g^{bb}_k$ step by step.

\subsection{Truncations of the Hilbert expansion}\label{subsec_Truncations}

Our goal is to prove the incompressible Euler limit from the scaled Boltzmann equation by the above Hilbert-type expansion. The key point is to prove that the remainders of expansion will go to zero as the Knudsen number $\eps \to 0$. Mathematically, this means that we seek a special class of solutions to the scaled Boltzmann equation in the regime of sufficiently small Knudsen number $\eps$. The more terms are expanded, the more special the solutions are. We hope that the terms in the expansion are as less as possible.
\begin{equation}\label{Hilbert_truncation}
  \begin{aligned}
    g_\eps = \sum_{k=0}^{13} \sqrt{\eps}^k \big\{  g_{k}(t,x,v) + g^b_k(t, \bar{x},\tfrac{x_3}{\sqrt{\eps}}, v) + g^{bb}_k(t, \bar{x}, \frac{x_3}{\eps^2}, v) \big\} \\
    +\sqrt{\eps}^{8}\tilde{g}_{R,\eps}(t,x,v) \geq 0\,.
  \end{aligned}
\end{equation}
For the above ansatz, we take $N = 9$ in the Taylor expansions before. From plugging \eqref{Hilbert_truncation} into \eqref{BE_g_eps} with the boundary condition \eqref{MBC_g_eps}, we obtain the remainder equation 
\begin{align*}
    \p_t \tilde{g}_{R, \eps} + \tfrac{1}{\eps} v \cdot \nabla_x \tilde{g}_{R, \eps} + \tfrac{1}{\eps^3} \L \tilde{g}_{R,\eps} = \eps^{2} \Gamma(\tilde{g}_{R,\eps}, \tilde{g}_{R,\eps}) 
    + R_\eps + R^b_{\eps} + R^{bb}_\eps \\
    + \eps^{\frac{i-4}{2}} \sum_{i=0}^{13}  [\Gamma(\tilde{g}_{R,\eps}, g_i + g^b_{i} + g^{bb}_i) + \Gamma( g_i + g^b_{i} + g^{bb}_i, \tilde{g}_{R,\eps})]\,, 
\end{align*} 
with Maxwell reflection boundary condition 
\begin{equation*}
  \gamma_- \tilde{g}_{R,\eps} (t,x,v) = (1 - \alpha) \tilde{g}_{R,\eps} (t,x,R_x  v)+ \sqrt{2 \pi \mu} \int_{v^\prime \cdot n>0} (v^\prime \cdot n) \sqrt{\mu(v^\prime)} \tilde{g}_{R,\eps}  (t,x,v^\prime)\d v^\prime\,,
\end{equation*}
where 
\begin{equation}\label{def_R}
  R_\eps = - \eps^{\frac{i - 8}{2}} \sum_{i= 8}^{13} \p_t g_i - \eps^{\frac{i- 10}{2}} \sum_{ i = 10}^{ 13} v \cdot \nabla_x g_i  + \eps^{\frac{i+j - 12}{2}} \sum_{\substack{i+j \geq 12 \\
  1 \leq i, j \leq 13}}\Gamma (g_i, g_j) \,,
\end{equation}
\begin{equation}\label{def_R_b}
  \begin{aligned}
    R^b_{\eps} = &- \eps^{\frac{i - 8}{2}} \sum_{i= 8}^{13} \p_t g^b_i - \eps^{\frac{i  - 10 }{2}} \sum_{ i = 10}^{ 13} \bar{v} \cdot \nabla_{\bar{x}} g^b_i - \eps^{\frac{i-11}{2}} \sum_{ i = 11}^{ 13} v_3 \p_\zeta g^b_i  + \eps^{\frac{i+j- 12}{2}} \sum_{\substack{i+j \geq 12 \\
    1 \leq i, j \leq 13}}\Gamma (g^b_i, g^b_j)\\
    &+ \eps^{\frac{i+j- 12}{2}} \sum_{\substack{i+j \geq 12 \\
    1 \leq i, j \leq 13}}[\Gamma (g^b_i, g^0_j) + \Gamma (g^0_j, g^b_i)] \\
   & +\eps^{\frac{i+j+l - 12}{2}} \sum_{\substack{i+j+l \geq 12 \\
    1 \leq i, j \leq 13, 1 \leq l \leq 9}} \frac{\zeta^l}{ l !} [\Gamma (g^{(l)}_i, g^b_j) + \Gamma (g^b_j, g^{(l)}_i )]\\
    &+\eps^{\frac{i + j - 2}{2}} \sum_{1 \leq i, j \leq 13} \frac{\zeta^{10}}{ 10 !} [\Gamma (g^{(10)}_i, g^b_j) + \Gamma (g_j, g^{(10)}_i )]\,,
  \end{aligned}
\end{equation}
and
\begin{equation}\label{def_R_bb}
  \begin{aligned}
    R^{bb}_{\eps} = & - \eps^{\frac{i - 8}{2}} \sum_{i= 8}^{13} \p_t g^{bb}_i - \eps^{\frac{i - 10}{2}} \sum_{ i = 10}^{13} \bar{v} \cdot \nabla_{\bar{x}} g^{bb}_i   + \eps^{\frac{i+j -12}{2}} \sum_{\substack{i+j \geq 12 \\
    1 \leq i, j \leq 13}}\Gamma (g^{bb}_i, g^{bb}_j)\\
    &+ \eps^{\frac{i+j- 12}{2}} \sum_{\substack{i+j \geq 12 \\
    1 \leq i, j \leq 13}}[\Gamma (g^{bb}_i, g^0_j+ g^{b,0}_j) + \Gamma (g^0_j + g^{b,0}_j, g^{bb}_i)]  \\
    &+\eps^{\frac{i+j+ 4l - 12}{2}} \sum_{\substack{i+j+ 4l \geq 12 \\
    1 \leq i, j \leq 13, 1 \leq l \leq 9}} \frac{\xi^l}{ l !} [\Gamma (g^{(l)}_i, g^{bb}_j) + \Gamma (g^{bb}_j, g^{(l)}_i )]\\
    &+\eps^{\frac{i+j+ 3l - 12}{2}} \sum_{\substack{i+j+ 3l \geq 12 \\
    1 \leq i, j \leq 13, 1 \leq l \leq 9}} \frac{\xi^l}{ l !} [\Gamma (g^{b,(l)}_i, g^{bb}_j) + \Gamma (g^{bb}_j, g^{b,(l)}_i )]\\
    & +\eps^{\frac{i + j + 28}{2}} \sum_{1 \leq i, j \leq 13} \frac{\xi^{10}}{ 10 !} [\Gamma (g^{(10)}_i, g^{bb}_j) + \Gamma (g^{bb}_j, g^{(10)}_i )]\\
    & +\eps^{\frac{i + j + 18}{2}} \sum_{1 \leq i, j \leq 13} \frac{\xi^{10}}{ 10 !} [\Gamma (g^{b,(10)}_i, g^{bb}_j) + \Gamma (g^{bb}_j, g^{b,(10)}_i )]\,.
  \end{aligned}
\end{equation}
where we used the fact that $x_3 = \sqrt{\eps} \zeta = \eps^2 \xi$.

 As illustrated in subsection \ref{subsec_idea}, we need to deal with the singular term \(\tfrac{1}{\eps^2}[\Gamma(\tilde{g}_{R,\eps},g_0 + g^{bb}_0) + \Gamma(g_0 + g^{bb}_0,\tilde{g}_{R,\eps}) ]\). We introduce $(\rho_\eps, u_\eps, \theta_\eps) = (1 + \eps(\rho_0 + \rho^b_0), \eps (u_0 + u^b_0), 1 + \eps (\theta_0 + \theta^b_0) )$ and the local Maxwellian $\mu_\eps$ given in \eqref{def_mu_eps}, where $(\rho_0, u_0 ,\theta_0)$ solves the incompressible Euler system \eqref{incompressible_Euler_0} and $(\rho^b_0, u^b_0 ,\theta^b_0)$ solves the incompressible Prandtl equation \eqref{nonlinear_Prandtl}. A Taylor expansion with respect to $\eps$ yields
\begin{equation*}
  \mu_\eps = \mu + \eps (g_0 + g^b_0) \sqrt{\mu} + \eps^2 r,
\end{equation*}
where $r = O(1)$ is the second order remainder.

Define $g_{R,\eps} = \sqrt{\tfrac{\mu}{\mu_\eps}} \tilde{g}_{R,\eps}$, we derive the equation for $g_{R,\eps}$:
\begin{equation}\label{remainder_eqn}
 \begin{aligned}
   \p_t g_{R,\eps} + \tfrac{1}{\eps} v \cdot \nabla_x g_{R,\eps} + \tfrac{1}{\eps^3} \L_\eps g_{R,\eps} + \tfrac{(\p_t + \frac{1}{\eps} v \cdot \nabla_x ) {\mu_\eps}}{2 \mu_\eps} g_{R,\eps} : = S :=\sum_{i=1}^{4}S_i\,,
 \end{aligned}
\end{equation}
where 
\begin{equation}\label{S_1_S_4}
  \begin{aligned}
    S_1 =& \eps^{2} \Gamma_\eps (g_{R,\eps}, g_{R,\eps})\,,\\
    S_2 =& \eps^{\frac{i-4}{2}} \sum_{i=1}^{13}  [\Gamma_\eps(g_{R,\eps}, \sqrt{\tfrac{\mu}{\mu_\eps}}(g_i + g^b_{i} + g^{bb}_i)) + \Gamma_\eps ( \sqrt{\tfrac{\mu}{\mu_\eps}}(g_i + g^b_{i} + g^{bb}_i), g_{R,\eps})] \,,\\
    S_3 =& - \tfrac{1}{\eps} [\Gamma_\eps(\tfrac{r}{\sqrt{\mu_\eps}},g_{R,\eps}) + \Gamma_\eps(g_{R,\eps},\tfrac{r}{\sqrt{\mu_\eps}})] \,,\\
    S_4 =& \sqrt{\tfrac{\mu}{\mu_\eps}}(R_\eps + R^b_{\eps} + R^{bb}_\eps )\,.
  \end{aligned}
\end{equation}
where $\L_\eps, \Gamma_\eps$ are defined in \eqref{def_L_eps_Gamma_eps}.
As for the boundary condition for $g_{R,\eps}$, we observe from \eqref{BC_u_01_ub_01}, together with the identities $\rho_0 + \theta_0 =0$ and $\rho^b_0 + \theta^b_0 =0$, that $\mu_\eps = \mu$ on the boundary. Therefore, the boundary condition for $g_{R,\eps}$ is given by
\begin{equation}\label{g_R_BC}
    \gamma_- g_{R,\eps} (t,x,v) = (1 - \alpha) g_{R,\eps} (t,x,R_x  v)+ \alpha P_\gamma g_{R,\eps} (t,x,v)\,.
\end{equation}
Here, we recall that $R_x v = (\bar{v}, -v_3)$, and $P_\gamma$ is given in \eqref{def_P_gamma}.

\section{Solutions to the constructed expansions}\label{Sec_solution_expansion}
\subsection{The nonlinear incompressible Euler system \eqref{incompressible_Euler_0}}
\begin{equation}
  \begin{cases}
    \p_t u_0 + u_0 \cdot \nabla_x u_0 + \nabla_x p_0 =0\,,\\
    \p_t \theta_0 + u_0 \cdot \nabla_x \theta_0 = 0\,,\\
    \operatorname{div} u_0 = 0 \,,
  \end{cases}
 \end{equation}
with the boundary condition 
\begin{equation}
  u_{0,3}(t,\bar{x},0)  = 0\,,
\end{equation}
and the initial conditions \eqref{IC_incompressible_Euler}.
\subsection{The linear incompressible Euler system \eqref{linear_Euler}}
\begin{equation}
	\begin{cases}
		\p_t \rho_k + \operatorname{div} u_{k+2} =0\,,\\
		\p_t u_k + \operatorname{div} ( u_k \otimes u_0 + u_0 \otimes u_k) + \nabla_x (\rho_{k+2} + \theta_{k+2}) - \tfrac{2}{3}\nabla_x (u_0 \cdot u_k) = \mathcal{F}_{k-1} \,,\\
		\tfrac{3}{2}\p_t \theta_k + \operatorname{div} u_{k+2} + \tfrac{5}{2} \operatorname{div} (u_k \theta_0 + u_0 \theta_k) = \mathcal{G}_{k-1}\,,
	\end{cases}
 \end{equation}
with the boundary condition 
\begin{equation}
  u_{k,3}(t,\bar{x},0) = d_k(t,\bar{x})\,, 
\end{equation}
where $d_k(t,\bar{x}) \ (k\geq 1)$ are determined by $g_j\ (j\leq k-1)$, $g^b_j\ (j\leq k-1)$ and $u^b_{k,3}$. Furthermore, the initial conditions are given by \eqref{IC_incompressible_Euler}.

\subsection{The nonlinear incompressible Prandtl system \eqref{nonlinear_Prandtl}}
\begin{equation}
	\begin{cases}
		\sum_{j=1}^{2}\p_j u^b_{0,j} + \p_\zeta u^b_{1,3} =0\,,\\
		\partial_tu^b_{0,i} +\sum_{j=1}^{2} u^b_{0,j} \p_j u^b_{0,i} + u^b_{1,3} \p_\zeta u^b_{0,i} - \kappa_1 \p^2_\zeta u^b_{0,i} \\
		\qquad\qquad + \p_j \sum_{j=1}^{2} (u^b_{0,i} u^0_{0,j}) + \sum_{j=1}^{2} u^b_{0,j} \p_j u^0_{0,i} + \p_\zeta  [ u^b_{0,i}( u^0_1 + \zeta u^{(1)}_0)_3] = 0 \quad (i=1,2)\,,\\
		\p_t \theta^b_0 +  \sum_{j=1}^{2}   u^b_{0,j} \p_j \theta^b_0 + u^b_{1,3} \p_\zeta \theta^b_0 -  \kappa_2 \p^2_\zeta \theta^b_0 + \sum_{j=1}^{2} \p_j (u^0_{0,j} \theta^b_0) \\
    \qquad\qquad+ \sum_{j=1}^{2} u^b_{0,j} \p_j \theta^0_0 + \p_\zeta [\theta^b_{0}(u^0_1 + \zeta u^{(1)}_0)_3] =0\,,\\
    \lim_{\zeta \to \infty}( {u}^b_{0,1}, {u}^b_{0,2}, u^b_{1,3},\theta^b_0) =0\,,
	\end{cases}
\end{equation}
with the Dirichlet boundary condition 
\begin{equation}
  u^b_{0,i} = - u^0_{0,i} \ (i=1,2)\,,
  \theta^b_0 = - \theta^0_0\,,
\end{equation}
and the initial conditions \eqref{IC_Prandtl} with $k=0$.
\subsection{The linear incompressible Prandtl system \eqref{linear_incompressible_Prandtl}}
\begin{equation}
  \begin{cases}
    \sum_{j=1}^{2}\p_j u^b_{k,j} + \p_\zeta u^b_{k+1,3} = -\p_t \rho_{k-2}^b\,,\\
    \partial_tu^b_{k,i}  + \sum_{j=1}^{2} (u^b_{0}+ u^0_0)_j \p_j u^b_{k,i} + u^b_{k+1,3} \p_\zeta (u^b_{0} + u^0_0)_i - \kappa_1 \p^2_\zeta u^b_{k,i} +  \sum_{j=1}^{2} u^b_{k,j} \p_j(u^b_0 +u^0_{0})_i\\
    \qquad\quad + \sum_{j=1}^{2} u^b_{k,i} \p_j u^0_{0,j} + \p_\zeta u^b_{k,i} u^b_{1,3} + \p_\zeta [ u^0_{k+1,3}u^b_{0,i} + u^b_{k,i}( u^0_1 + \zeta u^{(1)}_0)_3] =\mathrm{f}^b_{k-1,i} \ (i=1,2)\,,\\
    \p_t \theta^b_k - \kappa_2 \p^2_\zeta \theta^b_k +  \sum_{j=1}^{2} \p_j  [ (u^b_{0}+ u^0_0)_j \theta^b_k ] + \p_\zeta [\theta^b_{k}(u^b_1 + u^0_1 + \zeta u^{(1)}_0)_3] = \mathrm{g}^b_{k}\,,\\
    \lim_{\zeta \to \infty}( {u}^b_{k,1}, {u}^b_{k,2}, u^b_{k+1,3}, \theta^b_k )(t,\bar{x},\zeta)=0\,,
  \end{cases}
\end{equation}
with the Dirichlet boundary condition 
\begin{equation}
  u^b_{k,i} = \u_k\,,
  \theta^b_k = \Theta_k\,,
\end{equation}
where $\u_k,\Theta_k$ are quantities determined by $g^b_j\ (j\leq k-1)$ and $g_j\ (l \leq k)$, and are already known. Furthermore, the initial conditions are given by \eqref{IC_Prandtl}.

\subsection{The Knudsen layer equation} Inside the Knudsen layer, the following system is derived:
\begin{equation}\label{Knudsen_layer_eqn}
  \begin{cases} 
    v_3 \p_\xi g^{bb}_k + \L g^{bb}_k = S^{bb}_k \,,\\
     \gamma_- g^{bb}_k (t,\bar{x},0,v) = (1 - \alpha) g^{bb}_k (t,\bar{x},0,R_x  v)+ \sqrt{2 \pi \mu} \int_{v^\prime \cdot n>0} (v^\prime \cdot n) \sqrt{\mu(v^\prime)} g^{bb}_k  (t,\bar{x},0,v^\prime)\d v^\prime + f^{bb}_k\,,\\
     \lim_{\xi \rightarrow \infty} g^{bb}_k = 0 \,. 
  \end{cases}
\end{equation}
According to \cite{Jiang-Luo-Wu-arXiv}, the solution to the Knudsen layer equation \eqref{Knudsen_layer_eqn} exhibits exponential decay in both the normal variable $\xi$ and the velocity variable $v$, i.e., 
$$\|e^{c_1 \xi + c_2|v|^2} g^{bb}_k \|_\infty \leq C({f^{bb}_k ,S^{bb}_k}) $$
for some $c_1,c_2>0$, where the constant $C({f^{bb}_k ,S^{bb}_k})$ depends on $f^{bb}_k ,S^{bb}_k$.\footnote{ In \cite{Jiang-Luo-Wu-arXiv}, the authors considered the cases where $S^{bb}_k \in \N^\perp$. We can decompose $S^{bb}_k = S^{bb}_{k,1} + S^{bb}_{k,2}$, where $S^{bb}_{k,1} \in \N$ and the corresponding solution can be explicitly constructed. The part with $S^{bb}_{k,2} \in \N^\perp$ then satisfies the assumptions required to apply the result in \cite{Jiang-Luo-Wu-arXiv}. We refer the reader to \cite{Bardos-Caflisch-Nicolaenko-1986-CPAM, Guo-Huang-Wang-2021-ARMA} for more details.} Moreover, tangential and time derivatives of the solution to \eqref{Knudsen_layer_eqn} can be readily obtained due to the linearity and regularity properties of the system.

The following proposition follows directly from the techniques developed in the Appendix of \cite{Wang-Wang-Zhang-ARMA}. For simplicity, we omit the detailed proof.
\begin{proposition}\label{Prop_g_k_g_b_k}
  Let $(\rho^{in}_k , u^{in}_k , \theta^{in}_k)$ and $(\rho^{b,in}_k , u^{b,in}_k , \theta^{b,in}_k) $ satisfy
  \begin{equation*}
  \begin{aligned}
     \sum_{\m \in \mathbb{N}^3_0, \m=(\m_1, \m_2, \m_3)} \tfrac{(2\sigma(0))^{2 |\m|}}{(\m !)^2}  \|\p_1^{\m_1}\p_2^{\m_2}\p_3^{\m_3}(\rho^{in}_k , u^{in}_k , \theta^{in}_k) \|_{L^\infty_x \cap L^2_x} \leq C \,,\\
      \sum_{\m \in \mathbb{N}^3_0, \m=(\m_1, \m_2, \m_3)}\tfrac{(2\sigma(0))^{2 |\m|}}{(\m !)^2}  \|\p_1^{\m_1}\p_2^{\m_2}\p_3^{\m_3}(\rho^{b,in}_k , u^{b,in}_k , \theta^{b,in}_k) \|_{L^\infty_x \cap L^2_x} \leq C \,,
  \end{aligned}
  \end{equation*}
for $1 \leq k \leq 13$. Then there exists $T_0>0 $ such that the corresponding solutions $(\rho_k , u_k , \theta_k) $ and $(\rho^{b}_k , u^{b}_k , \theta^{b}_k)$ exist in the time interval $[0,T_0$. Moreover, these solutions satisfy
    \begin{equation*}
  \begin{aligned}
     \sum_{\m \in \mathbb{N}^3_0, \m=(\m_1, \m_2, \m_3)} \tfrac{(\sigma(0))^{2 |\m|}}{(\m !)^2}  \|\p_1^{\m_1}\p_2^{\m_2}\p_3^{\m_3}(\rho_k , u_k , \theta_k) \|_{L^\infty_x \cap L^2_x} \leq C_0 \,,\\
      \sum_{\m \in \mathbb{N}^3_0, \m=(\m_1, \m_2, \m_3)}\tfrac{(\sigma(0))^{2 |\m|}}{(\m !)^2}  \|\p_1^{\m_1}\p_2^{\m_2}\p_3^{\m_3}(\rho^{b}_k , u^{b}_k , \theta^{b}_k) \|_{L^\infty_x \cap L^2_x} \leq C_0 \,.
  \end{aligned}
  \end{equation*}
As a direct consequence, applying the properties of the inverse operator $\L^{-1}$ established in \cite{JLT-2022-arXiv}, we derive the following estimates:
\begin{equation}\label{R_analytic}
\begin{aligned}
    \sum_{\m \in \mathbb{N}_0^2} \tfrac{\sigma(t)^{2 |\m|}}{(\m !)^2}  \|\mu^{-a} \p^\m (g_k, g^b_k,g^{bb}_k) \|_{L^\infty_{x,v}\cap L^2_{x,v}} \leq C_0  \,,\\
   \sum_{\m \in \mathbb{N}_0^2} \tfrac{\sigma(t)^{2 |\m|}}{(\m !)^2}  \|\mu^{-a} \p^\m (R_\eps , R^b_\eps, R^{bb}_\eps) \|_{L^\infty_{x,v}\cap L^2_{x,v}} \leq C_0\,.
\end{aligned}
\end{equation}
 for some $0<a <\tfrac{1}{2}$. 
\end{proposition}

\section{Uniform estimates for the remainder equation}\label{Sec_remainder_uniform}

This section is devoted to establishing uniform estimates for the remainder term. Due to the loss of derivatives in the equation, our analysis is carried out in an analytic function space.

We first recall the notation:
\begin{equation*}
  \p^\m = \p^{\m_1}_1 \p^{\m_2}_2  \quad \textrm{for} \quad \!  \m = (\m_1, \m_2) \in \mathbb{N}_0^2 \,.
\end{equation*}
Throughout this section, for notational simplicity, we denote $g = g_{R,\eps}$ and $h=h_{R,\eps}$. We define $g_\m$ by $\p^\m g$. We also recall the definitions of the analytic norms introduced in \eqref{def_X_Y}. 

Since our fluid quantities $(\rho_0,u_0,\theta_0)$ and $(\rho^b_0,u^b_0,\theta^b_0)$ are the shear flow, the differential $\partial_1$ and $\p_2$ communicate with $\L_\eps$ and $\Gamma_\eps$. Acting $\p^\m $ on \eqref{remainder_eqn}, we obtain
\begin{equation}\label{remainder_eqn_alpha}
 \begin{aligned}
   \p_t g_\m + \tfrac{1}{\eps} v \cdot \nabla_x g_\m + \tfrac{1}{\eps^3} \L_\eps g_\m + \tfrac{(\p_t + \frac{1}{\eps} v \cdot \nabla_x ) {\mu_\eps}}{2 \mu_\eps} g_\m = S_\m \,,
 \end{aligned}
\end{equation}
subject to the Maxwell reflection boundary condition
\begin{equation}\label{g_m_BC}
    \gamma_- g_{\m} (t,x,v) = (1 - \alpha) g_{\m} (t,x,R_x  v)+ \alpha P_\gamma g_\m(t,x,v) \,.
\end{equation}
Here, the local Maxwellian $\mu_\eps$ is defined in \eqref{def_mu_eps}:
\begin{equation*}
  \mu_\eps(v) : = \tfrac{\rho_\eps}{(2 \pi \theta_\eps )^{\frac{3}{2  }}}e^{-\frac{|v - u_\eps|^2}{2 \theta_\eps}} \,,
\end{equation*}
with $(\rho_\eps, u_\eps, \theta_\eps) = (1 + \eps(\rho_0 + \rho^b_0), \eps (u_0 + u^b_0), 1 + \eps (\theta_0 + \theta^b_0) )$.  
Recall the definition of $\chi_j$ for $j = 0,1,2,3,4$ given in \eqref{def_chi_0_4}. The macroscopic quantities $(a_\m,b_\m,c_\m) = \p^\m(a,b,c)$ are given by
\begin{equation}\label{def_am_bm_cm}
  a_\m = \int_{\R^3} g_\m \chi_0 \d v\,,\  b_{\m,j} = \int_{\R^3} g_\m \chi_j \d v \ (j=1,2,3)\,,\ c_\m = \tfrac{2}{3} \int_{\R^3} g_\m \chi_4 \d v \,.
\end{equation} 
We also adopt the following notations
$$\A_{ij}(v_\eps) = (v_{\eps,i} v_{\eps,j} - \tfrac{|v_\eps|^2}{3} \delta_{ij} ) \sqrt{\mu(v_\eps)}\,, \quad \B_i(v_\eps) = \tfrac{(|v_\eps|^2-5)v_{\eps,i}}{2}\sqrt{\mu(v_\eps)}\,,$$
where $v_\eps = \tfrac{v - u_\eps}{\sqrt{\theta_\eps}}$.

\subsection{$L^2$ estimates in analytic spaces} In this subsection, our objective is to derive the $L^2$ analytic estimates for $g$. We start by deriving the equations of the macroscopic variables $(a_\m,b_\m,c_\m)$.
\begin{lemma}
  It follows that
  \begin{equation}\label{eqn_am_bm_cm}
    \begin{cases}
      &\eps \p_t a_\m  + \sqrt{\theta_\eps} \nabla_x \cdot b_\m  = - \nabla_x \cdot (u_\eps a_\m ) + \mathfrak{f}_{\m} \,,\\
      &\eps \p_t b_\m  + \sqrt{\theta_\eps} \nabla_x (a_\m  + c_\m ) = - \nabla_x \cdot (u_\eps \otimes b_\m )  -\nabla_x \cdot \int_{\R^3}  \theta_\eps^{-\frac{1}{4}} \A(v_\eps) \P_\eps^\perp g_\m  \d v + \mathfrak{g}_{\m} \,,\\
      &\eps \p_t c_\m  + \tfrac{2}{3}\sqrt{\theta_\eps} \nabla_x \cdot b_\m  = - \nabla_x \cdot (u_\eps c_\m ) - \tfrac{2}{3} \nabla_x \cdot \int_{\R^3} \theta_\eps^{-\frac{1}{4}} \B(v_\eps) \P_\eps^\perp g_\m  \d v  + \mathfrak{h}_{\m} \,,\\
    \end{cases}
  \end{equation}
  where the scalar functions $\mathfrak{f}_\m(t,x) , $ and $ \mathfrak{h}_\m (t,x)$, as well as the vector function $\mathfrak{g}_\m (t,x)$, are defined in \eqref{def_f_1alpha}, \eqref{def_f_3alpha}  and \eqref{def_f_2alpha}, respectively. Moreover, it holds that 
  \begin{equation}\label{fm_gm_hm_L2}
    \|\mathfrak{f}_\m , \mathfrak{g}_\m , \mathfrak{h}_\m  \|_{L^2_x} \leq C(\sqrt{\eps} \|g_\m  \|_2 + \eps \|S_{4,\m} \|_2)\,.
  \end{equation}
\end{lemma}
\begin{proof} We first derive the equation for $a_\m $. Multiplying \eqref{remainder_eqn_alpha} by $\eps \chi_0$, using the fact that $\L_\eps f,\Gamma_\eps(f,g) \in \mathcal{N}_\eps^\perp$, we obtain
\begin{equation}
 \eps \tfrac{\d}{\d t} \int_{\R^3} g_\m  \chi_0 \d v - \eps \int_{\R^3} g_\m  \p_t \chi_0 \d v + \int_{\R^3} v \cdot \nabla_x g_\m  \chi_0 \d v + \eps \int_{\R^3} \tfrac{(\p_t + \frac{1}{\eps} v \cdot \nabla_x) \mu_\eps}{2 \mu_\eps} g_\m \chi_0 \d v = \eps \int_{\R^3} S_{4,\m}  \chi_0 \d v\,.
\end{equation}
Noting that $v = \sqrt{\theta_\eps} v_\eps + u_\eps$, it follows that
\begin{equation}
  \begin{aligned}
    \int_{\R^3} v \cdot \nabla_x g_\m  \chi_0 \d v = &\nabla_x \cdot \int_{\R^3} v  g_\m  \chi_0 \d v - \int_{\R^3} v \cdot  \nabla_x \chi_0 g_\m   \d v \\
    = & \nabla_x \cdot \int_{\R^3} (\sqrt{\theta_\eps} v_\eps + u_\eps)  g_\m  \chi_0 \d v - \int_{\R^3} v \cdot  \nabla_x \chi_0 g_\m   \d v \\
    =& \nabla_x \cdot (\sqrt{\theta_\eps} b_\m ) + \nabla_x \cdot (u_\eps a_\m ) -  \int_{\R^3} v \cdot  \nabla_x \chi_0 g_\m   \d v \,.
  \end{aligned}
\end{equation}
It thereby holds
  \begin{equation}\label{def_f_1alpha}
    \begin{aligned}
      \eps \p_t a_\m  + \sqrt{\theta_\eps} \nabla_x \cdot b_\m  =& - \nabla_x \cdot (u_\eps a_\m ) - b_\m \cdot \nabla_x \sqrt{\theta_\eps} + \int_{\R^3} v \cdot  \nabla_x \chi_0 g_\m   \d v \\
      &+ \eps \int_{\R^3} g_\m  \p_t \chi_0 \d v - \eps \int_{\R^3} \tfrac{(\p_t + \frac{1}{\eps} v \cdot \nabla_x) \mu_\eps}{2 \mu_\eps} g \chi_0 \d v + \eps \int_{\R^3} S_{4,\m}  \chi_0 \d v\\
      := &- \nabla_x \cdot (u_\eps a_\m ) + \mathfrak{f}_{\m }\,.
    \end{aligned}
  \end{equation}
Recall that $(\rho_\eps, u _\eps , \theta_\eps)(t,x,v) = \big(1 + \eps (\rho_0(t,x,v) + \rho^b_0(t,\bar{x},\tfrac{x_3}{\sqrt{\eps}},v)) , \eps (u_0 + u^b_0(t,\bar{x},\tfrac{x_3}{\sqrt{\eps}},v)), 1 + \eps (\theta_0(t,x,v) + \theta^b_0(t,\bar{x},\tfrac{x_3}{\sqrt{\eps}},v))\big)$. As a consequence, $\nabla_x \mu_\eps = O(\sqrt{\eps})$ and $\p_t \mu_\eps = O(\eps) $. Hence,
\begin{equation}
  \|\mathfrak{f}_\m \|_2 \leq C(\sqrt{\eps} \| g_\m \|_2 + \eps\|S_{4,\m} \|_2)\,.
\end{equation}

For the equation of $b_\m $, we multiply \eqref{remainder_eqn_alpha} by $\chi_j \ (j=1,2,3)$ and integrate over $v \in \R^3$, 
\begin{equation}
 \eps \tfrac{\d}{\d t} \int_{\R^3} g_\m  \chi_j \d v - \eps \int_{\R^3} g_\m  \p_t \chi_j \d v + \int_{\R^3} v \cdot \nabla_x g_\m  \chi_j \d v + \eps \int_{\R^3} \tfrac{(\p_t + \frac{1}{\eps} v \cdot \nabla_x) \mu_\eps}{2 \mu_\eps} g_\m \chi_j \d v = \eps \int_{\R^3} S_{4,\m}  \chi_j \d v \,.
\end{equation}
Observe that $ \chi_j = \sqrt{\tfrac{\mu_\eps}{\rho_\eps}} v_{\eps,j}$ and $\mu_\eps(v) = \tfrac{\rho_\eps}{\theta_\eps^{\frac{3}{2}}} \mu(v_\eps)$. Then, the third term above can be calculated as
\begin{align*}
   \int_{\R^3} v \cdot \nabla_x g_\m  \chi_j \d v =& \nabla_x \cdot \int_{\R^3} v  g_\m  \chi_j \d v - \int_{\R^3} v \cdot  \nabla_x \chi_j g_\m   \d v \\
   =& \sum_{i=1}^3 \p_i \int_{\R^3}  \sqrt{\tfrac{ \theta_\eps \mu_\eps}{\rho_\eps }} v_{\eps,i} v_{\eps,j} g_\m  \d v + \nabla_x \cdot \int_{\R^3} u_\eps \chi_j g_\m  \d v - \int_{\R^3} v \cdot  \nabla_x \chi_j g_\m   \d v \\
   =& \sum_{i=1}^3 \p_i \int_{\R^3}  \theta_\eps^{-\frac{1}{4}} \A_{ij}(v_\eps) \P_\eps^\perp g_\m  \d v +  \p_j (\sqrt{\theta_\eps}(a_\m  + c_\m )) \\
   &+ \nabla_x \cdot ( u_\eps b_{\m ,j} ) - \int_{\R^3} v \cdot  \nabla_x \chi_j g_\m   \d v \,, 
\end{align*}
where we used the fact that $\A_{ij}(v_\eps) \in \N^\perp_\eps$. Consequently,
  \begin{equation}\label{def_f_2alpha}
    \begin{aligned}
      \eps \p_t b_\m  + \sqrt{\theta_\eps} \nabla_x (a_\m  + c_\m )=& - \nabla_x \cdot (u_\eps \otimes b_\m ) -\nabla_x \cdot \int_{\R^3}  \theta_\eps^{-\frac{1}{4}} \A(v_\eps) \P_\eps^\perp g_\m  \d v - (a_\m  + c_\m ) \nabla_x \sqrt{\theta_\eps} \\
      &+ \int_{\R^3} v \cdot  \nabla_x \chi g_\m   \d v + \eps \int_{\R^3} g_\m  \p_t \chi \d v - \eps \int_{\R^3} \tfrac{(\p_t + \frac{1}{\eps} v \cdot \nabla_x) \mu_\eps}{2 \mu_\eps} g_\m \chi \d v\\
      & + \eps \int_{\R^3} S_{4,\m}  \chi \d v\\
      := &- \nabla_x \cdot (u_\eps \otimes b_\m )  -\nabla_x \cdot \int_{\R^3}  \theta_\eps^{-\frac{1}{4}} \A(v_\eps) \P_\eps^\perp g_\m  \d v + \mathfrak{g}_{\m }\,.
    \end{aligned}
  \end{equation}

As for $c_\m $, we multiply \eqref{remainder_eqn_alpha} by $\tfrac{2}{3}\eps \chi_4$ and integrate over $v \in \R^3$, 
\begin{equation*}
\begin{aligned}
\tfrac{2}{3} \eps \tfrac{\d}{\d t} \int_{\R^3}  g_\m  \chi_4 \d v -\tfrac{2}{3} \eps \int_{\R^3} g_\m  \p_t \chi_4 \d v &+ \tfrac{2}{3}\int_{\R^3} v \cdot \nabla_x g_\m  \chi_4 \d v \\
&+ \eps \int_{\R^3} \tfrac{(\p_t + \frac{1}{\eps} v \cdot \nabla_x) \mu_\eps}{3 \mu_\eps} g \chi_4 \d v = \tfrac{2}{3} \eps \int_{\R^3} S_{4,\m}  \chi_4 \d v \,. 
\end{aligned}
\end{equation*}
Using the fact that $\chi_4 = \sqrt{\tfrac{\mu_\eps}{\rho_\eps}} (\tfrac{|v_\eps|^2 -3}{2})$ and $\B(v_\eps ) \in \N^\perp_\eps$, it holds that 
 \begin{align*}
   \tfrac{2}{3}\int_{\R^3} v \cdot \nabla_x g_\m  \chi_4 \d v = & \tfrac{2}{3} \nabla_x \cdot \int_{\R^3} v  g_\m  \chi_4 \d v - \tfrac{2}{3} \int_{\R^3} v \cdot  \nabla_x \chi_4 g_\m   \d v \\
   = & \nabla_x \cdot \int_{\R^3}  \sqrt{\tfrac{ \theta_\eps \mu_\eps}{\rho_\eps}} (\tfrac{|v_\eps|^2 -3}{3})v_\eps g_\m  \d v + \tfrac{2}{3} \nabla_x \cdot \int_{\R^3} u_\eps \chi_4 g_\m  \d v - \tfrac{2}{3} \int_{\R^3} v \cdot  \nabla_x \chi_4 g_\m   \d v \\
   = & \tfrac{2}{3}\nabla_x \cdot \int_{\R^3}  \theta_\eps^{-\frac{1}{4}} \B(v_\eps) \P_\eps^\perp g_\m  \d v + \tfrac{2}{3} \nabla_x \cdot  (\sqrt{\theta_\eps}b_\m ) \\
   &+ \nabla_x \cdot ( u_\eps c_{\m } ) - \tfrac{2}{3} \int_{\R^3} v \cdot  \nabla_x \chi_4 g_\m   \d v \,. \\ 
 \end{align*} 
As a result, the equation for $c_\m$ is derived
  \begin{equation}\label{def_f_3alpha}
    \begin{aligned}
      \eps \p_t c_\m  + \tfrac{2}{3}\sqrt{\theta_\eps} \nabla_x \cdot b_\m  =& - \nabla_x \cdot (u_\eps c_\m ) - \tfrac{2}{3} \nabla_x \cdot \int_{\R^3} \theta_\eps^{-\frac{1}{4}} \B(v_\eps) \P_\eps^\perp g_\m  \d v  - \tfrac{2}{3}b_\m \cdot \nabla_x \sqrt{\theta_\eps} \\
      & + \tfrac{2}{3}\int_{\R^3} v \cdot  \nabla_x \chi_4 g_\m   \d v  + \tfrac{2}{3} \eps \int_{\R^3} g_\m  \p_t \chi_4 \d v - \eps \int_{\R^3} \tfrac{(\p_t + \frac{1}{\eps} v \cdot \nabla_x) \mu_\eps}{3 \mu_\eps} g_\m \chi_4 \d v \\
      &+ \tfrac{2}{3} \eps \int_{\R^3} S_{4,\m}  \chi_4 \d v\\
      := &- \nabla_x \cdot (u_\eps c_\m ) - \tfrac{2}{3} \nabla_x \cdot \int_{\R^3} \theta_\eps^{-\frac{1}{4}} \B(v_\eps) \P_\eps^\perp g_\m  \d v  + \mathfrak{h}_{\m } \,.
    \end{aligned}
  \end{equation}

  Finally, observe that $ \p_t (\rho_\eps,u_\eps,\theta_\eps) = O(\eps)$ and $ \p_3 (\rho_\eps,u_\eps,\theta_\eps) = O(\sqrt{\eps})$, the relation \eqref{fm_gm_hm_L2} holds directly by the H\"older inequality.
\end{proof}

In the next lemma, we employ the conservation law \eqref{eqn_am_bm_cm} to derive the estimates for the singular term  $$\int_{\Omega \times \R^3} \tfrac{(\p_t + \frac{1}{\eps} v \cdot \nabla_x ) {\mu_\eps}}{2 \mu_\eps} g_\m ^2 \d x \d v \,, $$
which is crucial for closing our energy estimates. Notice that $(\p_t + \frac{1}{\eps} v \cdot \nabla_x ) {\mu_\eps} = O(\tfrac{1}{\sqrt{\eps}})$.
\begin{lemma}\label{Lmm_mu_eps}Under the same assumptions as in Proposition \ref{Prop_g_k_g_b_k}, there exists a sufficiently small constant $\eps_0$ such that for all $0<\eps<\eps_0$,
  \begin{equation}
    \begin{aligned}
      \int_{\Omega \times \R^3} \tfrac{(\p_t + \frac{1}{\eps} v \cdot \nabla_x ) {\mu_\eps}}{2 \mu_\eps} g_\m ^2 \d x \d v  \leq & \eps \tfrac{\d}{\d t} \mathbb{G}_\m (t) + C_1 \| g_\m \|_2(\sum_{i=1}^{2}\|\p_i g_{\m+1} \|_2)  \\
      &+ C\Big(  \|g_\m \|_2^2 + \eps^{-\frac{5}{2}} \|\P^\perp_\eps g_\m \|_\nu^2 + \eps^7 \|h_\m \|^2_\infty + \|S_{4,\m}\|_2^2 \Big) \,,
    \end{aligned}
  \end{equation}
  for some constants $C,C_1 >0$, and a function $\mathbb{G}_\m$ satisfying $|\mathbb{G}_\m(t)| \leq C \|g_\m(t) \|_2^2\,.$
\end{lemma}
\begin{proof}
We begin by applying the macro-micro decomposition to $g_\m $
  \begin{equation}
   \begin{aligned}
    \int_{\Omega \times \R^3} \tfrac{(\p_t + \frac{1}{\eps} v \cdot \nabla_x ) {\mu_\eps}}{2 \mu_\eps} g_\m ^2 \d x \d v  
    =& \int_{\Omega \times \R^3} \tfrac{ (\p_t + \frac{1}{\eps} v \cdot \nabla_x ) {\mu_\eps}}{2 \mu_\eps} ( (\P_\eps g_\m )^2 + 2 \P_\eps g_\m  \P_\eps^\perp g_\m  +(\P^\perp_\eps g_\m )^2)\d x \d v \\
    :=& \I_{1} + \I_{2} + \I_{3} \,.
   \end{aligned}
  \end{equation}
We first treat the microscopic part $\I_3$. Recall \eqref{def_h} and \eqref{mu_eps_mu_M}. Observing that $\tfrac{(\p_t + \frac{1}{\eps} v \cdot \nabla_x ) {\mu_\eps}}{2 \mu_\eps}$ is a cubic polynomial in $v$, it follows that, for $a = \tfrac{2}{3 - \gamma}$,
  \begin{equation}\label{estimate_I_3}
  \begin{aligned}
    \I_{3} = &\int_{|v| \geq \tfrac{1}{\eps^a}}(\cdots) + \int_{|v| \leq \tfrac{1}{\eps^a}}(\cdots) \\
    \leq & C \|(\p_t + \tfrac{1}{\eps} \nabla_x)(\rho_\eps, u_\eps, \theta_\eps) \|_2 \|\{ 1 + |v|^2\}^{3/2}  \P^\perp_\eps g_\m \mathbbm{1}_{|v| \geq \tfrac{1}{\eps^a}}\|_\infty \|g_\m \|_2\\
    &+ C \|(\p_t + \tfrac{1}{\eps} \nabla_x)(\rho_\eps, u_\eps, \theta_\eps) \|_\infty \|\{ 1 + |v|^2\}^{3/4} \P^\perp_\eps g_\m \mathbbm{1}_{|v| \leq \tfrac{1}{\eps^a}}\|^2_2   \\
    \leq& \eps^{\frac{7}{2}}C \|h_\m \|_\infty \|g_\m \|_2 +  \eps^{-\frac{5}{2}} C\|\P^\perp_\eps g_\m \|_\nu^2 \\
    \leq & C (\eps^{7} \|h_\m \|^2_\infty  + \|g_\m \|^2_2 +  \eps^{-\frac{5}{2}} \|\P^\perp_\eps g_\m \|_\nu^2) \,,
  \end{aligned}
\end{equation}
where we used the fact that
\begin{equation*}
    \|\{ 1 + |v|^2\}^{3/2}  \P^\perp g_\m \mathbbm{1}_{|v| \geq \tfrac{1}{\eps^a}}\|_\infty  \leq C \| \{ 1 + |v|^2\}^{\frac{ 2 \gamma - 6}{2}} h_\m\mathbbm{1}_{|v| \geq \tfrac{1}{\eps^a}}\|_\infty \leq \eps^4 C \| h_\m\|_\infty \,,
\end{equation*}
and
\begin{equation*}
    \|\{ 1 + |v|^2\}^{3/4} \P^\perp g_\m \mathbbm{1}_{|v| \leq \tfrac{1}{\eps^a}}\|^2_2 = \|\{ 1 + |v|^2\}^{\frac{3 - \gamma}{4}} \nu^{\frac{1}{2}} \P^\perp g_\m \mathbbm{1}_{|v| \leq \tfrac{1}{\eps^a}}\|^2_2 \leq  \eps^{-2} C\| \P^\perp_\eps g_\m \|_\nu^2 \,.
\end{equation*}
The H\"older inequality implies
\begin{equation}\label{estimate_I_2}
    \I_{2} \leq \|\nu^{-\frac{1}{2}}\tfrac{(\p_t + \frac{1}{\eps} v \cdot \nabla_x ) {\mu_\eps}}{ \mu_\eps} \P_\eps g_\m  \|_2 \|\P_\eps^\perp g_\m  \|_\nu\\
    \leq C (\| g_\m \|_2^2 + \tfrac{1}{\eps}\| \P^\perp g_\m \|^2_\nu)  \,.
\end{equation}

It remains to estimate the term $\I_{1}$. Note that $v = \sqrt{\theta_\eps} v_\eps + u_\eps$,
  \begin{equation}
    \begin{aligned}
      \I_{1} =&  \int_{\Omega \times \R^3} \tfrac{ (\p_t + \frac{1}{\eps} v_3 \p_3) \mu_\eps  }{2 \mu_\eps}  (\P_\eps g_\m )^2\d x \d v \\
       =& \int_{\Omega \times \R^3} \tfrac{  (\p_t +  \frac{1}{\eps} u_{\eps,3}  \p_3)  \mu_\eps  }{2 \mu_\eps}  (\P_\eps g_\m )^2\d x \d v + \int_{\Omega \times \R^3} \tfrac{   \sqrt{\theta_\eps}  v_{\eps,3} \p_3  \mu_\eps  }{2 \eps \mu_\eps}  (\P_\eps g_\m )^2\d x \d v  
      \,.
    \end{aligned}
  \end{equation}
The first term above is regular and does not exhibit singularity,
\begin{equation}\label{estimate_I_1_1}
    \int_{\Omega \times \R^3} \tfrac{  (\p_t +  \frac{1}{\eps} u_{\eps,3}  \p_3)  \mu_\eps  }{2 \mu_\eps}  (\P_\eps g_\m )^2\d x \d v  \leq C\sqrt
{\eps}  \|\P_\eps g_\m \|_2^2 \,. 
\end{equation}
Observe that 
\begin{equation*}
  \d \mu_\eps = \mu_\eps \big(\tfrac{\d \rho_\eps}{\rho_\eps}  + \tfrac{(v- u_\eps)\cdot \d u_\eps}{\theta_\eps}  + (\tfrac{|v - u_\eps|^2}{2 \theta_\eps} - \tfrac{3}{2}) \tfrac{\d \theta_\eps}{\theta_\eps}\big)\,.
\end{equation*}
The second term on the right-hand side of $\I_1$ can be further calculated as
\begin{equation}
  \begin{aligned}
  &\int_{\Omega \times \R^3} \tfrac{   \sqrt{\theta_\eps}  v_{\eps,3} \p_3  \mu_\eps  }{2 \eps \mu_\eps}  (\P_\eps g_\m )^2\d x \d v  \\
      = & \tfrac{1}{\eps}  \sum_{i=1}^2 \int_{\Omega} \p_3 u_{\eps,i} b_{\m ,3} b_{\m ,i} \d x + \tfrac{1}{\eps} \int_{\Omega} (\tfrac{\p_3 \rho_\eps}{\rho_\eps} + \tfrac{\p_3 \theta_\eps}{\theta_\eps}) \sqrt{\theta_\eps} b_{\m ,3}(a_\m  + c_\m ) \d x \\
    &+ \tfrac{5}{2 \eps} \int_{\Omega} \tfrac{\p_3 \theta_\eps}{\sqrt{\theta_\eps}} b_{\m ,3} c_\m  \d x + \int_{\Omega \times \R^3} \tfrac{v_{\eps,3}^2 \p_3 u_{\eps,3}}{2 \eps } (\P_\eps g_\m)^2 \d x \d v 
    := \sum_{k=1}^4 \J_k \,.
  \end{aligned}
\end{equation}
These terms are singular because $\tfrac{1}{\eps} \p_3 (\rho_\eps, u_\eps, \theta_\eps) = O(\tfrac{1}{\sqrt{\eps}})$. To handle this singularity, we will employ the conservation law to convert it into a loss of tangential spatial derivatives, which can be controlled and absorbed within analytic energy estimates.

Now we turn to the control of the term $\J_1$. From the Maxwell reflection boundary condition \eqref{g_m_BC} and the fact that $\chi_3 (t,\bar{x}, 0 , v)= v_3 \sqrt{\mu} $, it follows that
$$b_{\m,3}(t, \bar{x}, 0) = \int_{\R^3} v_3 \sqrt{\mu} g_\m \d v  = 0\,.$$
Hence, $b_{\m ,3}(t, x) = \int_{0}^{x_3}\p_3 b_{\m ,3}(t, \bar{x}, x_3^\prime ) \d x_3^\prime $. We then apply the mass conservation law \eqref{eqn_am_bm_cm}$_1$ to obtain
\begin{align*}
  \J_1 = & \tfrac{1}{\eps}  \sum_{i=1}^2 \int_{\Omega} \p_3 u_{\eps,i}  \int_{0}^{x_3}\p_3 b_{\m ,3} \d x_3^\prime  b_{\m ,i} \d x \\
  = & \tfrac{1}{\eps}  \sum_{i=1}^2 \int_{\Omega} \p_3 u_{\eps,i}  \int_{0}^{x_3}(- \eps \tfrac{\p_t a_\m }{\sqrt{\theta_\eps}} - \sum_{j=1}^{2} \p_j b_{\m ,j} - \tfrac{\nabla_x \cdot (u_\eps a_\m )}{\sqrt{\theta_\eps}} + \tfrac{\mathfrak{f_\m}}{\sqrt{\theta_\eps}})\d x_3^\prime  b_{\m ,i} \d x \\
  = & -   \sum_{i=1}^2 \int_{\Omega} \p_3 u_{\eps,i}  \int_{0}^{x_3} \tfrac{\p_t a_\m }{\sqrt{\theta_\eps}} \d x_3^\prime  b_{\m ,i} \d x 
  - \tfrac{1}{\eps}  \sum_{i=1}^2 \int_{\Omega} \p_3 u_{\eps,i}  \int_{0}^{x_3}  \sum_{j=1}^{2} \p_j (b_{\m ,j} ) \d x_3^\prime  b_{\m ,i} \d x \\
  & - \tfrac{1}{\eps}  \sum_{i=1}^2 \int_{\Omega} \p_3 u_{\eps,i}  \int_{0}^{x_3}  \sum_{j=1}^{2} \tfrac{ \p_j  (u_{\eps,j} a_\m )}{\sqrt{\theta_\eps}} \d x_3^\prime  b_{\m ,i} \d x \\
 & -\tfrac{1}{\eps}  \sum_{i=1}^2 \int_{\Omega} \p_3 u_{\eps,i}  \int_{0}^{x_3} \tfrac{\p_3 (u_{\eps,3} a_\m )}{\sqrt{\theta_\eps}} \d x_3^\prime  b_{\m ,i} \d x 
  +\tfrac{1}{\eps}  \sum_{i=1}^2 \int_{\Omega} \p_3 u_{\eps,i}  \int_{0}^{x_3} \tfrac{\mathfrak{f_\m}}{\sqrt{\theta_\eps}} \d x_3^\prime  b_{\m ,i} \d x \\
  = &\sum_{k=1}^{5}\J_{1,k} \,.
\end{align*}
To remove the loss of time regularity in $\J_{1,1}$, we make use of the chain rule
\begin{equation*}
  \begin{aligned}
    \J_{1,1} = &   \sum_{i=1}^2 \int_{\Omega} \p_3 u_{\eps,i}  \int_{0}^{x_3} \tfrac{  a_\m }{\sqrt{\theta_\eps}} \d x_3^\prime  \p_t b_{\m ,i} \d x +    \sum_{i=1}^2 \int_{\Omega} \p_3 u_{\eps,i}  \int_{0}^{x_3} \tfrac{ a_\m }{\p_t \sqrt{\theta_\eps}} \d x_3^\prime  b_{\m ,i} \d x\\
    & + \sum_{i=1}^2 \int_{\Omega} \p_t \p_3 u_{\eps,i}  \int_{0}^{x_3} \tfrac{  a_\m }{\sqrt{\theta_\eps}} \d x_3^\prime  b_{\m ,i} \d x- \tfrac{\d}{\d t}   \sum_{i=1}^2 \int_{\Omega} \p_3 u_{\eps,i}  \int_{0}^{x_3} \tfrac{  a_\m }{\sqrt{\theta_\eps}} \d x_3^\prime  b_{\m ,i} \d x \,.
  \end{aligned}
\end{equation*}
For the first term on the right-hand side above, we apply the momentum conservation law $\eqref{eqn_am_bm_cm}_2$, to obtain 
  \begin{align}\label{J_1_k_1}
    &\nonumber  \sum_{i=1}^2 \int_{\Omega} \p_3 u_{\eps,i}  \int_{0}^{x_3} \tfrac{  a_\m }{\sqrt{\theta_\eps}} \d x_3^\prime  \p_t b_{\m ,i} \d x \nonumber\\
    = &\nonumber\tfrac{1}{\eps}  \sum_{i=1}^2 \int_{\Omega} \p_3 u_{\eps,i}  \int_{0}^{x_3} \tfrac{  a_\m }{\sqrt{\theta_\eps}} \d x_3^\prime  (- \sqrt{\theta_\eps} \p_i(a_\m  + c_\m ) - \nabla_x \cdot (u_\eps b_{\m,i}) \\
    &\nonumber\qquad\qquad\qquad- \nabla_x \cdot \int_{\R^3} \theta_\eps^{-\frac{1}{4}} \A(v_\eps) \P^\perp_\eps g_\m  \d v + \mathfrak{g}_{\m ,i})\d x \\
    = & \tfrac{1}{\eps}  \sum_{i=1}^2 \int_{\Omega} \p_3 u_{\eps,i}  \int_{0}^{x_3} \tfrac{  a_\m }{\sqrt{\theta_\eps}} \d x_3^\prime  (- \sqrt{\theta_\eps} \p_i(a_\m  + c_\m ) - \sum_{j=1}^{2}\p_j (u_{\eps,j} b_{\m,i}) \\
    &\nonumber\qquad\qquad\qquad- \sum_{j=1}^{2}\p_j \int_{\R^3} \theta^{-\frac{1}{4}}_\eps \A(v_\eps) \P^\perp_\eps g_\m  \d v + \mathfrak{g}_{\m ,i})\d x \\
    &\nonumber- \tfrac{1}{\eps}  \sum_{i=1}^2 \int_{\Omega} \p_3 u_{\eps,i}  \int_{0}^{x_3} \tfrac{  a_\m }{\sqrt{\theta_\eps}} \d x_3^\prime  (  \p_3 (u_{\eps,j} b_{\m,i}) + \p_3\int_{\R^3} \theta^{-\frac{1}{4}}_\eps \A(v_\eps) \P^\perp_\eps g_\m  \d v  )\d x \,.
  \end{align} 
Note that $x_3 \p_3 u_{\eps, i} = O(\eps)$. The first term on the right-hand side of \eqref{J_1_k_1} loss only tangtial derivatives, and can be estimated as
\begin{equation}
  \tfrac{1}{\eps}  \sum_{i=1}^2 \int_{\Omega} x_3 \p_3 u_{\eps,i}  \tfrac{\int_{0}^{x_3} \tfrac{  a_\m }{\sqrt{\theta_\eps}} \d x_3^\prime }{x_3} (\cdots)\d x \\
  \leq C \|g_\m  \|_2 ( \sum_{i=1}^{2}\|\p_i g_{\m} \|_2 + \sqrt{\eps} \|g_\m \|_2 + \eps \|S_{4,\m} \|_2) \,,
\end{equation}
where we used the Hardy inequality and \eqref{fm_gm_hm_L2}. The second term on the right-hand side of \eqref{J_1_k_1} loss the normal derivatives.  Noticing that \(\int_{0}^{x_3} \tfrac{a_\m}{\sqrt{\theta_\eps}} \d x_3^\prime\) vanishes on the boundary, we integrate by parts to obtain, 
  \begin{align*}
  &- \tfrac{1}{\eps}  \sum_{i=1}^2 \int_{\Omega} \p_3 u_{\eps,i}  \int_{0}^{x_3} \tfrac{  a_\m }{\sqrt{\theta_\eps}} \d x_3^\prime \p_3\big(u_{\eps,j} b_{\m,i} +  \int_{\R^3} \theta^{-\frac{1}{4}}_\eps \A(v_\eps) \P^\perp_\eps g_\m  \d v\big)\d x \\
    =& \tfrac{1}{\eps}  \sum_{i=1}^2 \int_{\Omega} \p^2_3 u_{\eps,i}  \int_{0}^{x_3} \tfrac{  a_\m }{\sqrt{\theta_\eps}} \d x_3^\prime  \big(   u_{\eps,j} b_{\m,i} +  \int_{\R^3} \theta^{-\frac{1}{4}}_\eps \A(v_\eps) \P^\perp_\eps g_\m  \d v  \big)\d x \\
    &+ \tfrac{1}{\eps}  \sum_{i=1}^2 \int_{\Omega} \p_3 u_{\eps,i}    \tfrac{  a_\m }{\sqrt{\theta_\eps}}   \big(   (u_{\eps,j} b_{\m,i}) +  \int_{\R^3} \theta^{-\frac{1}{4}}_\eps \A(v_\eps) \P^\perp_\eps g_\m  \d v   \big)\d x \\
    \leq & \tfrac{C}{\eps}(\|x_3 \p_3^2 u_{\eps,i} \|_\infty + \| \p_3 u_{\eps,i} \|_\infty) \|a_\m  \|_2 (\eps \| b_\m  \|_2 +  \| \P^\perp_\eps g_\m  \|_2)\\
    \leq &  C \sqrt{\eps}\|g_\m \|^2_2 + \tfrac{C}{\eps^{\frac{3}{2}}} \|\P^\perp_\eps g_\m\|_\nu^2  \,,
  \end{align*} 
where the fact that $\|x_3 \p_3^2 u_{\eps,i} \|_\infty + \| \p_3 u_{\eps,i} \|_\infty = O(\sqrt{\eps})$ has been used. As a result, 
  \begin{align*}
    \J_{1,1} \leq & C \|g_\m  \|_2 ( \sum_{i=1}^{2}\|\p_i g_{\m} \|_2)+ C (\eps \|S_{4,\m} \|^2_2 + \sqrt{\eps} \|g_\m \|^2_2 + \tfrac{1}{\eps^{\frac{3}{2}}} \|\P^\perp_\eps g_\m\|_\nu^2 )\\
    & - \tfrac{\d}{\d t}   \sum_{i=1}^2 \int_{\Omega} \p_3 u_{\eps,i}  \int_{0}^{x_3} \tfrac{  a_\m }{\sqrt{\theta_\eps}} \d x_3^\prime   b_{\m ,i} \d x \,.
  \end{align*}
For $\J_{1,2}, \J_{1,3}$ and $\J_{1,5}$, we use the Hardy inequality to derive
\begin{equation*}
  \begin{aligned}
    \J_{1,2} +\J_{1,3} + \J_{1,5} = & - \tfrac{1}{\eps}  \sum_{i=1}^2 \int_{\Omega} x_3 \p_3 u_{\eps,i} \tfrac{\int_{0}^{x_3} (\cdots) \d x_3^\prime }{x_3} b_{\m ,i} \d x \\
    \leq & C \| g_\m  \|_2 \|( \sum_{i=1}^{2}\|\p_i g_{\m} \|_2 +  \sqrt{\eps}\| g_\m \|_2  + \eps \| S_{4,\m}\|_2)\,.
  \end{aligned}
\end{equation*}
For $\J_{1,4}$, we note that $u_{\eps,3} (t,\bar{x},0) =0$, and hence it holds that
\begin{equation*}
  \begin{aligned}
    \J_{1,4} = & -\tfrac{1}{\eps}  \sum_{i=1}^2 \int_{\Omega} \p_3 u_{\eps,i}  \int_{0}^{x_3} \p_3(\tfrac{  u_{\eps,3} a_\m }{\sqrt{\theta_\eps}}) \d x_3^\prime  b_{\m ,i} \d x + \tfrac{1}{\eps}  \sum_{i=1}^2 \int_{\Omega} \p_3 u_{\eps,i}  \int_{0}^{x_3} \tfrac{ u_{\eps,3} a_\m  }{\p_3 \sqrt{\theta_\eps}} \d x_3^\prime  b_{\m ,i} \d x \\
    = & -\tfrac{1}{\eps}  \sum_{i=1}^2 \int_{\Omega} \p_3 u_{\eps,i}   (\tfrac{  u_{\eps,3} a_\m }{\sqrt{\theta_\eps}})   b_{\m ,i} \d x + \tfrac{1}{\eps}  \sum_{i=1}^2 \int_{\Omega} x_3 \p_3 u_{\eps,i} \tfrac{\int_{0}^{x_3} \frac{ u_{\eps,3} a_\m  }{\p_3 \sqrt{\theta_\eps}} \d x_3^\prime}{x_3}  b_{\m ,i} \d x \\
    \leq & C \sqrt{\eps} \|g_\m \|_2^2  \,.
  \end{aligned}
\end{equation*}
Consequently, we obtain the estimate for $\J_1$
\begin{equation}\label{estimate_J_1}
 \begin{aligned}
   \J_1 \leq& C \|g_\m  \|_2 ( \sum_{i=1}^{2}\|\p_i g_{\m} \|_2)+ C (\eps \|S_{4,\m} \|^2_2 + \sqrt{\eps} \|g_\m \|^2_2 + \tfrac{1}{\eps^\frac{3}{2}} \|\P^\perp_\eps g_\m\|_\nu^2 )\\
     & - \tfrac{\d}{\d t}   \sum_{i=1}^2 \int_{\Omega} \p_3 u_{\eps,i}  \int_{0}^{x_3} \tfrac{  a_\m }{\sqrt{\theta_\eps}} \d x_3^\prime   b_{\m ,i} \d x \,.
 \end{aligned}
\end{equation}

The terms $\J_j \ (j=2,3,4)$ can be handled in a similar manner, by employing the conservation law \eqref{eqn_am_bm_cm} and the Hardy inequality. Since the computations are analogous yet tedious, we omit the details for brevity.

Collecting \eqref{estimate_I_3}-\eqref{estimate_I_1_1}, and \eqref{estimate_J_1}, we conclude the proof of the lemma.
\end{proof}

Now we are already to derive the $L^2$ estimates for $g$ in analytic spaces. In the sequel, we take
 $$\sigma(t) = 2 -\lambda t$$
 for some $\lambda >0 $ determined later. We always assume $\sigma(t) \in [1,2]$.
\begin{lemma}\label{Lmm_L_2}  Under the same assumptions as in Proposition \ref{Prop_g_k_g_b_k}, there exists a sufficiently small constant $\eps_0$ such that for all $0<\eps<\eps_0$,
  \begin{equation}
  \begin{aligned}
    & \tfrac{1}{2}\tfrac{\d}{\d t} \|g \|_X^2 + \lambda \| g\|^2_{X^\frac{1}{2}}+ \tfrac{(2 \alpha - \alpha^2)}{2 \eps} \sum_{\m \in \mathbb{N}_0^2} \tfrac{\sigma(t)^{2 |\m|}}{(\m !)^2}\int_{\Sigma_+}(v\cdot n)\left|(I-P_\gamma)g_\m \right|^2 \d \bar{x}\d v  + \tfrac{c}{\eps^3} \|\nu^{\frac{1}{2}} \P^\perp_\eps g \|_{X}^2 \\
     \leq & \eps \tfrac{\d}{\d t} \sum_{\m \in \mathbb{N}_0^2} \tfrac{\sigma(t)^{2 |\m|}}{(\m !)^2}  \mathbb{G}_\m(t)+  C\eps^7 \| h\|_Y^2 \| g\|_X^2 + C \| g \|_{X}^2 + C + (C_1 + 2 C \eps \lambda ) \| g \|_{X^\frac{1}{2}}^2 \,.
  \end{aligned}
  \end{equation}
  for some constant $c>0$ and function $\mathbb{G}_\m \leq C \| g_\m \|_2^2$, where $P_\gamma g$ is defined in \eqref{def_P_gamma}.
\end{lemma}
\begin{proof}
Multiplying \eqref{remainder_eqn_alpha} by $\tfrac{\sigma(t)^{2 \m}}{(\m !)}g_\m $ and integrating over $(x,v) \in \Omega \times \R^3$,
\begin{equation}
  \begin{aligned}
  &\sum_{\m \in \mathbb{N}_0^2} \tfrac{\sigma(t)^{2 |\m|}}{(\m !)^2} \Big(\tfrac{1}{2} \tfrac{\d}{\d t} \|g_\m  \|_2^2 + \tfrac{1}{2 \eps} \int_{\Omega \times \R^3} v \cdot \nabla_x g_\m ^2 \d x \d v + \tfrac{1}{\eps^3} \int_{\Omega \times \R^3} g_\m  \L_\eps g_\m  \d x \d v \Big)\\
     = &- \sum_{\m \in \mathbb{N}_0^2} \tfrac{\sigma(t)^{2 |\m|}}{(\m !)^2}\int_{\Omega \times \R^3} \tfrac{(\p_t + \frac{1}{\eps} v \cdot \nabla_x ) {\mu_\eps}}{2 \mu_\eps} g_\m ^2 \d x \d v + \sum_{\m \in \mathbb{N}_0^2} \tfrac{\sigma(t)^{2 |\m|}}{(\m !)^2}\langle g_\m , S_\m  \rangle \,.
  \end{aligned}
\end{equation}
A direct calculation shows
\begin{equation}
  \tfrac{1}{2}\sum_{\m \in \mathbb{N}_0^2}\tfrac{\sigma(t)^{2 \m}}{(\m !)^2}\tfrac{\d}{\d t} \|g_\m  \|_2^2 =\tfrac{1}{2} \tfrac{\d}{\d t} \|g \|_{X}^2 + \lambda \|g \|_{X^\frac{1}{2}}^2 \,.
\end{equation}
Using the boundary condition \eqref{g_m_BC}, we have
\begin{equation}\label{3.1.3}
\begin{aligned}
  &\tfrac{1}{2 \eps} \int_{\Omega \times \R^3} v \cdot \nabla_x g_\m ^2 \d x \d v=\tfrac{1}{2 \eps} \int_{\Sigma}(v\cdot n)|g_\m|^2 \d \bar{x}\d v\\
   =&\tfrac{1}{2 \eps} \int_{\Sigma_-}(v\cdot n)\left[(1-\alpha)g_\m(t,x,R_x v) +\alpha P_\gamma g_\m \right]^2 \d\bar{x} \d v + \tfrac{1}{2 \eps}\int_{\Sigma_+}(v\cdot n)g_\m ^2 \d \bar{x}\d v\\
   =&\tfrac{(2-\alpha)\alpha}{2 \eps}\int_{\Sigma_+}(v\cdot n)g_\m^2 \d\bar{x}\d v -\tfrac{(2-\alpha)\alpha}{2 \eps}\sqrt{2\pi}\int_{\partial \Omega}\left(\int_{v\cdot n>0}(v\cdot n) g_\m \sqrt{\mu(v)}\d v\right)^2 \d\bar{x}\\
   =&\tfrac{(2-\alpha)\alpha}{2 \eps}\int_{\Sigma_+}(v\cdot n)\left|(I-P_\gamma)g_\m \right|^2d\bar{x}\d v \,.
\end{aligned}
\end{equation}
The coercivity of the linear operator $\L_\eps$ in \eqref{def_tilde_c_0} implies
\begin{equation}
 \sum_{\m \in \mathbb{N}_0^2} \tfrac{\sigma(t)^{2 |\m|}}{(\m !)^2}  \tfrac{1}{\eps^3} \int_{\Omega \times \R^3} g_\m  \L_\eps g_\m  \d x \d v \geq \tfrac{ \tilde{c}_0 }{\eps^3} \|\nu^{\frac{1}{2}} \P^\perp_\eps g \|_{X}^2 \,.
\end{equation} 

Recall from \eqref{def_h} that $h = \sqrt{\tfrac{\mu_\eps}{\mu_M}} \omega_\beta g$. Following the same argument as in \cite[Lemma 2.3]{Guo-2002-CPAM}, it follows that
\begin{equation}
 \begin{aligned}
&\sum_{\m \in \mathbb{N}_0^2} \tfrac{\sigma(t)^{2 |\m|}}{(\m !)^2} \langle g_\m, S_{1,\m} \rangle\\
 =& \sum_{\m \in \mathbb{N}_0^2, \n+ \j = \m} \tfrac{\sigma(t)^{2 |\m|}}{(\m !)^2}\eps^{2} C_{\m }^\n \int_{\Omega \times \R^3} \Gamma_\eps (g_\n,g_\j) g_\m  \d x \d v \\
\leq & \sum_{\m \in \mathbb{N}_0^2, \n+ \j = \m} \tfrac{\sigma(t)^{2 |\m|}}{(\m !)^2}\eps^{2} C_{\m }^\n C   \| \P_\eps^\perp g_\m  \|_\nu (\|h_\n \|_{\infty} \|g_\j \|_2 + \|h_\j \|_{\infty} \|g_\n \|_2 ) \\
\leq & \tfrac{\delta}{\eps^3} \|\nu^{\frac{1}{2}} \P^\perp_\eps g  \|_X^2 + \sum_{\m \in \mathbb{N}_0^2, \n+ \j = \m} \tfrac{\sigma(t)^{2 |\m|}}{(\m !)^2} (C^\n_\m)^2C_\delta \eps^7 (\|h_\n \|^2_{\infty} \|g_\j \|^2_2 + \|h_\j \|^2_{\infty} \|g_\n \|^2_2 ) \\
\leq & \tfrac{\delta}{\eps^3} \|\nu^{\frac{1}{2}} \P^\perp_\eps g  \|_X^2 + C_\delta \eps^7 \| h \|_Y^2 \|g\|^2_{X} \,.
 \end{aligned}
\end{equation} 
Using Proposition \ref{Prop_g_k_g_b_k}, it holds
\begin{equation}
\begin{aligned}
   \langle g_\m, S_{2,\m} \rangle = &\sum_{i=1}^{13} \eps^{\frac{i-4}{2}} C^\n_\m \big( \Gamma_\eps (g_\n, \p^\j (\sqrt{\tfrac{\mu}{\mu_\eps}}(g_i + g^b_i + g^{bb}_i))) + \Gamma_\eps ( \p^\j (\sqrt{\tfrac{\mu}{\mu_\eps}}(g_i + g^b_i + g^{bb}_i)), g_\n) , g_\m  
  \big) \\
  \leq & C \Big(\sum_{i=1}^{13} \eps^{\frac{i-4}{2}} C^\n_\m \| \omega_\beta \p^\j (\sqrt{\tfrac{\mu}{\mu_\eps}}(g_i + g^b_i + g^{bb}_i)) \|_\infty \Big)\| g_\n\|_\nu \|\P^\perp_\eps g_\m  \|_\nu  \\
  \leq & \tfrac{\delta + C_\delta C_0 \eps^3}{\eps^3} \| \P^\perp_\eps g_\m \|_\nu^2 + C_\delta C_0 \| g \|_X^2\,,
\end{aligned}
\end{equation}
where we used $\|g_\n \|_\nu \leq \|\P_\eps g_\n \|_\nu + \|\P_\eps^\perp g_\n \|_\nu \leq \| g_\n \|_2+ \|\P_\eps^\perp g_\n \|_\nu$. Similarly,
\begin{equation}
\begin{aligned}
  \sum_{\m \in \mathbb{N}_0^2} \tfrac{\sigma(t)^{2 |\m|}}{(\m !)^2} \langle g_\m, S_{3,\m} \rangle = & - \sum_{\m \in \mathbb{N}_0^2, \n + \j = \m} \tfrac{\sigma(t)^{2 |\m|}}{(\m !)^2}\tfrac{1}{\eps} C^\n_\m \big( \Gamma_\eps (g_\n, \tfrac{r_\j}{\sqrt{\mu_\eps}}) + \Gamma_\eps ( \tfrac{r_\j}{\sqrt{\mu_\eps}}, g_\n) , g_\m  
  \big) \\
  \leq & \sum_{\m \in \mathbb{N}_0^2} \tfrac{\sigma(t)^{2 |\m|}}{(\m !)^2}\tfrac{C}{\eps} C^\n_\m \|\omega_\beta \tfrac{r_\j}{\sqrt{\mu_\eps}}\|_\infty \| g_\n\|_\nu \|\P^\perp_\eps g_\m  \|_\nu  \\
  \leq & \tfrac{\delta}{\eps^3} \| \nu^{\frac{1}{2}} \P^\perp_\eps g \|_X^2 + C_\delta \eps \| g \|_X^2 \,.
\end{aligned}
\end{equation}
Applying Proposition \ref{Prop_g_k_g_b_k} together with the H\"older inequality yields
 \begin{equation}
 \sum_{\m \in \mathbb{N}_0^2} \tfrac{\sigma(t)^{2 |\m|}}{(\m !)^2}\langle g_\m, S_{4,\m} \rangle \leq \|S_{4} \|_X \|g \|_X \leq\|g \|_X^2 + C_0  \,.
\end{equation}
It follows from Lemma \ref{Lmm_mu_eps} that
  \begin{equation}
    \begin{aligned}
    & \sum_{\m \in \mathbb{N}_0^2} \tfrac{\sigma(t)^{2 |\m|}}{(\m !)^2} \int_{\Omega \times \R^3} \tfrac{(\p_t + \frac{1}{\eps} v \cdot \nabla_x ) {\mu_\eps}}{2 \mu_\eps} g_\m ^2 \d x \d v  \\
     \leq & \eps \sum_{\m \in \mathbb{N}_0^2} \tfrac{\sigma(t)^{2 |\m|}}{(\m !)^2} \tfrac{\d}{\d t} \mathbb{G}_\m (t) +C_1 \sum_{\m \in \mathbb{N}_0^2} \tfrac{\sigma(t)^{2 |\m|}}{(\m !)^2} \| g_\m \|_2 (\sum_{i=1}^{2}\|\p_i g_{\m} \|_2) \\
    &+ \sum_{\m \in \mathbb{N}_0^2} \tfrac{\sigma(t)^{2 |\m|}}{(\m !)^2} C\Big( \| g\|_X^2 + \eps^{-\frac{5}{2}} \|\nu^{\frac{1}{2}}\P^\perp_\eps g_\m \|_X^2 + \eps^7 \|h \|^2_Y  \Big)\,.
    \end{aligned}
  \end{equation}
Using the fact that $|\mathbb{G}_\m| \leq C \| g_\m\|_2^2$, we obtain
  \begin{equation}
     \begin{aligned}
      \eps \sum_{\m \in \mathbb{N}_0^2} \tfrac{\sigma(t)^{2 |\m|}}{(\m !)^2} \tfrac{\d}{\d t} \mathbb{G}_\m (t) = & \eps \tfrac{\d}{\d t}  \sum_{\m \in \mathbb{N}_0^2} \tfrac{\sigma(t)^{2 |\m|}}{(\m !)^2}\mathbb{G}_\m (t) - 2\eps \lambda \sum_{\m \in \mathbb{N}_0^2} \tfrac{ | \m |\sigma(t)^{2 |\m| - 1}}{(\m !)^2} \mathbb{G}_\m (t) \\
      \leq & \eps \tfrac{\d}{\d t}  \sum_{\m \in \mathbb{N}_0^2} \tfrac{\sigma(t)^{2 |\m|}}{(\m !)^2}\mathbb{G}_\m (t) + 2 C \eps \lambda  \| g \|^2_{X^\frac{1}{2}}\,.
     \end{aligned}
  \end{equation}
By direct computations,
\begin{equation}
   \begin{aligned}
    &\sum_{\m \in \mathbb{N}_0^2} \tfrac{\sigma(t)^{2 |\m|}}{(\m !)^2}\| g_\m \|_2 (\sum_{i=1}^{2}\|\p_i g_{\m} \|_2) \\
     = &  \sum_{\m \in \mathbb{N}_0^2} \big(\tfrac{\sigma(t)^{2 |\m| - 1} |\m|}{(\m !)^2}\big)^\frac{1}{2}\| g_\m \|_2 \times  \big(\tfrac{\sigma(t)^{2 |\m| + 1} (|\m| + 1)}{((\m_1 +1) ! \m_2 !)^2}\big)^\frac{1}{2} \|\p_1 g_{\m} \|_2 \tfrac{(\m_1+ 1)}{\sqrt{|\m| (|\m | +1)}}\\
     & + \sum_{\m \in \mathbb{N}_0^2} \big(\tfrac{\sigma(t)^{2 |\m| - 1} |\m|}{(\m !)^2}\big)^\frac{1}{2}\| g_\m \|_2 \times  \big(\tfrac{\sigma(t)^{2 |\m| + 1} (|\m| + 1)}{(\m_1  ! (\m_2 + 1) !)^2}\big)^\frac{1}{2} \|\p_2 g_{\m} \|_2 \tfrac{(\m_2+ 1)}{\sqrt{|\m| (|\m | +1)}}\\
    \leq &\|g \|_{X^\frac{1}{2}}^2 \,.
   \end{aligned}
\end{equation}

By combining all the estimates derived above and choosing $\eps_0$ and $\delta$ sufficiently small, the lemma follows.
\end{proof}

\subsection{\(L^\infty\) estimates in analytic spaces} 
In this subsection, we deduce the $L^\infty$ estimates in analytic spaces for the solutions of the remainder equation \eqref{remainder_eqn_alpha}.  As in \cite{Guo-2010-ARMA} or \cite{Guo-Huang-Wang-2021-ARMA}, we introduce the following backward characteristics. Given $(t,x,v)$, we define $[X(s),V(s)]$ satisfying
\begin{equation*}
  \begin{cases}
    \frac{\d}{\d s }X(s)= \tfrac{1}{\eps} V(s)\,,\quad \frac{\d}{\d s}V(s)=0\,,\\
    [X(t;t,x,v),V(t;t,x,v)]=[x,v]\,.
  \end{cases}
\end{equation*}
Then $[X(s;t,x,v),V(s;t,x,v)]=[x-\tfrac{1}{\eps}(t-s)v,v]=[X(s),V(s)]$. For $(x,v)\in \Omega\times \mathbb{R}^3$, the backward exist time $t_b(x,v)>0$ is defined to be the first moment at which the backward characteristics line $[X(s;0,x,v);V(s;0,x,v)]$ emerges from $\partial\Omega$. Hence, it holds that
$$t_b(x,v)=\inf\{t>0:x- \tfrac{1}{\eps}tv\notin \Omega\}\,,$$
which indicates that $x-t_b v\in \partial\Omega$. We also define $x_b(x,v)=x-\tfrac{1}{\eps}t_bv$. Note that we use $t_b(x,v)$ whenever it is well-defined.
For the half-space problem, the backward trajectory does not hit the boundary for the case $v_3<0$, and the backward trajectory will hit the boundary once for the case $v_3>0$. By the Maxwell reflection boundary condition, the backward characteristics can be written as
\begin{equation*}
\begin{aligned}
    X_{cl}(s;t,x,v)=&\mathbbm{1}_{[t_1,t)}(s)\{x- \tfrac{1}{\eps}(t-s)v\}\\
    &+\mathbbm{1}_{(-\infty,t_1)}(s)\{x_b-\tfrac{1}{\eps}[(1-\alpha)R_{x_1}v+\alpha v^\prime ](t_1-s)\}\,,\\
    V_{cl}(s;t,x,v)=&\mathbbm{1}_{[t_1,t)}(s)v+\mathbbm{1}_{(-\infty,t_1)}(s)[(1-\alpha)R_{x_1}v+\alpha v^\prime ]\,,
\end{aligned}
\end{equation*}
where $v^\prime \in \mathcal{V}:=\{v^\prime :v^\prime \cdot n>0\}$ and
$$(t_1,x_1)=(t-t_b(x,v),x_b(x,v))\,.$$

As in \cite{Caflish-1980-CPAM,Guo-Jang-Jiang-2010-CPAM}, we define 
\begin{equation*}
  \L_M g = -\tfrac{1}{\sqrt{\mu_M}} \{ B(\mu_\eps, \sqrt{\mu_M} g) + B( \sqrt{\mu_M} , \mu_\eps)\} = (\nu_\eps - K )g\,,
\end{equation*}
where $\nu_\eps$ is defined in \eqref{def_nu_eps} and \(Kg =K_{2}g - K_{1}g  \) with
\begin{align*}
K_{1}g &= \int\limits_{\mathbb{R}^{3}\times\mathbb{S}^{2}} b(\theta)|v_1 -v|^{\gamma}\sqrt{\mu_{M}(v_1)}\frac{\mu_\eps (v)}{\sqrt{\mu_{M}(v)}}\,g(v_1)\,\d v_1\,\d\omega \,,  \\
K_{2}g &= \int\limits_{\mathbb{R}^{3}\times\mathbb{S}^{2}} b(\theta)|v_1 -v|^{\gamma}\mu_\eps(v_1^{\prime})\frac{\sqrt{\mu_{M}(v^{\prime})}}{\sqrt{\mu_{M}(v)}}\,g(v^{\prime})\,\d v_1\,\d\omega \\
&\quad + \int\limits_{\mathbb{R}^{3}\times\mathbb{S}^{2}} b(\theta)|v_1 -v|^{\gamma}\mu_\eps(v^{\prime})\frac{\sqrt{\mu_{M}(v_1^\prime)}}{\sqrt{\mu_{M}(v)}}\,g(v_1^\prime)\,\d v_1\,\d\omega \,.
\end{align*}

Consider a smooth cutoff function \(0 \leq \Upsilon_{m} \leq 1\) such that for any \(m > 0\),
\[
\Upsilon_{m}(s) \equiv 1 \text{ for } s \leq m \,, \quad \Upsilon_{m}(s) \equiv 0 \text{ for } s \geq 2m \,.
\]
Then define
\begin{align*}
K^{m}g &= \int\limits_{\mathbb{R}^{3}\times\mathbb{S}^{2}} b(\theta)|v_1 -v|^{\gamma}\,\Upsilon_{m}(|v_1 -v|)\,\sqrt{\mu_{M}(v_1)}\frac{\mu_\eps(v)}{\sqrt{\mu_{M}(v)}}\,g(v_1)\,\d v_1\,\d\omega \\
&\quad - \int\limits_{\mathbb{R}^{3}\times\mathbb{S}^{2}} b(\theta)|v_1 -v|^{\gamma}\,\Upsilon_{m}(|v_1 -v|)\mu_\eps(v_1^\prime)\frac{\sqrt{\mu_{M}(v^{\prime})}}{\sqrt{\mu_{M}(v)}}\,g(v^{\prime})\,\d v_1\,\d\omega \\
&\quad - \int\limits_{\mathbb{R}^{3}\times\mathbb{S}^{2}} b(\theta)|v_1 -v|^{\gamma}\,\Upsilon_{m}(|v_1 -v|)\mu_\eps(v^{\prime})\frac{\sqrt{\mu_{M}(v_1^\prime)}}{\sqrt{\mu_{M}(v)}}\,g(v_1^\prime)\,\d v_1\,\d\omega \,,
\end{align*}
and also define
\[
K^{c} = K - K^{m}\,.
\]
The next lemma is Lemma 2.3 of \cite{Guo-Jang-Jiang-2010-CPAM}.
\begin{lemma}\label{Lmm_K_m}
\begin{equation}\label{est_K_m}
|K^{m}g(v)| \leq C m^{3+\gamma} \nu_\eps \|g\|_{\infty} 
\end{equation}
and \(K^{c}g(v) = \int_{\mathbb{R}^{3}} k(v,v^{\prime})g(v^{\prime}) \d v^{\prime}\) where the kernel \(k\) satisfies, for some \(c > 0\),
\begin{equation}\label{est_K_c}
k(v,v^{\prime}) \leq C_{m} \frac{\exp\{-c |v-v^{\prime}|^{2}\}}{|v-v^{\prime}|(1+|v|+|v^{\prime}|)^{1-\gamma}}\,. 
\end{equation}
\end{lemma}

\begin{lemma}[$L^\infty$ Estimate]\label{Lmm_L_infty} Recall the definition of $h$ in \eqref{def_h}. Under the same assumptions as in Proposition \ref{Prop_g_k_g_b_k}, there holds
  \begin{equation}\label{h_L_infty}
    \sup_{t \in [0,T_0]} \| \eps^3 h_\m (t) \|_\infty \leq C \big( \| \eps^3 h_\m (0) \|_\infty + \sup_{t \in [0,T_0]} \|g_\m(t)\|_2 + \eps^6 \sup_{t \in [0,T_0]} \|\nu^{-1} \omega_\beta  \sqrt{\tfrac{\mu_\eps}{\mu_M}} S_\m (t)\|_\infty \big) \,.
  \end{equation}  
\end{lemma}
\begin{proof}
  Define $K_\omega g =  \omega_\beta K(\tfrac{g}{\omega_\beta}) $.  Based on the remainder equation \eqref{remainder_eqn}, we deduce that $h_{\m}$ satisfies
  \begin{equation}\label{3.2.2}
    \begin{cases}
          \partial_t h_{\m}+ \tfrac{1}{\eps} v \cdot\nabla_x h_{\m}+\tfrac{1}{\varepsilon^3 }\nu_\eps h_{\m}-\tfrac{1}{\varepsilon^3}K_\omega h_{\m}= \omega_\beta  \sqrt{\tfrac{\mu_\eps}{\mu_M}} S_{\m} \,,\\
          h_{\m}(t,x,v)=\int_{v^\prime \cdot n>0}\mathcal{R}(v,v^\prime )h_{\m}(t,x,v^\prime ) \d v^\prime \,, \quad (x,v)\in \Sigma_- \,,
    \end{cases}
    \end{equation}
  where we used the notation
  \begin{equation}\label{3.2.3}
  \mathcal{R}(v,v^\prime )=(1-\alpha)\delta(v^\prime -(v-2(v\cdot n)n))+\alpha\sqrt{2\pi} \sqrt{\tfrac{\mu_M(v^\prime )}{\mu_M(v)}}\mu(v)|v^\prime \cdot n|\frac{\omega_\beta(v)}{\omega_\beta(v^\prime )} \,.
  \end{equation}
  By integrating $\eqref{3.2.2}_1$ along the backward characteristics, we obtain 
  \begin{equation}\label{3.2.4}
    \begin{aligned}
        h_\m(t,x,v)=&\mathbbm{1}_{t_1\leq 0}\Bigg\{ \exp\{-\tfrac{1}{\eps^3}\int_{0}^{t}{\nu_\eps(\tau) \d \tau} \} h_\m(0,x-\tfrac{1}{\eps}tv,v) \\
        &+ \int_{0}^{t}  \exp\{- \tfrac{1}{\eps^3}\int_{s}^{t} \nu_\eps(\tau) \d \tau \}  \left[\tfrac{1}{\varepsilon^3}K_\omega  h_\m + \omega_\beta  \sqrt{\tfrac{\mu_\eps}{\mu_M}}S_\m\right](s,x+\tfrac{1}{\eps}(s-t)v,v) \d s  \Bigg\} \\
        &+ \mathbbm{1}_{t_1>0}\Bigg\{\exp\{-\tfrac{1}{\eps^3}\int_{t_1}^{t}{\nu_\eps(\tau) \d \tau} \} h_\m(t_1,x_1,v) \\
        &+ \int_{t_1}^{t} \exp\{- \tfrac{1}{\eps^3}\int_{s}^{t} \nu_\eps(\tau) \d \tau \} \left[\tfrac{1}{\varepsilon^3}K_\omega  h_\m + \omega_\beta \sqrt{\tfrac{\mu_\eps}{\mu_M}} S_\m\right](s,x+ \tfrac{1}{\eps}(s-t)v,v) \d s  \Bigg\} \,.
    \end{aligned}
    \end{equation}
  Using the boundary condition $\eqref{3.2.2}_2$, together with \eqref{3.2.4} again, it follows that
  \begin{equation}\label{3.2.5}
    h_{\m}(t,x,v)=\sum_{i=1}^{5}\mathcal{I}_{i} \,,
  \end{equation}
  with
  \begin{align*}
     \mathcal{I}_1 =&\mathbbm{1}_{t_1\leq 0}\exp\{-\tfrac{1}{\eps^3}\int_{0}^{t}{\nu_\eps(\tau) \d \tau} \} h_{\m}(0,x- \tfrac{1}{\eps}tv,v)\\
     &+\mathbbm{1}_{t_1>0}\exp\{-\tfrac{1}{\eps^3}\int_{0}^{t}{\nu_\eps(\tau) \d \tau} \} \int_{v_1\cdot n>0}\mathcal{R}(v,v_1)h_{\m}(0,x_1-\tfrac{1}{\eps}t_1v_1,v_1) \d v_1 \,, \\
     \mathcal{I}_2 =&\int_{\max\{0,t_1\}}^{t}\exp\{-\tfrac{1}{\eps^3}\int_{s}^{t}{\nu_\eps(\tau) \d \tau} \}(\omega_\beta \sqrt{\tfrac{\mu_\eps}{\mu_M}} S_\m)(s,x+ \tfrac{1}{\eps}(s-t)v,v) \d s\\
      &+\mathbbm{1}_{t_1>0}\int_{v_1\cdot n>0}\int_{0}^{t_1}\exp\{-\tfrac{1}{\eps^3}\int_{s}^{t}{\nu_\eps(\tau) \d \tau} \}\mathcal{R}(v,v_1) (\omega_\beta \sqrt{\tfrac{\mu_\eps}{\mu_M}} S_\m )(s,x_1+\tfrac{1}{\eps}(s-t_1)v_1,v_1)\d s \d v_1 \,, \\
     \mathcal{I}_3  =& \tfrac{1}{\varepsilon^3}\int_{\max\{0,t_1\}}^{t}\exp\{-\tfrac{1}{\eps^3}\int_{s}^{t}{\nu_\eps(\tau) \d \tau} \}K^m_\omega h_{\m}(s,x+ \tfrac{1}{\eps}(s-t)v,v)\d s   \\
       &+ \tfrac{1}{\varepsilon^3} \mathbbm{1}_{t_1>0}\int_{v_1\cdot n>0}\int_{0}^{t_1}\mathcal{R}(v,v_1)\exp\{-\tfrac{1}{\eps^3}\int_{s}^{t}{\nu_\eps(\tau) \d \tau} \} K^m_\omega h_\m(s,x_1+ \tfrac{1}{\eps}(s-t_1)v_1,v_1)\d s \d v_1 \,,\\
     \mathcal{I}_4 =& \tfrac{1}{\varepsilon^3}\int_{\max\{0,t_1\}}^{t}\exp\{-\tfrac{1}{\eps^3}\int_{s}^{t}{\nu_\eps(\tau) \d \tau} \}K^c_\omega h_{\m}(s,x+ \tfrac{1}{\eps}(s-t)v,v)\d s \,, 
   \end{align*}
  and
  \begin{equation*}
    \mathcal{I}_5=\tfrac{1}{\varepsilon^3} \mathbbm{1}_{t_1>0}\int_{v_1\cdot n>0}\int_{0}^{t_1}\mathcal{R}(v,v_1)\exp\{-\tfrac{1}{\eps^3}\int_{s}^{t}{\nu_\eps(\tau) \d \tau} \} K^c_\omega h_\m(s,x_1+ \tfrac{1}{\eps}(s-t_1)v_1,v_1)\d s \d v_1 \,.
  \end{equation*}
  
  First, we observe that
  \begin{equation*}
    \nu_\eps(\tau) \sim \int |v - v_1|^\gamma \mu  (v_1)\d v_1 \sim \nu = \l v \r^\gamma\,,
  \end{equation*}
  it thereby follows
  \begin{equation*}
    \int_{0}^{t} \exp\{- \tfrac{1}{\eps^3}\int_{s}^{t} \nu_\eps(\tau) \d \tau \} \nu_\eps(\tau) \d s\leq c \int_{0}^{t} \exp \{- \tfrac{c \nu (t-s)}{\eps^3} \} \nu \d s = O(\eps^3)\,.
  \end{equation*}
  Then, it is straightforward to see that
  \begin{equation*} 
    |\mathcal{I}_1|\leq C\|h_{\m}(0)\|_\infty\,,
  \end{equation*} 
 and 
  \begin{equation*}
  | \I_2 | \leq C \eps^3\sup_{s \in [0,t]} \|\nu^{-1} \omega_\beta   \sqrt{\tfrac{\mu_\eps}{\mu_M}} S_{\m}(s) \|_\infty  \,.
  \end{equation*} 
  By \eqref{est_K_m}, it follows
  \begin{equation*}
  | \I_3 | \leq C m^{3 + \gamma} \sup_{s \in [0,t]} \| h_\m(s)\|_\infty  \,. 
  \end{equation*}
  In what follows we present a detailed computation for the non-local term $\mathcal{I}_5$. The estimates for $\mathcal{I}_4$ are similar and much easier and will be omitted for brevity.
 
 We now denote $(t_0^\prime ,x_0^\prime ,v_0^\prime ) = (s,X_{cl}(s;t_1,x_1,v_1),v^\prime )$, and define a new backward-time cycle as follows:
  $$(t_0^\prime ,x_1^\prime ,v_1^\prime )=(t^\prime _1-t_b(x^\prime _0,v_0^\prime ),x_b(x^\prime _0,v^\prime _0),v^\prime _1 ) \,,$$
  for $v_1^\prime \in \mathcal{V^\prime }=\{v_1^\prime :v_1^\prime \cdot n>0\}$. Let $k_\omega(v, v^\prime)$ be the corresponding kernel associated with $K^c_\omega$. We then employ \eqref{3.2.4} again to obtain
  \begin{align*}
      \mathcal{I}_5=& \tfrac{1}{\eps^3}\mathbbm{1}_{t_1>0}\int_{v_1\cdot n>0}\int_{0}^{t_1}\mathcal{R}(v,v_1)\exp\{-\tfrac{1}{\eps^3}\int_{s}^{t}{\nu_\eps(\tau) \d \tau} \} \int_{\mathbb{R}^3}k_\omega  (v_1,v^\prime )\\
      &\qquad\qquad \times h_{\m}(s,x_1+ \tfrac{1}{\eps}(s-t_1)v_1,v^\prime )\d s \d v^\prime \d v_1:=\sum_{i=1}^{4}\mathcal{J}_i \,,
  \end{align*}
  with
  \begin{align*}
        \mathcal{J}_1 = & \tfrac{1}{\eps^3} \mathbbm{1}_{t_1>0} \int_{v_1\cdot n>0} \int_{0}^{t_1} \mathcal{R}(v,v_1) \exp\{-\tfrac{1}{\eps^3}\int_{s}^{t}{\nu_\eps(\tau) \d \tau} \}  \int_{\mathbb{R}^3} k_\omega  (v_1,v^\prime )\\
        &\times \left\{ \mathbbm{1}_{t_1^\prime \leq 0} \exp\{-\tfrac{1}{\eps^3}\int_{0}^{t^\prime _0}{\nu_\eps(\tau) \d \tau} \} h_{\m}(0,x^\prime _0- \tfrac{1}{\eps}t^\prime _0v^\prime _0,v^\prime _0) \right. \\
        &\qquad + \left. \mathbbm{1}_{t^\prime _1>0} \exp\{-\tfrac{1}{\eps^3}\int_{t_0^\prime }^{t_1^\prime }{\nu_\eps(\tau) \d \tau} \} \int_{v_1^\prime \cdot n>0} \mathcal{R}(v_0^\prime ,v_1^\prime ) h_{\m}(0,x_1^\prime - \tfrac{1}{\eps} t_1^\prime v_1^\prime ,v_1^\prime ) \d v_1^\prime  \right\} \d s  \d v^\prime  \d v_1 \,,
    \end{align*}

     \begin{align*}
        \mathcal{J}_2 = &\tfrac{1}{\eps^3} \mathbbm{1}_{t_1>0}\int_{v_1\cdot n>0}\int_{0}^{t_1}\mathcal{R}(v,v_1)\exp\{-\tfrac{1}{\eps^3}\int_{s}^{t}{\nu_\eps(\tau) \d \tau} \} \int_{\mathbb{R}^3}k_\omega  (v_1,v^\prime ) \\
        & \times \int_{\max\{0,t_1^\prime \}}^{t_0^\prime } \exp\{-\tfrac{1}{\eps^3} \int_{s^\prime }^{t_0^\prime }{\nu_\eps(\tau) \d \tau} \} \omega_\beta \sqrt{\tfrac{\mu_\eps}{\mu_M}} S_\m \left(s^\prime ,x^\prime _0+  \tfrac{1}{\eps}(s^\prime -t^\prime _0)v_0^\prime ,v^\prime _0\right) \d s ^\prime  \d s  \d v^\prime  \d v_1 \\
        & +\tfrac{1}{\eps^3} \mathbbm{1}_{t_1>0,t_1^\prime >0}\int_{v_1\cdot n>0}\int_{0}^{t_1}\mathcal{R}(v,v_1)\exp\{-\tfrac{1}{\eps^3}\int_{s}^{t}{\nu_\eps(\tau) \d \tau} \} \int_{\mathbb{R}^3}k_\omega  (v_1,v^\prime ) \\
        & \times \int_{v_1^\prime \cdot n>0}\int_{0}^{t_1^\prime }\mathcal{R}(v_0^\prime ,v_1^\prime ) \exp\{-\tfrac{1}{\eps^3} \int_{s^\prime }^{t_0^\prime }{\nu_\eps(\tau) \d \tau} \} \omega_\beta \sqrt{\tfrac{\mu_\eps}{\mu_M}} S_\m(s^\prime ,x_1^\prime +  \tfrac{1}{\eps} (s^\prime -t^\prime _1)v_1^\prime ,v_1^\prime ) \d s ^\prime  \d s  \d v^\prime  \d v_1^\prime \d v_1 \,, 
     \end{align*}
  \begin{align*}
        \mathcal{J}_3 = & \tfrac{1}{\eps^6} \mathbbm{1}_{t_1>0}\int_{v_1\cdot n>0}\int_{0}^{t_1}\mathcal{R}(v,v_1)\exp\{-\tfrac{1}{\eps^3}\int_{s}^{t}{\nu_\eps(\tau) \d \tau} \} \int_{\mathbb{R}^3}k_\omega  (v_1,v^\prime ) \\
        & \times \int_{\max\{0,t_1^\prime \}}^{t_0^\prime }  \exp\{-\tfrac{1}{\eps^3} \int_{s^\prime }^{t_0^\prime }{\nu_\eps(\tau) \d \tau} \} K^m_\omega h_\m \left(s^\prime ,x^\prime _0+ \tfrac{1}{\eps}(s^\prime -t^\prime _0)v_0^\prime ,v^\prime _0\right) \d s ^\prime  \d s  \d v^\prime  \d v_1 \\
        & +\tfrac{1}{\eps^6} \mathbbm{1}_{t_1>0,t_1^\prime >0}\int_{v_1\cdot n>0}\int_{0}^{t_1}\mathcal{R}(v,v_1)\exp\{-\tfrac{1}{\eps^3}\int_{s}^{t}{\nu_\eps(\tau) \d \tau} \} \int_{\mathbb{R}^3}k_\omega  (v_1,v^\prime ) \\
        & \times \int_{v_1^\prime \cdot n>0}\int_{0}^{t_1^\prime }\mathcal{R}(v_0^\prime ,v_1^\prime ) \exp\{-\tfrac{1}{\eps^3} \int_{s^\prime }^{t_0^\prime }{\nu_\eps(\tau) \d \tau} \} K^m_\omega h_\m(s^\prime ,x_1^\prime + \tfrac{1}{\eps}(s^\prime -t^\prime _1)v_1^\prime ,v_1^\prime ) \d s ^\prime  \d s  \d v^\prime  \d v_1^\prime \d v_1 \,,
    \end{align*}
  \begin{multline*}
        \mathcal{J}_4 =\tfrac{1}{\varepsilon^6} \mathbbm{1}_{t_1>0}\int_{v_1\cdot n>0}\int_{0}^{t_1}\mathcal{R}(v,v_1)\exp\{-\tfrac{1}{\eps^3}\int_{s}^{t}{\nu_\eps(\tau) \d \tau} \}\int_{\mathbb{R}^3}k_\omega  (v_1,v^\prime )\\
         \times\int_{\max\{0,t_1^\prime \}}^{t_0^\prime }\exp\{-\tfrac{1}{\eps^3}  \int_{s^\prime }^{t_0^\prime }{\nu_\eps(\tau) \d \tau} \}
        \times K^c_\omega  h_{\m}(s^\prime ,x^\prime _0+ \tfrac{1}{\eps}(s^\prime -t^\prime _0)v_0^\prime ,v^\prime _0)\d s ^\prime \d s \d v^\prime \d v_1 \,,
  \end{multline*}
  and
    \begin{align*}
        \mathcal{J}_5=\tfrac{1}{\varepsilon^6}\mathbbm{1}_{t_1>0,t_1^\prime >0}\int_{v_1\cdot n>0}\int_{0}^{t_1}\mathcal{R}(v,v_1)\exp\{-\tfrac{1}{\eps^3}\int_{s}^{t}{\nu_\eps(\tau) \d \tau} \} \int_{\mathbb{R}^3}k_\omega  (v_1,v^\prime ) \int_{v_1^\prime \cdot n>0}\int_{0}^{t_1^\prime }\\
        \exp\{-\tfrac{1}{\eps^3}\int_{s^\prime }^{t_0^\prime }{\nu_\eps(\tau) \d \tau} \}
        \times \mathcal{R}(v_0^\prime ,v_1^\prime )K^c_\omega  h_\m(s^\prime ,x_1^\prime + \tfrac{1}{\eps}(s^\prime -t^\prime _1)v_1^\prime ,v_1^\prime )\d s ^\prime \d s \d v^\prime \d v_1^\prime \d v_1 \,.
      \end{align*}
  Recalling \eqref{est_K_c}, we have 
  \begin{equation}\label{est_k_omega}
    k_\omega(v, v^\prime) \leq \tfrac{C \exp\{- \tilde{c} |v - v^\prime |^2\}}{|v - v^\prime| (1 + |v| + |v^\prime|)^{1-\gamma}} \,,
  \end{equation}
   which implies that
  \begin{equation*}
    \int_{\R^3} |k_\omega (v,v^\prime) \d v^\prime \leq C \nu(v)\,.
  \end{equation*}
  Consequently, we have
  \begin{equation}\label{3.2.17}
    |\mathcal{J}_1|\leq \|h_{\m}(0)\|_\infty\,,
  \end{equation}
  and 
  \begin{equation}\label{3.2.18}
    |\mathcal{J}_2|\leq C \eps^3 \sup_{s \in [0,t]}\| \nu^{-1} \omega_\beta \sqrt{\tfrac{\mu_\eps}{\mu_M}} S_\m (s) \|_\infty \,.
  \end{equation}
  By \eqref{est_K_m},
  \begin{equation}
\|\J_4 | \leq C m^{3 + \gamma} \sup_{s \in [0,t]} \| h_\m (s) \|_\infty \,.
  \end{equation}
  Next, we only compute $\J_5$ because the estimates for $\mathcal{J}_4$ are similar and easier. Note that $K^c_\omega h(v^\prime_1) = \int_{\R^3} k_\omega (v_1^\prime , v^{\prime \prime}) h(v^{\prime \prime}) \d v^{\prime \prime}$. We divide our computations into the following three cases:
  
  \vspace{4pt}\noindent
  \textbf{Case 1.} $|v_1|\geq N_0$ or $ |v_1^\prime |\geq N_0$ with $N_0$ suitable large. It follows from \eqref{est_K_c} that
  \begin{equation*} 
    \int_{\mathbb{R}^3}|k_\omega  (v_1,v^\prime )|\d v^\prime \leq \frac{C}{N_0} \,,
  \end{equation*}
  or
  \begin{equation*} 
    \int_{\mathbb{R}^3}|k_\omega  (v_1^\prime ,v^{\prime \prime}  )|\d v^{\prime \prime}  \leq \frac{C}{N_0} \,,
  \end{equation*}
  which together with \eqref{3.2.3}, yields that
  \begin{equation*}
    |\J_5|\leq \frac{C}{N_0}\sup_{s \in [0,t]}\|h_{\m}(s)\|_\infty \,.
  \end{equation*}
  
  \noindent
  \textbf{Case 2.} $|v_1|\leq N_0$ and $|v^\prime |\geq 2N_0$ or $|v_1^\prime |\leq N_0$ and $|v^{\prime \prime}  |\geq 2N_0$. Note that either $|v^\prime -v_1|\geq N_0$ or $|v_1^\prime -v^{\prime \prime}  |\geq N_0$ holds, and thus either of the following holds
  \begin{equation*} 
   | k_\omega  (v_1,v^\prime )|\leq C e^{-\frac{\eta N_0^2}{8}}| k_\omega  (v_1,v^\prime )e^{\frac{\eta|v_1-v^\prime |^2}{8}}|\,,
  \end{equation*}
  or
  \begin{equation*} 
    |k_\omega  (v_1^\prime ,v^{\prime \prime}  )|\leq C e^{-\frac{\eta N_0^2}{8}} |k_\omega  (v_1^\prime ,v^{\prime \prime}  ) e^{\frac{\eta|v_1^\prime -v^{\prime \prime}  |^2}{8}}|\,,
  \end{equation*}
  for some small $\eta$. Then, applying \eqref{est_k_omega}, we obtain
  \begin{equation*} 
    |\J_5|\leq C_\eta e^{-\frac{\eta}{8}N_0^2}\sup_{s \in [0,t]}\|h_\m(s)\|_\infty \,.
  \end{equation*}

  \noindent
  \textbf{Case 3a.} $|v_1|\leq N_0,|v^\prime |\leq 2N_0$ and $|v_1^\prime |\leq N_0,|v^{\prime \prime}  |\leq 2N_0$, $s^\prime \geq t_1^\prime -\kappa_*\varepsilon^3$,
  \begin{equation*}
    |\J_5|\leq C_{N_0}\kappa_*\sup_{s \in [0,t]}\|h_{\m}(s)\|_\infty \,.
  \end{equation*}

  \noindent
  \textbf{Case 3b.} $|v_1|\leq N_0,|v^\prime |\leq 2N_0$ and $|v_1^\prime |\leq N_0,|v^{\prime \prime}  |\leq 2N_0$, $s^\prime \leq t_1^\prime -\kappa_*\varepsilon^3$,
  \begin{equation}\label{3.2.27}
    \begin{aligned}
        \J_5 = & \tfrac{1}{\varepsilon^6} \mathbbm{1}_{t_1>0,t_1^\prime >0} \int_{{v_1\cdot n>0,|v_1|\leq N_0}} \int_{v_1^\prime \cdot n>0,|v_1^\prime |\leq N_0} \int_{0}^{t_1} \int_{0}^{t_1^\prime -\varepsilon k_*} \exp\{-\tfrac{1}{\eps^3}\int_{s}^{t}{\nu_\eps(\tau) \d \tau} \} \\
        &\times \exp\{-\tfrac{1}{\eps^3}\int_{s^\prime }^{t_0^\prime }{\nu_\eps(\tau) \d \tau} \}  \int_{\substack{|v^\prime |\leq 2N_0}} \int_{\substack{|v^{\prime \prime}  |\leq 2N_0}} k_\omega  (v_1,v^\prime ) k_\omega  (v_1^\prime ,v^{\prime \prime}  ) \mathcal{R}(v,v_1) \mathcal{R}(v_0^\prime ,v_1^\prime ) \\
        & \times h_{\m}(s^\prime ,x_1^\prime + \tfrac{1}{\eps}(s^\prime -t_1^\prime )v_1^\prime ,v^{\prime \prime}  ) \, \d s ^\prime  \, \d s  \, \d v^\prime  \, \d v_1^\prime  \, \d v_1 \, \d v^{\prime \prime}  \,.
    \end{aligned}
    \end{equation}
  
  Based on \eqref{est_K_c}, $k_\omega  (v_1,v^\prime $) has a possible integrable singularity of $\frac{1}{|v_1-v^\prime |}$. We define $k_{N_0, \omega}$ with compact support 
  \begin{equation*}
    k_{N_0,\omega}(p,v^\prime )=\mathbbm{1}_{|p-v^\prime |\geq \frac{1}{m},|v^\prime |\leq m}k_\omega  (p,v^\prime )
  \end{equation*}
such that
  \begin{equation*} 
    \sup_{|p|\leq 3N_0}\int_{|v^\prime |\leq 3N_0}|k_{N_0,\omega}(p,v^\prime )-k_\omega  (p,v^\prime )|\d v^\prime \leq \frac{1}{N_0} \,.
  \end{equation*}
  We split
  \begin{equation*}
  \begin{aligned}
    &k_\omega  (v_1,v^\prime )k_\omega  (v_1^\prime ,v^{\prime \prime}  )\\
     =&\{k_\omega  (v_1,v^\prime )-k_{N_0,\omega}(v_1,v^\prime )\}k_{w}(v_1^\prime ,v^{\prime \prime}  )  +\{k_\omega  (v_1^\prime ,v^{\prime \prime}  )-k_{N_0,\omega}(v_1^\prime ,v^{\prime \prime}  )\}k_{N_0,\omega}(v_1,v^\prime )\\
     &+k_{N_0,\omega}(v_1,v^\prime )k_{N_0,\omega}(v^\prime ,v^{\prime \prime}  ) \,.
  \end{aligned}
  \end{equation*}
  The first two differences leads to a small contribution to $\J_5$,
  \begin{equation*} 
    \frac{C}{N_0}\sup_{s \in [0,t]}\|h_{\m}(s)\|_\infty \,.
  \end{equation*}
  For the main contribution of $k_{N_0,\omega}(v_1,v^\prime )k_{N_0,\omega}(v_1^\prime ,v^{\prime \prime}  )$. Using the boundary condition \eqref{3.2.3}, \eqref{3.2.27} reduces to the specular reflection part and the diffusion reflection part
    \begin{equation}
      \begin{aligned}
          & \frac{1-\alpha}{\varepsilon^6} \mathbbm{1}_{t_1>0,t_1^\prime >0} \int_{\substack{v_1\cdot n>0, \\ |v_1|\leq N_0}} \int_{0}^{t_1} \int_{0}^{t_1^\prime -\varepsilon^3 k_*} \int_{|v^\prime |\leq 2N_0} \int_{|v^{\prime \prime}  |\leq 2N_0} \exp\{-\tfrac{1}{\eps^3}\int_{s}^{t}{\nu_\eps(\tau) \d \tau} \} \\
          & \quad \times \exp\{-\tfrac{1}{\eps^3}\int_{s^\prime }^{t_0^\prime }{\nu_\eps(\tau) \d \tau} \}   k_{N_0,\omega}(v_1,v^\prime ) k_{N_0,\omega}(R_x v^\prime ,v^{\prime \prime}  ) \mathcal{R}(v,v_1)  \\
          & \quad \times h_{\m}(s^\prime ,x_1^\prime +\tfrac{1}{\eps}(s^\prime -t_1^\prime )R_x v^\prime ,v^{\prime \prime}  )  \d s ^\prime   \d s   \d v^\prime      \d v_1  \d v^{\prime \prime}   \\
          & + \frac{\alpha}{\varepsilon^6} \mathbbm{1}_{t_1>0,t_1^\prime >0} \int_{\substack{v_1\cdot n>0, \\ |v_1|\leq N_0}} \int_{\substack{v_1^\prime \cdot n>0, \\ |v_1^\prime |\leq N_0}} \int_{0}^{t_1} \int_{0}^{t_1^\prime -\varepsilon^3 k_*} \int_{|v^\prime |\leq 2N_0} \int_{|v^{\prime \prime}  |\leq 2N_0} \exp\{-\tfrac{1}{\eps^3}\int_{s}^{t}{\nu_\eps(\tau) \d \tau} \} \\
          & \quad \times \exp\{-\tfrac{1}{\eps^3}\int_{s^\prime }^{t_0^\prime }{\nu_\eps(\tau) \d \tau} \}   k_{N_0,\omega}(v_1,v^\prime ) k_{N_0,\omega}(v_1^\prime ,v^{\prime \prime}  ) \mathcal{R}(v,v_1) \sqrt{2\pi} \sqrt{ \tfrac{\mu_M(v^\prime _1)}{\mu_M(v^\prime _0)}} \mu(v^\prime) \tfrac{\omega_\beta(v^\prime)}{\omega_\beta(v^\prime_1)} |v_1^\prime \cdot n| \\
          & \quad \times h_{\m}(s^\prime ,x_1^\prime + \tfrac{1}{\eps}(s^\prime -t_1^\prime )v_1^\prime ,v^{\prime \prime}  ) \, \d s ^\prime  \, \d s  \, \d v^\prime  \, \d v_1^\prime  \, \d v_1 \, \d v^{\prime \prime}   \\
          & := \mathcal{J}_{5,1} + \mathcal{J}_{5,2} \,.
      \end{aligned}
      \end{equation}
  Recall $v_0^\prime =v^\prime $. To estimate $\mathcal{J}_{5,1}$, we make a change of variable $v^\prime \to y=x_1^\prime + \tfrac{1}{\eps}(s^\prime -t_1^\prime )R_xv^\prime $. One has
    \begin{equation*}
      \mid \det\left(\tfrac{\partial y}{\partial v^\prime }\right)\mid \geq (\varepsilon^2 k_*)^3>0 \quad \text{for} \quad \! s^\prime \in[0,t_1^\prime -k_*\varepsilon^3] \,,
    \end{equation*}
    which yields that
    \begin{equation*}
    \begin{aligned}
        &\int_{|v^\prime |\leq 2 N_0}|h_{\m}(s^\prime ,x_1^\prime +(s^\prime -t_1^\prime )R_xv^\prime ,v^{\prime \prime}  )|\d v^\prime \\
        \leq& C_{N_0}\left(\int_{|v^\prime |\leq 2N_0}\mathbbm{1}_{\Omega}\left(x_1^\prime +(s_1^\prime -t_1^\prime )R_xv^\prime \right)|h_{\m}(s^\prime ,x_1^\prime +(s^\prime -t_1^\prime )R_xv^\prime ,v^{\prime \prime}  )|^2\d v^\prime \right)^{\frac{1}{2}}\\
        \leq& \tfrac{C_{N_0}}{(\varepsilon^2 k_*)^{\frac{3}{2}}}\left(\int_{\Omega}|h_{\m}(s^\prime ,y,v^{\prime \prime}  )|^2dy\right)^\frac{1}{2} \,.
    \end{aligned}
    \end{equation*}
    Together with the definition of $h_{\m}$, $\mathcal{J}_{5,1}$ can be bounded by
    $$\tfrac{C_{N_0}}{(\varepsilon^2 k_*)^{\frac{3}{2}}}\sup_{s\in [0,t]}\|g_\m (s)\|_2 \,.$$
    Similarly, by performing the change of variable $v^\prime _1\to y_1=x_1^\prime + \tfrac{1}{\eps}(s^\prime -t_1^\prime )v_1^\prime $ and applying analogous arguments as above,  we can bound $\mathcal{J}_{5,2}$ by $\frac{C_{N_0}}{(\varepsilon^2 k_*)^{\frac{3}{2}}}\sup_{s\in [0,t]}\|g_\m (s)\|_2$. 
  Consequently, it holds that
  \begin{equation}\label{3.2.30}
    \J_5\leq \tfrac{C}{N_0}\sup_{s \in [0,t]}\|h_{\m}(s)\|_\infty+\tfrac{C_{N_0,k_*}}{\varepsilon^3}\sup_{s \in [0,t]}\|g_\m (s)\|_2 \,.
  \end{equation}

  Collecting all the above estimates, one has
  \begin{equation}\label{3.2.31}
  \begin{aligned}
     \sup_{s \in [0,t]} \|h_{\m}(s)\|_\infty\leq &C\|h_{\m}(0)\|_\infty+ C \eps^3 \sup_{s \in [0,t]} \|\nu^{-1} \omega_\beta \sqrt{\tfrac{\mu_\eps}{\mu_M}} S_\m (s)\|_\infty\\
     &+(C m^{3 + \gamma}+\tfrac{C_m}{N_0}+\kappa_* C_{N_0,m}) \sup_{s \in [0,t]}\|h_{\m}(s)\|_\infty 
      +\tfrac{C_{N_0,k_*,m}}{\varepsilon^3}\sup_{s \in [0,t]}\|g_{\m}(s)\|_2 \,,
  \end{aligned}
  \end{equation}
  which concludes Lemma \ref{Lmm_L_infty} by choosing $\varepsilon>0, m > 0, \kappa_*> 0$ sufficiently small, and $N_0 >0$ sufficiently large.
  
\end{proof}

\begin{corollary}\label{Coro_L_infty} Under the same assumptions as in Proposition \ref{Prop_g_k_g_b_k}, it follows that
  \begin{equation}
  \sup_{t \in [0, T_0]} \eps^6 \|h (t)\|^2_Y \leq C \eps^6 \|h (0)\|^2_Y + C \sup_{t \in [0,T_0]} \|g(t) \|^2_X +C \eps^{12} \,.
  \end{equation}
\end{corollary}
\begin{proof}
  It follows from Lemma \ref{Lmm_L_infty} that 
    \begin{equation}
  \sup_{t \in [0, T_0]} \eps^6 \|h (t)\|^2_Y \leq C \eps^6 \|h (0)\|^2_Y + C \sup_{t \in [0, T_0]} \eps^{12} \| \nu^{-1} \omega_\beta \sqrt{\tfrac{\mu_\eps}{\mu_M}} S (t) \|_Y^2+ C \sup_{t \in [0,T_0]}\|g(t) \|^2_X   \,.
  \end{equation}
  It follows from the same proof as in \cite[Lemma 10]{Guo-2010-ARMA} that
    \begin{equation}
    |\nu^{-1} \tfrac{\omega_\beta}{\sqrt{\mu_M}} B(\tfrac{h_\n \sqrt{\mu_M}}{\omega_\beta}, \tfrac{h_\j \sqrt{\mu_M}}{\omega_\beta}) | \leq C  \| h_\n \|_\infty\| h_\j \|_\infty \,.
  \end{equation}
  Hence,
  \begin{equation}
    \begin{aligned}
      \eps^{12} \| \nu^{-1} \omega_\beta \sqrt{\tfrac{\mu_\eps}{\mu_M}} S_{1} \|_Y^2 =& \eps^{16} \sum_{\m \in \mathbb{N}^2_0, \n + \j = \m}\tfrac{\sigma(t)^{2 |\m|}}{(\m !)^2} (C_\m^\n)^2 \| \nu^{-1}  {\tfrac{\omega_\beta }{\sqrt{\mu_M}}} B (\tfrac{h_\n \sqrt{\mu_M}}{\omega_\beta},\tfrac{h_\j \sqrt{\mu_M}}{\omega_\beta}) \|_\infty^2 \\
       \leq& \eps^{16} \sum_{\m \in \mathbb{N}^2_0, \n + \j = \m}\tfrac{\sigma(t)^{2 |\m|}}{(\m !)^2} (C_\m^\n)^2  \| h_\n \|_\infty^2 \| h_\j \|_\infty^2 \\
       \leq& \eps^{16} \| h\|_Y^4 \,.
    \end{aligned}
  \end{equation}
  Similarly, it holds by using Proposition \ref{Prop_g_k_g_b_k} that
   \begin{equation}
    \begin{aligned}
     &\eps^{12} \|\nu^{-1} \omega_\beta \sqrt{\tfrac{\mu_\eps}{\mu_M}} S_{2 } \|_Y^2 \\
     =&\eps^{12} \sum_{\m \in \mathbb{N}^2_0, \n + \j = \m}\tfrac{\sigma(t)^{2 |\m|}}{(\m !)^2}(C_\m^\n)^2 \|\sum_{i=1}^{13} \eps^{\frac{i-4}{2}} \nu^{-1}\tfrac{\omega_\beta}{\sqrt{\mu_M}} \big\{ B(\sqrt{\mu}(g_i + g^b_i + g^{bb}_i)_\j, \tfrac{h_\n \sqrt{\mu_M}}{\omega_\beta}) \\
     &\qquad\qquad\qquad\qquad\qquad\qquad+ B(\tfrac{h_\n \sqrt{\mu_M}}{\omega_\beta}, \sqrt{\mu}(g_i + g^b_i + g^{bb}_i))_\j \big\} \|_\infty^2 \\
      \leq & C \eps^9 \| h  \|_Y^2\|\tfrac{\omega_\beta \sqrt{\mu}}{\sqrt{\mu_M}} \sum_{i=1}^{13}(g_i + g^b_i + g^{bb}_i) \|_Y^2 \leq C \eps^9 \| h  \|_Y^2 \,,
    \end{aligned} 
  \end{equation} 
and 
\begin{equation}
    \begin{aligned}
     &\eps^{12} \|\nu^{-1} \omega_\beta \sqrt{\tfrac{\mu_\eps}{\mu_M}} S_{3 } \|_Y^2 \\
     =&\eps^{12} \sum_{\m \in \mathbb{N}^2_0, \n + \j = \m}\tfrac{\sigma(t)^{2 |\m|}}{(\m !)^2}(C_\m^\n)^2 \| \tfrac{1}{\eps} \nu^{-1}\tfrac{\omega_\beta}{\sqrt{\mu_M}} \big\{ B(r_\j, \tfrac{h_\n \sqrt{\mu_M}}{\omega_\beta})  + B(\tfrac{h_\n \sqrt{\mu_M}}{\omega_\beta}, r_\j) \big\} \|_\infty^2 \\
      \leq & C \eps^{10} \| h  \|_Y^2\|\tfrac{\omega_\beta }{\sqrt{\mu_M}} r \|_Y^2 \leq C \eps^{10} \| h  \|_Y^2 \,.
    \end{aligned} 
  \end{equation} 
  Proposition \ref{Prop_g_k_g_b_k} further implies that
\begin{equation}
    \begin{aligned}
     &\eps^{12} \|\nu^{-1} \omega_\beta \sqrt{\tfrac{\mu_\eps}{\mu_M}} S_{4} \|_Y^2  
      =& C \eps^{12} \| \nu^{-1}\omega_\beta \sqrt{\tfrac{\mu}{\mu_M}}(R_\eps + R^b_\eps + R^{bb}_\eps) \|_Y^2  \leq C \eps^{12} \,.
    \end{aligned} 
  \end{equation}
As a result, we have
    \begin{equation}
  \sup_{t \in [0, T_0]} \eps^6 \|h (t)\|^2_Y \leq C \eps^6 \|h (0)\|^2_Y + C \sup_{t \in [0, T_0]}\eps^{16}\|h \|_Y^4 + C\sup_{t \in [0, T_0]} \eps^9 \|h (t)\|^2_Y  + C \sup_{t \in [0,T]} \|g(t) \|^2_X +C \eps^{12} \,.
  \end{equation}
The proof of the corollary is finished.
\end{proof}

\subsection{Proof of Theorem \ref{Main_theorem}} It follows from Lemma \ref{Lmm_L_2} and Lemma \ref{Lmm_L_infty} that
  \begin{equation}
  \begin{aligned}
    & \tfrac{1}{2}\tfrac{\d}{\d t} \|g \|_X^2 + \lambda \| g\|^2_{X^\frac{1}{2}}+ \tfrac{(2 \alpha - \alpha^2)}{2 \eps} \sum_{\m \in \mathbb{N}_0^2} \tfrac{\sigma(t)^{2 |\m|}}{(\m !)^2}\int_{\Sigma_+}(v\cdot n)\left|(I-P_\gamma)g_\m \right|^2d\bar{x}\d v  + \tfrac{\tilde{c}_0}{\eps^3} \|\nu^{\frac{1}{2}} \P^\perp_\eps g \|_{X}^2 \\
     \leq & \eps \tfrac{\d}{\d t} \sum_{\m \in \mathbb{N}_0^2} \tfrac{\sigma(t)^{2 |\m|}}{(\m !)^2}  \mathbb{G}_\m(t)+  C(\eps^7 \| h(0)\|_Y^2 + \eps \sup_{s \in [0,t]} \|g (s)\|^2_X) \| g\|_X^2 \\
    & + C \| g \|_{X}^2 + C + (C_1 + 2 C \eps \lambda ) \| g \|_{X^\frac{1}{2}}^2 \,.
  \end{aligned}
  \end{equation}
We now  choose $\lambda = 2 C_1$, $T_1 >0$ and $\eps_0$ small so that
\begin{equation*}
  \sigma (t) \geq 1 \quad \textrm{for} \quad \! t \in [0,T_1], \quad  4 C_1 T_1 \leq 1\,.
\end{equation*}
Let $T = \min (T_0,T_1)$. Note that $\|\mathbb{G}_\m \| \leq C \| g_\m \|^2_2$. Then by Lemma \ref{Lmm_L_infty} and the Gr\"onwall inequality, we have for $t \in[0,T]$,
\begin{equation}
\begin{aligned}
    \| g (t)\|_X^2   \leq C(\| g(0)\|_{X^2} +1)e^{C t ( \eps \sup_{s \in[0,T]}\| g (s)\|_X^2 + \eps^7\| h(0) \|_Y^2 + 1)} \,.
\end{aligned}
\end{equation}
For bounded $\sup_{s \in[0,T]}\| g (s)\|_X^2$  and for $\eps$ small, using the Taylor expansion of the exponential function in the above inequality, we have
\begin{equation}
\begin{aligned}
     \| g (t)\|_X^2 \leq C_T(\|g(0)\|_X^2 + 1) ( \eps \sup_{s \in[0,T]}\| g (s)\|_X^2 + \eps^7 \| h(0) \|_Y^2 + 1)\,.
\end{aligned}
\end{equation}
Then by letting $\eps$ small, we conclude that 
\begin{equation*}
       \| g (t)\|_X^2 \leq C_T(\|g(0)\|_X^2 + 1) ( \eps^7 \| h(0) \|_Y^2 + 1) \,.
\end{equation*}
This proves our main theorem.

\section{Formal derivation for general $q>1$}\label{Sec_general_q}
We consider the Boltzmann equation under the general scaling parameter $q >1$.
\begin{equation}
  \eps \p_t F_\eps + v\cdot \nabla_x F_\eps = \tfrac{1}{\eps^{q}} B(F_\eps, F_\eps)\,.
\end{equation}
We take $F_\eps = \mu + \eps \sqrt{\mu}g_\eps$, and the equation fo $g_\eps$ is
\begin{equation}
  \eps \p_t g_\eps + v\cdot \nabla_x g_\eps + \tfrac{1}{\eps^q} \L g= \tfrac{1}{\eps^{q-1}} \Gamma(g_\eps, g_\eps)\,.
\end{equation}
The formal Hilbert expansion of $g_\eps$ must account for all powers of $\eps$ appearing in the equation. Specifically, the asymptotic structure is determined by the characteristic thicknesses of the boundary layers: the Prandtl layer has thickness $O(\eps^{\frac{q-1}{2}})$, and the Knudsen layer has thickness $O(\eps^q)$. To consistently capture these multiple scales, we propose the following multi-scale ansatz:
\begin{equation}
  g_\eps =\sum_{m,n,j\geq 0} \eps^{m+ \frac{q-1}{2} n + q j}G_{mnj}\,,
\end{equation}
where $G_{mnj} = g_{mnj}+g^b_{mnj}+ g^{bb}_{mnj}$, representing the interior solution, the Prandtl boundary layer, and the kinetic (Knudsen) boundary layer, respectively. It is crucial to observe that 
$$\eps^{q j } = \eps^{\frac{q - 1}{2}\cdot 2 j + j}  \,,$$ 
which can be absorbed into the indices $m$ and $n$. Thus, the ansatz simplifies to 
\begin{equation}
  g_\eps = \sum_{m,n\geq 0}\eps^{m  + \frac{q-1}{2} n}G_{mn}\,.
\end{equation}
We distinguish three cases according to the arithmetic nature of $q$.

\subsection*{\ Case 1. $q \in \mathbb{Z}, q\geq 2$.} In this case, $\frac{q - 1}{2} \in \mathbb{Q}$. If $q $ is odd, then $\frac{q-1}{2} \in \mathbb{Z}$, and the expansion takes the standard Hilbert-type form: 
\begin{equation}
 g_\eps =   G_{00} + \eps G_{01} + \eps^2 G_{02} + \cdots\,.
\end{equation}
If $q$ is even, then $\frac{q -1 }{2} \notin \mathbb{Z}$, and the expansion includes fractional powers of $\eps$:
\begin{equation}
 \begin{aligned}
   g_\eps = &G_{00} + \eps G_{01} + \eps^2 G_{02} + \cdots\\
   & +\eps^{\frac{q -1}{2}} \{G_{10} + \eps G_{11} + \eps^2 G_{12} + \cdots\} \,. \\ 
 \end{aligned}
\end{equation}
In particular, the scaling $q = 2$ corresponds to the case considered in this paper, and the expansion structure coincides with the form above.

\subsection*{ Case 2. $q \in \mathbb{Q} / \mathbb{Z}, q = \tfrac{b}{a} >1, \gcd(a,b) = 1$.}  
We further divide this case into two subcases depending on the parity of $b -a$
\begin{itemize}
  \item $b -a$ is even. The exponent $\tfrac{q - 1}{2} = \tfrac{b - a }{2 a}$ is rational, and the expansion terminates at a finite level:
\begin{equation}
 \begin{aligned}
  g_\eps = & G_{00} + \eps G_{01} + \eps^2 G_{02} + \cdots\\
   &+ \eps^{\frac{q -1}{2}}\{ G_{10} + \eps G_{11} + \eps^2 G_{12} + \cdots\}\\ 
   &+ \cdots\\
   &+ \eps^{\frac{(a-1)(q -1)}{2}}\{ G_{(a-1)0} + \eps G_{(a-1)1} + \eps^2 G_{(a-1)2} + \cdots\}\,.
 \end{aligned}
\end{equation}

\item $b -a$ is odd. To preserve integrality of all exponents, the expansion must extend further, up to order $(2a - 1)$:
\begin{equation}
 \begin{aligned}
   g_\eps = &  G_{00} + \eps G_{01} + \eps^2 G_{02} + \cdots\\
   &+\eps^{\frac{q -1}{2}}\{ G_{10} + \eps G_{11} + \eps^2 G_{12} + \cdots\}\\ 
   &+ \cdots\\
   &+ \eps^{\frac{(2a-1)(q -1)}{2}}\{ G_{(2a-1)0} + \eps G_{(2a-1)1} + \eps^2 G_{(2a-1)2} + \cdots\}\,.
 \end{aligned}
\end{equation}

\end{itemize}
\subsection*{ Case 3: $q \in \mathbb{R} / \mathbb{Q}, q >1$.}  When $q$ is irrational, the exponent $\frac{q - 1}{2}$ is also irrational, and the powers of $\eps$ appearing in the expansion cannot be expressed as finite linear combinations of a fixed finite set of exponents with non-negative integer coefficients.
 As a result, the expansion has infinitely many rows of different fractional orders:
\begin{equation}
\begin{aligned}
g_\varepsilon =\ & G_{00} + \varepsilon G_{01} + \varepsilon^2 G_{02} + \cdots \\
&+ \varepsilon^{\frac{q - 1}{2}} \{ G_{10} + \varepsilon G_{11} + \varepsilon^2 G_{12} + \cdots \} \\
&+ \varepsilon^{ \frac{2(q - 1)}{2}} \{ G_{20} + \varepsilon G_{21} + \eps^2 G_{22}+ \cdots \}\\
& + \cdots \,.
\end{aligned}
\end{equation}
This expansion involves infinitely many layers, each associated with a distinct irrational power of $\eps$, reflecting the non-periodic structure of the scales in this setting. Although there are infinitely many rows in this case due to the irrational nature of the scaling, we will truncate the ansatz at a suitable order in both rows and columns. The resulting truncated ansatz thus contains only finitely many terms.

Once the multi-scale ansatz is constructed, one can proceed with the formal derivation of the fluid-dynamic equations. This is achieved by substituting the expansion into the kinetic equation and collecting terms of the same order in $\eps$. In this way, we recover the governing equations for the interior solution, the Prandtl boundary layer, and the Knudsen layer, following a procedure analogous to the classical Hilbert expansion described in Section \ref{sec_Hilbert}. By truncating the ansatz at a suitable order and applying a similar technical method as in Section \ref{Sec_remainder_uniform}, we can justify the incompressible Euler limit for general $q > 1$.

\bigskip

\bibliography{reference}

\end{document}